\documentclass[a4paper, 12pt]{amsart}

\usepackage{lmodern}
\usepackage[english]{babel}
\usepackage[latin1]{inputenc}
\usepackage{microtype}
\usepackage{amsfonts,amssymb,amsmath,amsthm,mathrsfs,mathtools}
\usepackage{color, graphicx}
\usepackage[dvipsnames]{xcolor}
\usepackage[all]{xy}
\usepackage[bookmarks=false]{hyperref}  
\usepackage{version}
\usepackage[inline]{enumitem}
\usepackage{cleveref}
\usepackage{a4wide}
\usepackage{upgreek}

\makeatletter																																													
\DeclareMathSizes{\@xpt}{\@xpt}{6}{5}																																	
\makeatother																																													
\makeatletter
\def\namedlabel#1#2{\begingroup
    #2%
    \def\@currentlabel{#2}%
    \phantomsection\label{#1}\endgroup
}
\makeatother

\allowdisplaybreaks

\newenvironment{invisible}[1][\unskip]{
	\noindent
	\color{red}
	[{\textbf{\color{blue}TBH}: \textit{#1}}
}{]}

\newenvironment{personal}{{\noindent \textbf{\color{green}Personal notes (To Be Hidden)}:\quad}\color{green}}{}

\theoremstyle{plain}

\newtheorem{theorem}{Theorem}[section]
\newtheorem{lemma}[theorem]{Lemma}
\newtheorem{proposition}[theorem]{Proposition}
\newtheorem{corollary}[theorem]{Corollary}
\newtheorem*{theorem*}{Theorem}

\theoremstyle{definition}

\newtheorem{example}[theorem]{Example}

\theoremstyle{remark}
\newtheorem{remark}[theorem]{Remark}

\newcommand{\Z}{\mathbb{Z}}

\newcommand{\C}{\mathbb{C}}
\newcommand{\K}{\Bbbk}

\newcommand{\quasihopfmod}[1]{{}_{\bullet}^{\phantom{\bullet}}{#1}_{\bullet}^{\bullet}} 
\newcommand{\bimod}[1]{{}_{\bullet}^{}{#1}_{\bullet}^{}} 
\newcommand{\lmod}[1]{{}_{\bullet}^{}{#1}} 
\newcommand{\rmod}[1]{{#1}_{\bullet}^{}} 

\newcommand{\cl}[1]{\overline{#1}} 
\newcommand{\coinv}[2]{{#1}{}^{\mathrm{co}{#2}}} 

\newcommand{\inv}[2]{{\overline{#1}{}^{#2}}}

\newcommand{\op}[1]{{#1}^{\mathrm{op}}} 
\newcommand{\mf}[1]{\mathfrak{#1}} 

\newcommand{\rdual}[1]{{#1}^*} 
\newcommand{\tensor}[1]{\otimes_{{#1}}} 
\newcommand{\what}[1]{\widehat{#1}} 
\newcommand{\cmptens}[1]{\,\what{\otimes}_{#1}\,} 
\newcommand{\tildetens}[1]{\,\widetilde{\otimes}_{#1}\,}

\newcommand{\Nat}{\mathsf{Nat}} 
\renewcommand{\ker}{\mathsf{ker}} 

\newcommand{\id}{\mathsf{Id}} 

\newcommand{\Hom}[6]{{_{#1}^{#2}\mathsf{Hom}_{#3}^{#4}}\left({#5},{#6}\right)} 
\newcommand{\Homk}{\mathsf{Hom}} 
\newcommand{\End}[2]{\mathsf{End}_{#1}(#2)} 

\newcommand{\quasihopf}[1]{{}_{#1}^{\phantom{#1}}\M{}_{#1}^{#1}} 

\newcommand{\rhopf}[1]{\M_{#1}^{#1}}

\newcommand{\Bim}[2]{{}_{#1}\M{}_{#2}} 
\newcommand{\Rmod}[1]{\M_{#1}} 
\newcommand{\Rcomod}[1]{\M^{#1}} 
\newcommand{\Lmod}[1]{{}_{#1}\M} 
\newcommand{\Bimod}[1]{\Bim{#1}{#1}} 

\newcommand{\cC}{{\mathcal C}}
\newcommand{\cD}{{\mathcal D}}

\newcommand{\cF}{{\mathcal F}}
\newcommand{\cG}{{\mathcal G}}

\newcommand{\cL}{{\mathcal L}}
\newcommand{\cM}{{\mathcal M}}

\newcommand{\cR}{{\mathcal R}}

\newcommand{\cT}{{\mathcal T}}

\newcommand{\bH}{{\mathbb H}}

\newcommand{\M}{\mathfrak{M}} 
\newcommand{\I}{\mathbb{I}} 
\newcommand{\calpha}{\mf{a}}
\newcommand{\clambda}{\mf{l}}
\newcommand{\crho}{\mf{r}}
\newcommand{\cpsi}{\uppsi} 
\newcommand{\cvarphi}{\upvarphi} 

\newcommand{\ie}{i.e.~}
\newcommand{\eg}{e.g.~}
\renewcommand{\t}{~}
\newcommand{\can}{\mathfrak{can}} 

\excludeversion{invisible} 
\excludeversion{personal}

\title{Antipodes, preantipodes and Frobenius functors}
\author{Paolo Saracco}
\address{D\'epartement de Math\'ematique, Universit\'e Libre de Bruxelles, Boulevard du Triomphe, B-1050 Brussels, Belgium.}
\thanks{This paper was written while P. Saracco was member of the ``National Group for Algebraic and Geometric Structures and their Applications'' (GNSAGA-INdAM). He acknowledges FNRS support through a postdoctoral fellowship within the framework of the MIS Grant ``ANTIPODE'' (MIS F.4502.18, application number 31223212). He is also grateful to Alessandro Ardizzoni and Joost Vercruysse for their willingness in discussing the content of the present paper and to the anonymous referee for the careful reading of this work and the numerous valuable suggestions that contributed to improve it. In particular, for pointing out the possible connections with \cite{BulacuCaenepeelTorrecillas2}, that led to considerably improve \S\ref{ssec:unimodularity}.}
\urladdr{sites.google.com/view/paolo-saracco}
\email{paolo.saracco@ulb.ac.be}

\usepackage{fancyhdr}

\begin{document}

\begin{abstract}
We prove that a quasi-bialgebra admits a preantipode if and only if the associated free quasi-Hopf bimodule functor is Frobenius, if and only if the related (opmonoidal) monad is a Hopf monad. The same results hold in particular for a bialgebra, tightening the connection between Hopf and Frobenius properties.
\end{abstract}

\maketitle

\fancyhf{}
\renewcommand{\headrulewidth}{0pt}
\thispagestyle{fancy}
\lfoot{\smallskip\footnotesize Electronic version of an article published as {\em Journal of Algebra and Its Applications}, Vol. \textbf{20}, No. 7, 2021, 2150214 (DOI: \href{https://doi.org/10.1142/S0219498821501243}{10.1142/S0219498821501243}) \copyright\, World Scientific Publishing Company (\href{https://www.worldscientific.com/worldscinet/jaa}{www.worldscientific.com/worldscinet/jaa}).}

\tableofcontents


\section*{Introduction}

It has been known for a long time that Hopf algebras (with some additional finiteness condition) are strictly related with Frobenius algebras. In fact, Larson and Sweedler proved in \cite{LarsonSweedler} that any finite-dimensional Hopf algebra over a PID is automatically Frobenius and Pareigis extended this result in \cite{Pareigis} by proving that a bialgebra $B$ over a commutative ring $\K$ is a finitely generated and projective Hopf algebra with $\int\rdual{B}\cong\K$ 
if and only if it is Frobenius as an algebra with Frobenius homomorphism $\psi\in\int\rdual{B}$. 
Afterwards, great attention has been devoted to those bialgebras that are also Frobenius and whose Frobenius homomorphism is an integral (see \eg \cite{KadisonStolin,Pareigis-cohomology}) and to the interactions between Frobenius and Hopf algebra theory in general (see \cite{BeidarFongStolin1,BeidarFongStolin2,KadisonFrobenius,Lorenz}). In particular, there exist a number of results that extend Larson-Sweedler's and Pareigis' theorems to more general classes of Hopf-like structures (\cite{IovanovKadison,Kadison2,vanDaele1,vanDaele2}).

Following their increasing importance, many extensions of Hopf and Frobenius algebras have arisen. Let us mention (co)quasi-Hopf algebras, Hopf algebroids, Hopf monads, Frobenius monads, Frobenius monoids and Frobenius functors. At the same time new results appeared (\cite{Balan,Heunen,Kadison2}), showing that there is a deeper connection between the Hopf and the Frobenius properties that deserves to be uncovered. In \cite{Saracco-Frobenius} we realised that Frobenius functors may play an important role in this. In fact, we proved that a bialgebra $B$ is a one-sided Hopf algebra (in the sense of \cite{GreenNicholsTaft}) with anti-(co)multiplicative one-sided antipode if and only if the free Hopf module functor $-\otimes B:\M\to\rhopf{B}$ (key ingredient of the renowned Structure Theorem of Hopf modules) is Frobenius. In the finitely generated and projective case, this allowed us to prove a categorical extension of Pareigis' theorem (see \cite[Theorem 3.14]{Saracco-Frobenius}). In the present paper we continue our investigation in this direction by analysing another important adjoint triple strictly connected with bialgebras and their representations, namely the one associated with the free two-sided Hopf module functor $-\otimes B:\Lmod{B}\to\quasihopf{B}$. The study of the Frobenius property for this latter functor has proved to be more significant than the previous one for two main reasons. The first one is that being Frobenius for $-\otimes B:\Lmod{B}\to\quasihopf{B}$ has proven to be in fact equivalent to $B$ being a Hopf algebra. Even more generally, the following is our first main result.

\begin{theorem*}[{Theorem \ref{thm:mainThm}}]
The following are equivalent for a quasi-bialgebra $A$.
\begin{enumerate}[label=(\arabic*), ref=\emph{(\arabic*)}, leftmargin=1cm, labelsep=0.3cm]
\item $A$ admits a preantipode;
\item $-\otimes A:\left(\Lmod{A},\otimes,\K\right)\to \left(\quasihopf{A},\tensor{A},A\right)$ is a monoidal equivalence of categories;
\item $-\otimes A:\Lmod{A}\to\quasihopf{A}$ is Frobenius;
\item $\sigma_M:\Hom{A}{}{A}{A}{A\otimes A}{M}\to \cl{M}, f\mapsto \cl{f(1\otimes 1)}$ is an isomorphism for every $M\in\quasihopf{A}$, where $\cl{M}\cong M\tensor{A}{_\varepsilon \K}$.
\end{enumerate}
\end{theorem*}

The second one is that the monad $T\coloneqq \cl{(-)}\otimes B$ on $\quasihopf{B}$ induced by the adjunction $\cl{(-)}\dashv -\otimes B$, being the functor $-\otimes B:\Lmod{B}\to\quasihopf{B}$ a strong monoidal functor between monoidal categories, turns out to be an opmonoidal monad in the sense of \cite{McCrudden}. As such, it allows us to relate our approach by means of Frobenius functors with the theory of Hopf monads developed by Brugui\`{e}res, Lack and Virelizier in \cite{BruguieresLackVirelizier,BruguieresVirelizier}. In concrete, the following is our second main result.

\begin{theorem*}[{Theorem \ref{thm:main2}}]
The following are equivalent for a quasi-bialgebra $A$.
\begin{enumerate}[label=(\alph*), ref=\emph{(\alph*)}, leftmargin=1cm, labelsep=0.3cm]
\item $A$ admits a preantipode;
\item the natural transformation $\cpsi_{M,N}:\overline{M\tensor{A}N} \to \overline{M}\otimes \overline{N}, \, \overline{m\tensor{A}n}\mapsto \overline{m_0}\otimes \overline{m_1n}$, for $M,N\in\quasihopf{A}$, is a natural isomorphism;
\item the component $\cpsi_{A\widehat{\otimes}A,A\otimes A}$ of $\cpsi$, where $A\cmptens{}A=A\otimes A$ with a suitable quasi-Hopf bimodule structure, is invertible;
\item $\left(\cl{(-)},-\otimes A\right)$ is a lax-lax adjunction;
\item $\left(\cl{(-)},-\otimes A\right)$ is a Hopf adjunction;
\item $T=\cl{(-)}\otimes A$ is a Hopf monad on $\quasihopf{A}$.
\end{enumerate}
\end{theorem*}

Let us highlight that a consequence of the previous theorem is that $\cl{(-)}\otimes A$ is an example of an opmonoidal monad which is Hopf if and only if it is Frobenius (see Remark \ref{rem:HopfFrob}).

Even if we are mainly interested in the Hopf algebra case, there are valid motivations for us to work in the more general context of quasi-bialgebras and preantipodes, despite the slight additional effort. Quasi-bialgebras, and in particular quasi-Hopf algebras (\ie quasi-bialgebras with a quasi-antipode), naturally arise from the study of quantum groups and hence they are of general interest for the scientific community as well. Preantipodes, in turn, are proving to be in many situations a much better behaved analogue of antipodes for (co)quasi-bialgebras than quasi-antipodes (see \cite{ArdiPavaBosonization, ArdiPava, Saracco-Tannaka, Saracco}). The results of the present paper are an additional confirmation of this fact and hence, either in case their existence turns out to be equivalent to the existence of quasi-antipodes or in case they prove to be a more general notion, preantipodes also are expected to be of interest for the community and they deserve to be investigated further.

\medskip

The paper is organized as follows. In Section \ref{sec:preliminaries} we recall some general facts about adjoint triples, Frobenius functors, monoidal categories and quasi-bialgebras that will be needed in the sequel. Section \ref{sec:HopfBimod} is devoted to the study of when the free quasi-Hopf bimodule functor $-\otimes A:\Lmod{A}\to\quasihopf{A}$ for a quasi-bialgebra $A$ is Frobenius. The main results of this section are Theorem \ref{thm:mainThm}, characterizing quasi-bialgebras $A$ with preantipode as those for which $-\otimes A$ is Frobenius, and its consequence, Theorem \ref{thm:mainThmHopf}, rephrasing this fact for bialgebras. A detailed example, in a context where computations are easily handled, follows and then the section is closed by a collection of results connecting the theory developed herein with some of those in \cite{Saracco-Frobenius} (\S\ref{ssec:previous}) and in \cite{BulacuCaenepeelTorrecillas1,BulacuCaenepeelTorrecillas2} (\S\ref{ssec:unimodularity}). Finally, in Section \ref{sec:HopfMonads} we investigate the connection between the Frobenius property for $-\otimes A:\Lmod{A}\to\quasihopf{A}$ and the fact of being Hopf for the induced monad $T=\cl{(-)}\otimes A$. The main results here are Theorem \ref{thm:main2} and its consequence, Corollary \ref{cor:main2Hopf}.

\subsection*{Notations and conventions}

Throughout the paper, $\K$ denotes a base commutative ring (from time to time a field) and $A$ a quasi-bialgebra over $\K$ with unit $u:\K\to A$ (the unit element of $A$ is denoted by $1_A$ or simply $1$), multiplication $m:A\otimes A\to A$ (often denoted by simple juxtaposition), counit $\varepsilon:A\to \K$ and comultiplication $\Delta:A\to A\otimes A$. We write $A^+\coloneqq \ker(\varepsilon)$ for the augmentation ideal of $A$. The category of all (central) $\K$-modules is denoted by $\M$ and by $\Rmod{A}$ (resp. $\Lmod{A}$) and $\Bimod{A}$ we mean the categories of right (resp. left) modules and bimodules over $A$. The unadorned tensor product $\otimes$ is the tensor product over $\K$ and the unadorned $\Homk$ stands for the space of $\K$-linear maps. The coaction of a comodule is usually denoted by $\delta$ and the action of a module by $\mu$, $\cdot$ or simply juxtaposition. In order to handle comultiplications and coactions, we resort to the following variation of Sweedler's sigma notation:
\[
\Delta(a) = a_1\otimes a_2 \qquad \text{and} \qquad \delta(n) = n_0 \otimes n_1 \qquad \text{(summation understood)}
\] 
for all $a\in A$, $n\in N$ comodule. We often shorten iterated tensor products $A\otimes A\otimes \cdots \otimes A$ of $n$ copies of $A$ by $A^{\otimes n}$. When specializing to the coassociative framework, we use $B$ to denote a bialgebra over $\K$.


\section{Preliminaries}\label{sec:preliminaries}

\subsection{Adjoint triples}\label{sec:adjtriples}

\begin{personal}
{\color{blue}
This section contains observations and properties of adjoint triples that already appeared in literature or that are folklore. It doesn't seem that any result here is new.

\textbf{References}: 
\begin{itemize}
\item Street, \emph{Frobenius Monads and Pseudomonoids};
\item Lauda, \emph{Frobenius Algebras and Ambidextrous Adjunctions};
\item http://comonad.com/reader/2016/adjoint-triples/;
\item https://ncatlab.org/nlab/show/adjoint+triple; 
\item ...
\end{itemize}
}
\end{personal}

Let us recall quickly some facts about adjoint triples and Frobenius functors that we are going to use in the paper. For further details on these objects in connection with our setting, see for example \cite[\S1]{Saracco-Frobenius}. Given categories $\cC$ and $\cD $, we say that functors $\cL,\cR:\cC \rightarrow \cD $, $\cF:\cD \rightarrow \cC $ form an \emph{adjoint triple} if $\cL$ is left adjoint to $\cF$ which is left adjoint to $\cR$, in symbols $\cL\dashv \cF\dashv \cR$. They form an \emph{ambidextrous adjunction} if there is a natural isomorphism $\cL\cong \cR$. As a matter of notation, we set $\eta :\id \rightarrow \cF\cL,\epsilon :\cL\cF \rightarrow \id $ for the unit and counit of the left-most adjunction and $\gamma :\id \rightarrow \cR\cF,\theta :\cF\cR \rightarrow \id $ for the right-most one. If in addition $\cF$ is fully faithful, that is, if $\epsilon $ and $\gamma $ are natural isomorphisms (see \cite[Theorem IV.3.1]{MacLane}), then we have a distinguished natural transformation
\begin{equation*}
\sigma \coloneqq \left( \xymatrix{\cR \ar[r]^-{\left( \epsilon \cR\right) ^{-1}} & \cL\cF\cR \ar[r]^-{\cL\theta} & \cL}\right) .
\end{equation*}
A \emph{Frobenius pair} for the categories $\cC $ and $\cD $ is a couple of functors $\cF:\cC \rightarrow \cD $ and $\cG:\cD \rightarrow \cC $ such that $\cG $\ is left and right adjoint to $\cF$. A functor $\cF$ is said to be \emph{Frobenius} if there exists a functor $\cG $ which is at the same time left and right adjoint to $\cF$. The subsequent lemma collects some rephrasing of the Frobenius property for future reference.

\begin{lemma}\label{lemma:frobenius}
The following are equivalent for a functor $\cF:\cC \rightarrow \cD $
\begin{enumerate}[label=(\arabic*), ref=\emph{(\arabic*)}, leftmargin=1cm, labelsep=0.3cm]
\item\label{item:frobenius1} $\cF$ is Frobenius;
\item\label{item:frobenius2} there exists $\cR:\cD \rightarrow \cC $ such that $\left( \cF,\cR\right) $ is a Frobenius pair;
\item\label{item:frobenius3} there exists $\cL:\cD \rightarrow \cC $ such that $\left( \cL,\cF\right) $ is a Frobenius pair;
\item\label{item:frobenius4} there exist $\cL,\cR:\cD \rightarrow \cC $ such that $\cL\dashv \cF\dashv \cR$ is an ambidextrous adjunction.
\end{enumerate}
Moreover, if $\cF$ is fully faithful, anyone of the above conditions is equivalent to
\begin{enumerate}[resume, leftmargin=1cm, labelsep=0.3cm]
\item\label{item:frobenius5} there exist $\cL,\cR:\cD \rightarrow \cC $ such that  $\cL\dashv \cF\dashv \cR$ is an adjoint triple and $\sigma:\cR\to\cL$ is a natural isomorphism.
\end{enumerate}
\end{lemma}

Since we are interested in adjoint triples whose middle functor is fully faithful, Lemma \ref{lemma:frobenius} allows us to study the Frobenius property by simply looking at the invertibility of the canonical map $\sigma$. Observe that
\begin{equation}\label{eq:sigmaF}
\sigma\cF = \epsilon^{-1} \circ \gamma^{-1} \qquad \text{and} \qquad \cF\sigma = \eta\circ\theta,
\end{equation}
whence, in particular, $\sigma\cF$ is always a natural isomorphism.

\begin{invisible}
As a matter of terminology: if a morphism is a split epimorphism, then we say that it admits a \emph{section} (\ie a right inverse). If it is a split monomorphism, we say that it admits a \emph{retraction} (\ie a left inverse).

{\color{blue} Add analogue result with comonadic and right adjoint, in order to have that $\cF$ is both.}
\end{invisible}


\subsection{Monoidal categories}

Recall that a \emph{monoidal category} $\left( \cM ,\otimes ,\I ,\calpha ,\clambda ,\crho \right) $ is a category $\cM $ endowed with a functor $\otimes :\cM \times \cM \to \cM $ (the \emph{tensor product}), with a distinguished object $\I $ (the \emph{unit}) and with three natural isomorphisms 
\begin{gather*}
	\calpha :\otimes \circ (\otimes \times \id_{\cM})\to \otimes \circ (\id_{\cM }\times \otimes)\quad \text{(\emph{associativity constraint})} \\
	\clambda :\otimes \circ  (\I \times \id _{\cM })\to \id _{\cM }\quad \text{(\emph{left unit constraint})} \\
	\crho :\otimes \circ  (\id _{\cM }\times \I )\to \id _{\cM }\quad \text{(\emph{right unit constraint})} 
\end{gather*}
that satisfy the \emph{Pentagon} and the \emph{Triangle Axioms}, that is,
\begin{gather*}
\calpha_{X,Y,Z\otimes W} \circ \calpha_{X\otimes Y,Z,W} = \left(X\otimes \calpha_{Y,Z,W}\right) \circ \calpha_{X,Y\otimes Z,W} \circ \left(\calpha_{X,Y,Z}\otimes W\right), \\
\left(X\otimes \clambda_Y\right) \circ \calpha_{X,\I ,Y} = \crho_X\otimes Y,
\end{gather*}
for all $X,Y,Z,W$ objects in $\cM$.

If the endofunctor $X\otimes-:Y\mapsto X\otimes Y$ (resp. $-\otimes X:Y\mapsto Y\otimes X$) has a right adjoint for every $X$ in $\cM$, then $\cM$ is called a \emph{left-closed} (resp. \emph{right-closed}) monoidal category.

Given two monoidal categories $\left( \cM ,\otimes ,\I ,\calpha,\clambda ,\crho \right) $ and $\left( \cM ^{\prime },\otimes^{\prime} ,\I^{\prime},\calpha ^{\prime },\clambda ^{\prime },\crho ^{\prime }\right) $, a \emph{quasi-monoidal functor} $(\cF,\cvarphi_0,\cvarphi)$ between $\cM$ and $\cM'$ is a functor $\cF:\cM \to \cM ^{\prime }$ together with an isomorphism $\cvarphi _{0}:\I^{\prime}\to \cF\left( \I \right) $ and a family of isomorphisms $\cvarphi _{X,Y}:\cF \left( X\right) \otimes^{\prime} \cF\left( Y\right) \to \cF\left( X\otimes Y\right) $ for $X,Y$ objects in $\cM $, which are natural in both entrances. A quasi-monoidal functor $\cF$ is said to be \emph{neutral} if 
\begin{gather}\label{eq:neutral}
\begin{gathered}
\cF\left( \clambda_{ {X}} \right) \circ \cvarphi_{ {\I,X}} \circ \left(\cvarphi _{0}\otimes^{\prime}\cF\left( X\right)\right) = \clambda ^{\prime }_{ {\cF(X)}}, \\
\cF\left(\crho_{ {X}} \right) \circ \cvarphi_{ {X,\I}} \circ \left(\cF\left(X\right) \otimes^{\prime}\cvarphi _{0} \right) = \crho ^{\prime }_{ {\cF(X)}},
\end{gathered}
\end{gather}
and it is said to be \emph{strong monoidal} if, in addition,
\begin{equation}\label{eq:monoidal}
\cvarphi_{X,Y\otimes Z} \circ \left(\cF(X)\otimes^{\prime} \cvarphi_{Y,Z}\right)  \circ \calpha^{\prime }_{\cF(X),\cF(Y),\cF(Z)} = \cF \left( \calpha_{X,Y,Z}\right) \circ \cvarphi_{X\otimes Y,Z} \circ \left(\cvarphi_{X,Y} \otimes^{\prime} \cF(Z)\right)
\end{equation}
for all $X,Y,Z$ in $\cM$. Furthermore, it is said to be \emph{strict} if $\cvarphi_0$ and $\cvarphi$ are the identities. A strong monoidal functor $(\cF,\cvarphi_0,\cvarphi)$ such that $\cF$ is an equivalence of categories is called a \emph{monoidal equivalence}.

If $\cF$ comes together with a morphism $\cvarphi_0:\I'\to\cF(\I)$ and a natural transformation $\cvarphi _{X,Y}:\cF \left( X\right) \otimes^{\prime} \cF\left( Y\right) \to \cF\left( X\otimes Y\right) $ that are not necessarily invertible but that satisfy \eqref{eq:neutral} and \eqref{eq:monoidal} then it is called a \emph{lax monoidal functor} in \cite[Definition 3.1]{Aguiar} (also termed \emph{monoidal functor} in \cite{BruguieresLackVirelizier, BruguieresVirelizier}). If instead $\cF$ comes together with a morphism $\cpsi_0:\cF(\I)\to\I'$ and a natural transformation $\cpsi _{X,Y}: \cF\left( X\otimes Y\right) \to \cF \left( X\right) \otimes^{\prime} \cF\left( Y\right)$ (not necessarily invertible) satisfying the analogues of \eqref{eq:neutral} and \eqref{eq:monoidal} then it is called a \emph{colax monoidal functor} in \cite[Definition 3.2]{Aguiar} (also termed \emph{opmonoidal functor} in \cite{McCrudden} and \emph{comonoidal functor} in \cite{BruguieresLackVirelizier, BruguieresVirelizier}).

In \cite[Definition 3.8]{Aguiar}, a natural transformation $\upgamma$ between monoidal functors $(\cF,\cvarphi_0,\cvarphi)$ and $(\cG,\upphi_0,\upphi)$ from a monoidal category $(\cM,\otimes,\I,\calpha,\clambda,\crho)$ to $(\cM',\otimes',\I',\calpha',\clambda',\crho')$ is said to be a \emph{morphism of lax monoidal functors} (also called \emph{monoidal natural transformation} in \cite{BruguieresLackVirelizier, BruguieresVirelizier}) if 
\begin{equation}
\left(\upgamma_X\otimes'\upgamma_Y\right) \circ \cvarphi_{X,Y} = \upphi_{X,Y} \circ \upgamma_{X\otimes Y} \quad \text{and} \quad \upgamma_\I\circ \cvarphi_0 = \upphi_0.
\end{equation} 
Similarly, one defines \emph{morphisms of colax monoidal functors} (also called \emph{transformations of opmonoidal functors} in \cite{McCrudden} and \emph{comonoidal natural transformations} in \cite{BruguieresLackVirelizier, BruguieresVirelizier}). An adjoint pair of monoidal functors is called a \emph{lax-lax adjunction} in \cite[Definition 3.87]{Aguiar} (also termed a \emph{monoidal adjunction} in \cite{BruguieresVirelizier}) if the unit and the counit are morphisms of lax monoidal functors. Analogously, see \cite[Definition 3.88]{Aguiar}, one defines \emph{colax-colax adjunctions} (also termed \emph{comonoidal adjunctions} in \cite{BruguieresLackVirelizier}) as those for which the unit and the counit are morphisms of colax monoidal functors.

By adhering to the suggestions of the referee and because results from \cite{Aguiar} are widely used, we will adopt the terminology of \cite{Aguiar} all over the paper. The unique exception will be the use of the term \emph{opmonoidal monad} in \S\ref{sec:HopfMonads}, because the latter is, as far as the author knows, the most widely used in the study of Hopf monads and related constructions (see for example \cite{Bohm}, in particular \cite[Chapter 3]{Bohm}, and the references therein).

Henceforth, we will often omit the constraints when referring to a monoidal category.


\subsection{Quasi-bialgebras and quasi-Hopf bimodules}

Let $\K$ be a commutative ring. Recall from \cite[\S1, Definition]{Drinfel'd} that a \emph{quasi-bialgebra} over $\K$ is an algebra $A$ endowed with two algebra maps $\Delta:A\to A\otimes A$, $\varepsilon:A\to \K$ and a distinguished invertible element $\Phi\in A\otimes A\otimes A$ such that 
\begin{gather}
(A\otimes A\otimes \Delta)(\Phi)(\Delta\otimes A\otimes A)(\Phi)=(1\otimes \Phi)(A\otimes \Delta\otimes A)(\Phi)(\Phi\otimes 1), \label{eq:PhiCocycle} \\
(A\otimes \varepsilon\otimes A)(\Phi)= 1 \otimes 1, \label{eq:PhiCounital}
\end{gather}
$\Delta$ is counital with counit $\varepsilon$ and it is coassociative up to conjugation by $\Phi$, that is,
\begin{equation}\label{eq:quasi-coassociativity} 
\Phi(\Delta\otimes A)(\Delta(a))=(A\otimes\Delta)(\Delta(a))\Phi. 
\end{equation}
As a matter of notation, we will write $\Phi=\Phi^1\otimes \Phi^2\otimes \Phi^3=\Psi^1\otimes \Psi^2\otimes \Psi^3=\cdots$ and $\Phi^{-1} = \varphi^1\otimes \varphi^2\otimes \varphi^3 = \phi^1\otimes \phi^2\otimes \phi^3=\cdots$ (summation understood). A \emph{preantipode} (see \cite[Definition 1]{Saracco}) for a quasi-bialgebra is a $\K$-linear endomorphism $S:A\to A$ such that 
\begin{equation}\label{eq:defpreantipode}
S(a_1b)a_2 = \varepsilon(a)S(b) = a_1S(ba_2) \quad \text{and} \quad \Phi^1S(\Phi^2)\Phi^3 = 1
\end{equation}
for all $a,b\in A$. An \emph{antipode} (from time to time also called \emph{quasi-antipode}, to distinguish it from the ordinary antipode of Hopf algebras) is a triple $(s,\alpha,\beta)$, where $s:A\to A$ is an algebra anti-homomorphism and $\alpha,\beta\in A$ are elements, such that for all $a\in A$ we have
\[
\xymatrix @R=0pt @C=10pt{
s(a_1)\alpha a_2 = \varepsilon(a)\alpha, & a_1\beta s(a_2) = \varepsilon(a)\beta, \\
\Phi^1\beta s(\Phi^2)\alpha \Phi^3 = 1, & s(\varphi^1)\alpha \varphi^2 \beta s(\varphi^3) = 1.
}
\]
A quasi-bialgebra admitting an antipode is called a \emph{quasi-Hopf algebra} (see \cite[Definition, page 1424]{Drinfel'd}). By comparing \cite[Theorem 4]{Saracco} and \cite[Theorem 3.1]{Schauenburg-TwoChar}, we have that the following holds.

\begin{proposition}\label{prop:PreQuasiHopf}
Over a field $\K$, a finite-dimensional quasi-bialgebra $A$ admits a preantipode if and only if it is a quasi-Hopf algebra.
\end{proposition}

The subsequent lemma gives an equivalent characterization of quasi-bialgebras in terms of their categories of modules (see \cite[Theorem 1]{ArdiBulacuMen}).

\begin{lemma}
A $\K$-algebra $A$ is a quasi-bialgebra if and only if its category of left (equivalently, right) modules is a monoidal category with neutral quasi-monoidal underlying functor to $\K$-modules. The associativity constraint is given by $\calpha_{M,N,P}(m\otimes n\otimes p)=\Phi\cdot (m\otimes n\otimes p)$ for every $M,N,P\in\Lmod{A}$ and for all $m\in M, n \in N, p\in P$.
\end{lemma}

As a matter of notation, if the context requires to stress explicitly the (co)module structures on a particular $\K$-module $V$, we will adopt the following conventions. With a full bullet, such as $V_\bullet$ or $V^\bullet$, we will denote a given right action or coaction respectively  (analogously for the left ones). For example, a left comodule $V$ over a bialgebra $B$ in the category of $B$-bimodules will be also denoted by ${^\bullet_\bullet V^{\phantom{\bullet}}_\bullet}$. With $V^u\coloneqq V\otimes \K^u$ and $V_\varepsilon\coloneqq V\otimes \K_\varepsilon$ we will denote the trivial right comodule and right module structures on $V$, respectively (analogously for the left ones). 

\begin{remark}\label{rem:quasihopf}
Also the category $\Bim{A}{A}$ of $A$-bimodules over a quasi-bialgebra $A$ is monoidal with neutral quasi-monoidal underlying functor to $\K$-modules. In particular, the tensor product of two $A$-bimodules $M,N$ is (up to isomorphism) their tensor product over $\K$ with bimodule structure given by the \emph{diagonal actions}
\begin{equation}\label{eq:diagactions}
a\cdot (m\otimes n)\cdot b = a_1\cdot m\cdot b_1 \otimes a_2\cdot n\cdot b_2
\end{equation}
for all $a,b\in A$, $m\in M$, $n\in N$. The unit is $\K$ with two-sided action given by restriction of scalars along $\varepsilon$. The associativity constraint is given by conjugation by $\Phi$: for every $M,N,P\in\Bim{A}{A}$ and for all $m\in M,n\in N,p\in P$,
\begin{equation}\label{eq:asociativity}
\calpha_{M,N,P}(m\otimes n\otimes p) = \Phi\cdot(m\otimes n\otimes p)\cdot \Phi^{-1}.
\end{equation}
One may check that $(\varepsilon \otimes A\otimes A)(\Phi)=1\otimes 1 = (A\otimes A \otimes \varepsilon)(\Phi)$ and the same for $\Phi^{-1}$:
\begin{equation}\label{eq:epsivarphi}
(\varepsilon \otimes A\otimes A)\left(\Phi^{-1}\right)=1\otimes 1 = (A\otimes \varepsilon\otimes A )\left(\Phi^{-1}\right)=1\otimes 1 = (A\otimes A \otimes \varepsilon)\left(\Phi^{-1}\right).
\end{equation}
As a consequence, if for example ${_\bullet M}$ is a left $A$-module and ${_\bullet N_\bullet}, {_\bullet P_\bullet}$ are $A$-bimodules, then we may look at ${_\bullet M_\varepsilon}\in\Bim{A}{A}$ and $\calpha_{M,N,P}(m\otimes n\otimes p) = \Phi\cdot (m\otimes n\otimes p)$ for all $m\in M,n\in N,p\in P$. Therefore, we will use the notation $\calpha$ for the associativity constraint in the category of left, right and $A$-bimodules indifferently. In the same way, the tensor product of a left $A$-module $\lmod{M}$ and an $A$-bimodule $\bimod{N}$ is a bimodule with two-sided action given by \eqref{eq:diagactions}, \ie
\[
a\cdot (m\otimes n)\cdot b = a_1\cdot m\varepsilon(b_1) \otimes a_2\cdot n\cdot b_2 = a_1\cdot m \otimes a_2\cdot n\cdot b
\]
for all $a,b\in A$, $m\in M$, $n\in N$. We will denote $M\otimes N$ with the latter structures by $\lmod{M}\otimes\bimod{N}$. Furthermore, it can be checked that $A$, as a bimodule over itself and together with $\Delta$ and $\varepsilon$, is a comonoid in $\Bim{A}{A}$, so that we may consider the category $\quasihopf{A}=\left(\Bim{A}{A}\right)^A$ of the so-called \emph{quasi-Hopf bimodules}.
It is important to highlight that:
\begin{enumerate}[label=(\alph*), ref={(\alph*)}, leftmargin=1cm, labelsep=0.3cm]
\item The coassociativity of the coaction $\delta:M\to M\otimes A$ of a quasi-Hopf bimodule $M\in\quasihopf{A}$ is expressed by $\calpha_{M,A,A}\circ (\delta\otimes A)\circ \delta = (M \otimes \Delta)\circ \delta$, i.e., for all $m\in M$,
\begin{equation}\label{eq:coasscomod}
\Phi^1\cdot m_{0_0} \cdot \varphi^1 \otimes \Phi^2 m_{0_1}\varphi^2 \otimes \Phi^3 m_1 \varphi^3 = m_0 \otimes m_{1_1} \otimes m_{1_2}.
\end{equation}
\item If $N\in \Bimod{A}$ then ${_\bullet^{\phantom{\bullet}}N_\bullet^{\phantom{\bullet}}}\otimes \quasihopfmod{A} \in \quasihopf{A}$ with diagonal actions and $\delta \coloneqq \calpha_{N,A,A}^{-1}\circ (N\otimes \Delta)$, i.e., for all $n\in N$, $a,b,c\in A$,
\begin{equation}\label{eq:freecomod}
\begin{gathered}
a\cdot (n\otimes b) \cdot c = a_1\cdot n\cdot c_1 \otimes a_2bc_2 \qquad \text{and} \\
\delta(n\otimes a) = \varphi^1\cdot n\cdot \Phi^1 \otimes \varphi^2a_1\Phi^2 \otimes \varphi^3a_2\Phi^3.
\end{gathered}
\end{equation}
\end{enumerate}
\end{remark}

It is straightforward to check that the category of left modules over a quasi-bialgebra $A$ is not only monoidal, but in fact a right (and left) closed monoidal category with internal hom-functor $\Hom{A}{}{}{}{A\otimes N}{-}$ for all $N\in \Lmod{A}$ (for a proof, see \cite[Lemma 2.1.2]{PhD}). 

\begin{lemma}\label{lemma:modclosed}
Let $A$ be a quasi-bialgebra. Then the category ${_A\M}$ of left $A$-modules is left and right-closed. Namely, we have bijections
\begin{gather}
\xymatrix{
\Hom{A}{}{}{}{M\otimes N}{P} \ar@<+0.5ex>[r]^-{\varpi} & \Hom{A}{}{}{}{M}{\Hom{A}{}{}{}{A\otimes N}{P}}  \ar@<+0.5ex>[l]^-{\varkappa},
} \label{eq:lmodclosed}\\
\xymatrix{
\Hom{A}{}{}{}{N\otimes M}{P} \ar@<+0.5ex>[r]^-{\varpi'} & \Hom{A}{}{}{}{M}{\Hom{A}{}{}{}{N\otimes A}{P}}  \ar@<+0.5ex>[l]^-{\varkappa'},
}\notag
\end{gather}
natural in $M$ and $P$, given explicitly by
\begin{gather*}
\varpi(f)(m):a\otimes n \mapsto f(a\cdot m\otimes n), \qquad \varkappa(g):m\otimes n \mapsto g(m)(1\otimes n), \\
\varpi'(f)(m):n\otimes a \mapsto f(n\otimes a\cdot m), \qquad \varkappa'(g):n\otimes m \mapsto g(m)(n\otimes 1),
\end{gather*}
where the left $A$-module structures on $\Hom{A}{}{}{}{N\otimes A}{P}$ and $\Hom{A}{}{}{}{A\otimes N}{P}$ are induced by the right $A$-module structure on $A$ itself:
\begin{equation}\label{eq:actionhom}
(b\cdot f)(n\otimes a) = f(n\otimes ab) \qquad \text{and} \qquad (b\cdot g)(a\otimes n) = g(ab\otimes n)
\end{equation}
for all $a,b\in A$, $n\in N$, $f\in \Hom{A}{}{}{}{N\otimes A}{P}$ and $g\in \Hom{A}{}{}{}{A\otimes N}{P}$.
\end{lemma}

\begin{invisible}[Proof.]
Let $M,N,P$ be left $A$-modules. To make the exposition clearer, we will denote them via ${_\bullet M}$, ${_\bullet N}$ and ${_\bullet P}$ to underline the given actions. We are claiming that there is a bijection
\begin{equation*}
\xymatrix{
\Hom{A}{}{}{}{{_\bullet M}\otimes {_\bullet N}}{{_\bullet P}} \ar@<+0.5ex>[r]^-{\varpi} & \Hom{A}{}{}{}{{_\bullet M}}{\Hom{A}{}{}{}{{{_\bullet A}}\otimes {_\bullet N}}{{_\bullet P}}}   \ar@<+0.5ex>[l]^-{\varkappa}
}
\end{equation*}
natural in $M$ and $P$. Consider a generic $f\in \Hom{A}{}{}{}{{_\bullet M}\otimes {_\bullet N}}{{_\bullet P}}$. For all $m\in M$, $n\in N$ and $a,b,c\in A$ we have that
\begin{align*}
\big(\varpi(f)(c \cdot m)\big)\big(b\cdot(a\otimes n)\big) & = \sum f\Big(\big(b_1a\cdot (c\cdot m)\big)\otimes (b_2\cdot n)\Big) = f\Big(b\cdot \big((ac\cdot m)\otimes n\big)\Big) \\
 & =b \cdot \big(\varpi(f)(m)\big)\left(ac\otimes n\right)= b \cdot \big(c\cdot \varpi(f)(m)\big)\left(a\otimes n\right).
\end{align*}
Taking $c=1$ gives the left $A$-linearity of $\varpi(f)(m)$ while taking $b=1$ gives the left $A$-linearity of $\varpi(f)$, whence $\varpi$ is well-defined. On the other hand, for all $g\in \Hom{A}{}{}{}{{_\bullet M}}{\Hom{A}{}{}{}{{{_\bullet A}}\otimes {_\bullet N}}{{_\bullet P}}} $, $m\in M$, $n\in N$ and $a\in A$ we have also
\begin{align*}
\varkappa(g) & (a\cdot (m\otimes n)) = \sum g(a_1\cdot m)(1\otimes (a_2\cdot n)) = \sum \left(a_1\cdot g(m)\right)(1\otimes (a_2\cdot n)) \\
 &  = \sum g(m)(a_1\otimes (a_2\cdot n)) = g(m)(a\cdot (1\otimes n)) = a\cdot g(m)(1\otimes n) = a\cdot \varkappa(g)(m\otimes n),
\end{align*}
which implies that $\varkappa(g)$ is left $A$-linear and $\varkappa$ is well-defined as well. To check the naturality in $M$ and $P$ consider two left $A$-linear morphisms $h:M'\to M$ and $l:P\to P'$. Then
\begin{align*}
\Big(\big(\Hom{A}{}{}{}{A\otimes N}{l}\circ\varpi(f)\circ h\big)(m)\Big)(a\otimes n) & = \big(l\circ \varpi(f)(h(m))\big)(a\otimes n)= l\Big(\big(\varpi(f)(h(m))\big)(a\otimes n)\Big) \\
 & = l\big(f(a\cdot h(m)\otimes n)\big) = l\Big(f\big(h(a\cdot m)\otimes n\big)\Big) \\
 & = \Big(\varpi\big(l\circ f\circ (h\otimes N)\big)(m)\Big)(a\otimes n)
\end{align*}
To conclude, it is enough to check that $\varpi$ and $\varkappa$ are inverses each other. To this aim, we may compute directly
\begin{gather*}
\left(\varpi\varkappa(g)(m)\right)(a\otimes n) = \varkappa(g)(a\cdot m\otimes n) = g(a\cdot m)(1\otimes n) = g(m)(a\otimes n), \\
\left(\varkappa\varpi(f)\right)(m\otimes n) = \left(\varpi(f)(m)\right)(1\otimes n) = f(m\otimes n)
\end{gather*}
for all $m\in N$, $n\in N$, $a\in A$, $f\in \Hom{A}{}{}{}{{_\bullet M}\otimes {_\bullet N}}{{_\bullet P}}$ and $g\in \Hom{A}{}{}{}{{_\bullet M}}{\Hom{A}{}{}{}{{{_\bullet A}}\otimes {_\bullet N}}{{_\bullet P}}}$. Therefore, the first claim holds. The second claim can be proved analogously or may be deduced as follows. Consider the $\K$-modules $(N\otimes A)\tensor{A}M$ and $N\otimes (A\tensor{A}M)$ endowed with the $A$-actions
\begin{gather*}
a\cdot\big(\left(n\otimes b\right)\tensor{A}m\big) \coloneqq  \sum \left(\left(a_1\cdot n\right)\otimes a_2b\right)\tensor{A}m, \\
a\cdot\big(n\otimes \left(b\tensor{A}m\right)\big) \coloneqq  \sum \left(a_1\cdot n\right)\otimes \left(a_2b\tensor{A}m\right),
\end{gather*}
for all $a,b\in A$, $m\in M$ and $n\in N$. The canonical isomorphism $(N\otimes A)\tensor{A}M \cong N\otimes (A\tensor{A}M)$ (``the identity'') is a morphism in ${_A\M}$. Therefore, by the classical hom-tensor adjunction (see \cite[\S3]{Pareigis} for a very general approach), we have a chain of natural isomorphisms
\begin{align*}
\Hom{A}{}{}{}{{_\bullet N}\otimes {_\bullet M}}{{_\bullet P}} & \cong \Hom{A}{}{}{}{{_\bullet N}\otimes \left( {_\bullet A} \tensor{A} {M}\right)}{{_\bullet P}} \cong \Hom{A}{}{}{}{\left( {_\bullet N}\otimes {_\bullet A}\right) \tensor{A} {M}}{{_\bullet P}} \\
 & \cong \Hom{A}{}{}{}{{_\bullet M}}{\Hom{A}{}{}{}{{_\bullet N}\otimes {_\bullet A}}{_\bullet P}}
\end{align*}
whose composition gives exactly $\varpi'$ and $\varkappa'$.
\end{invisible}

Finally, let us recall that the category $\quasihopf{A}$ is a monoidal category in such a way that the forgetful functor $\quasihopf{A}\to\Bim{A}{A}$ is strong monoidal, that is to say, the tensor product is $\tensor{A}$ and the unit object $A$ itself. Given $M,N\in \quasihopf{A}$, the quasi-Hopf bimodule structure on $M\tensor{A}N$ is the following: for every $a,b\in A$, $m\in M$ and $n\in N$
\[
a\cdot (m\tensor{A}n)\cdot b = (a\cdot m) \tensor{A} (n\cdot b) \qquad \text{and} \qquad \delta(m\tensor{A}n) = m_0\tensor{A}n_0\otimes m_1n_1.
\]
Moreover, in light of \cite[Proposition 3.6]{Schauenburg-HopfDouble} the functor $-\otimes A$ is a strong monoidal functor from $(\Lmod{A},\otimes,\K)$ to $\left(\quasihopf{A},\tensor{A},A\right)$. In a nutshell, the argument revolves around the fact that
\begin{equation}\label{eq:tensAstrmon}
\begin{gathered}
\xymatrix@R=0pt{
(V\otimes A)\tensor{A}(W\otimes A) \ar[r]^-{\xi_{V,W}} & (V\otimes W) \otimes A \\
(v\otimes a)\tensor{A}(w\otimes b) \ar@{|->}[r] & \varphi^1\cdot v\otimes \varphi^2a_1\cdot w\otimes \varphi^3a_2b \\
(\Phi^1\cdot v\otimes 1)\tensor{A}(\Phi^2\cdot w\otimes \Phi^3a) & v\otimes w\otimes a \ar@{|->}[l]
}
\end{gathered}
\end{equation}
is an isomorphism of quasi-Hopf bimodules, natural in $V$ and $W$ objects of $\Lmod{A}$. 

\begin{personal}
{\color{blue} Main reference: \cite{Aguiar}. See also \cite[Example 3.2]{Aguiar-LopezFranco}.}
\end{personal}

\begin{remark}
For the sake of the interested reader, there is a categorical reason behind the monoidality of $\quasihopf{A}$. For a quasi-bialgebra $A$, the category $\Bim{A}{A}$ is a \emph{duoidal} (or \emph{2-monoidal}, in the terminology of \cite[Definition 6.1]{Aguiar}) category with monoidal structures $(\tensor{A},A)$ and $(\otimes,\K)$. The structure morphisms connecting the two are
\begin{gather*}
\varepsilon:A\to \K, \qquad \Delta: A\to A\otimes A, \qquad \K\tensor{A}\K\to\K:h\tensor{A}k\mapsto hk, \\
\zeta:(M\otimes P)\tensor{A}(N\otimes Q)\to (M\tensor{A}N)\otimes (P\tensor{A}Q) \\
(m\otimes p)\tensor{A}(n\otimes q)\mapsto (m\tensor{A}n)\otimes (p\tensor{A}q).
\end{gather*}
The quintuple $(A,\mu:A\tensor{A}A\cong A,\id_A,\Delta,\varepsilon)$ is a \emph{bimonoid} in $(\Bim{A}{A},\tensor{A},A,\otimes,\K)$, that is to say, $(A,\mu,\id_A)$ is a monoid in $(\Bim{A}{A},\tensor{A},A)$\begin{invisible}(it is the unit object)\end{invisible} and $(A,\Delta,\varepsilon)$ is a comonoid in $(\Bim{A}{A},\otimes,\K)$, plus certain compatibility conditions between the two structures. 
\begin{invisible}
Associativity:
\begin{gather*}
\def\objectstyle{\scriptstyle}
\def\labelstyle{\scriptscriptstyle}
\xymatrix{
((m\otimes n) \tensor{A} (p\otimes q)) \tensor{A} (x\otimes y) \ar[r]^-{\calpha} \ar[d]_-{\zeta\tensor{A} (X\otimes Y)} & (m\otimes n) \tensor{A} ((p\otimes q) \tensor{A} (x\otimes y)) \ar[d]^-{(M\otimes N)\tensor{A}\zeta} \\
((m \tensor{A}p ) \otimes  (n\tensor{A}q)) \tensor{A} (x\otimes y) \ar[d]_-{\zeta} & (m\otimes n) \tensor{A} ((p\tensor{A} x) \otimes (q\tensor{A} y)) \ar[d]^-{\zeta} \\
((m \tensor{A}p )\tensor{A} x) \otimes ((n\tensor{A}q) \tensor{A}y) \ar[r]_-{\calpha\otimes\calpha} & (m \tensor{A}(p \tensor{A} x)) \otimes (n\tensor{A}(q \tensor{A}y))
} \\
\def\objectstyle{\scriptstyle}
\def\labelstyle{\scriptscriptstyle}
\xymatrix{
((m\otimes p)\otimes x) \tensor{A} ((n\otimes q)\otimes y) \ar[r]^-{\calpha\tensor{A}\calpha} \ar[d]_-{\zeta} & (\Phi^1m\varphi^1\otimes (\Phi^2p\varphi^2\otimes \Phi^3x\varphi^3)) \tensor{A} (\Psi^1n\psi^1\otimes (\Psi^2q\psi^2\otimes \Psi^3y\psi^3)) \ar[d]^-{\zeta} \\
((m\otimes p) \tensor{A} (n\otimes q)) \otimes (x \tensor{A} y) \ar[dd]_-{\zeta\otimes (X\tensor{A}Y)} &  (\Phi^1m\varphi^1\tensor{A}\Psi^1n\psi^1)\otimes  ((\Phi^2p\varphi^2\otimes \Phi^3x\varphi^3) \tensor{A} (\Psi^2q\psi^2\otimes \Psi^3y\psi^3)) \ar[d]_-{(M\tensor{A}N)\otimes \zeta} \\
 & (\Phi^1m\varphi^1\tensor{A}\Psi^1n\psi^1)\otimes  ((\Phi^2p\varphi^2\tensor{A} \Psi^2q\psi^2) \otimes (\Phi^3x\varphi^3\tensor{A} \Psi^3y\psi^3)) \ar@{=}[d] \\
((m \tensor{A}n ) \otimes (p\tensor{A} q)) \otimes (x \tensor{A} y) \ar[r]_-{\calpha} & (\Phi^1m\tensor{A}n\psi^1)\otimes  ((\Phi^2p\tensor{A} q\psi^2) \otimes (\Phi^3x\tensor{A} y\psi^3))
}
\end{gather*}
commutes. Unitality:
\begin{gather*}
\def\objectstyle{\scriptstyle}
\def\labelstyle{\scriptscriptstyle}
\xymatrix{
a\tensor{A} (m\otimes n) \ar[d]_-{\clambda} \ar[r]^-{\Delta\tensor{A}(M\otimes N)} & (a_1\otimes a_2)\tensor{A} (m\otimes n) \ar[d]^-{\zeta} \\
\Delta(a)(m\otimes n) & (a_1\tensor{A} m)\otimes (a_2\tensor{A} n) \ar[l]^-{\clambda\otimes \clambda}
} \quad
\def\objectstyle{\scriptstyle}
\def\labelstyle{\scriptscriptstyle}
\xymatrix{
(m\otimes n)\tensor{A} a \ar[d]_-{\crho} \ar[r]^-{(M\otimes N)\tensor{A}\Delta} & (m\otimes n)\tensor{A} (a_1\otimes a_2) \ar[d]^-{\zeta} \\
(m\otimes n)\Delta(a) & (m\tensor{A} a_1)\otimes (n\tensor{A} a_2) \ar[l]^-{\crho\otimes \crho}
} \\
\def\objectstyle{\scriptstyle}
\def\labelstyle{\scriptscriptstyle}
\xymatrix{
hk\otimes (m\tensor{A}n) \ar[d]_-{\clambda} & (h\tensor{A}k)\otimes (m\tensor{A}n) \ar[l]_-{m\otimes (M\tensor{A}N)} \\
hkm\tensor{A}n & (h\otimes m)\tensor{A} (k\otimes n) \ar[u]_-{\zeta} \ar[l]^-{\clambda\tensor{A}\clambda}
} \quad
\def\objectstyle{\scriptstyle}
\def\labelstyle{\scriptscriptstyle}
\xymatrix{
(m\tensor{A}n)\otimes hk \ar[d]_-{\crho} & (m\tensor{A}n)\otimes (h\tensor{A}k) \ar[l]_-{(M\tensor{A}N)\otimes m} \\
m\tensor{A}nhk & (m\otimes h)\tensor{A} (n\otimes k) \ar[u]_-{\zeta} \ar[l]^-{\crho\tensor{A}\crho}
} 
\end{gather*}
commute. $(A,\Delta,\varepsilon)$ is a comonoid in $(\Bim{A}{A},\otimes,\K)$ and $(\K,m,\varepsilon)$ is a monoid in $(\Bim{A}{A},\tensor{A},A)$.

\medskip 

Compatibilities:
\begin{gather*}
\xymatrix{
(a_1\otimes a_2)\tensor{A}(b_1\otimes b_2) \ar[rr]^-{\zeta} & & (a_1\tensor{A}b_1)\otimes (a_2\tensor{A}b_2) \ar[d]^-{\mu\otimes\mu} \\
a\tensor{A}b \ar[u]^-{\Delta\tensor{A}\Delta} \ar[r]_-{\mu} & ab \ar[r]_-{\Delta} & a_1b_1\otimes a_2b_2
} \\
\xymatrix{
a\tensor{A}b \ar[d]_-{\mu} \ar[r]^-{\varepsilon\tensor{A}\varepsilon} & \varepsilon(a)\tensor{A}\varepsilon({b}) \ar[d]^-{m} \\
ab \ar[r]_-{\varepsilon} & \varepsilon(ab)
} \quad
\xymatrix{
a \ar[d]_-{\Delta} \ar[r]^-{\id_{A}} & a \ar[d]^-{\Delta} \\
a_1\otimes a_2 \ar[r]_-{\id_A\otimes \id_A} & a_1\otimes a_2
} \\
\xymatrix{
 & A \ar[dr]^-{\varepsilon} & \\
A \ar[ur]^-{\id_A} \ar[rr]_-{\varepsilon} & & \K 
}
\end{gather*}
commute.
\end{invisible}
By \cite[Proposition 6.41]{Aguiar}, the category $\quasihopf{A}$ of right comodules over the bimonoid $A$ in $\Bim{A}{A}$ is a monoidal category with tensor product, unit object and constraints induced from $(\Bim{A}{A},\tensor{A},A)$.
\begin{invisible}
It is not true that $-\otimes A$ becomes a monoidal comonad on $(\Bim{A}{A},\otimes,\K)$. The problem is that $-\otimes A : (\Bim{A}{A},\otimes,\K) \to (\Bim{A}{A},\tensor{A},A)$ is monoidal, while $-\otimes A: \Bim{A}{A}\to \Bim{A}{A}$ is a comonad. To be a monoidal comonad it should use only one monoidal structure.
\end{invisible}
\begin{personal}
In light of \cite[Example 3.2]{Aguiar-LopezFranco} this may be called a \emph{linear comonad}.
[This is not enough to conclude also that $-\otimes A$ is monoidal: I want it to be a kind of monoidal comonad on $\Lmod{A}$ or similar.]
\end{personal}
\end{remark}

\begin{personal}
Since bialgebroids over a commutative base are particular examples of bimonoids in a category of bimodules, maybe the present procedure can be adapted to that case.
\end{personal}




\section{Preantipodes and Frobenius functors}\label{sec:HopfBimod}

This section is devoted to the study of a distinguished adjoint triple that naturally arises when dealing with the so-called Structure Theorem for quasi-Hopf bimodules over a quasi-bialgebra $A$ (see \cite[\S3]{HausserNill}, \cite[\S2.2.1]{PhD}, \cite[\S2.1]{Saracco}). We will see that being Frobenius for the functors involved is equivalent to being equivalences and hence to the existence of a preantipode for $A$. As a by-product, we will find a new equivalent condition for a bialgebra to admit an antipode.

\subsection{The main result} 
For every quasi-Hopf bimodule $M$, the quotient $\overline{M}=M/MA^+$ is a left $A$-module with $a\cdot \overline{m}\coloneqq \overline{a\cdot m}$ for all $a\in A,m\in M$. On the other hand, for every left $A$-module $N$ the tensor product $N\otimes A$ is a quasi-Hopf bimodule with
\begin{equation}\label{eq:quasihopfbim}
a\cdot (n\otimes b) \cdot c = a_1\cdot n \otimes a_2bc \quad \text{and} \quad \delta(n\otimes b) = \varphi^1\cdot n \otimes \varphi^2b_1 \otimes \varphi^3b_2
\end{equation}
for all $m\in M, n\in N$ and $a,b,c\in A$ (see Remark \ref{rem:quasihopf} and \cite[\S2.2.1]{PhD}, \cite[\S2.1]{Saracco} for additional details). It is known that these constructions induce an adjunction 
\begin{equation*}
\begin{gathered}
\xymatrix @R=18pt{
\quasihopf{A} \ar@<-0.4ex>@/_/[d]_-{\overline{(-)}} \\
\Lmod{A} \ar@<-0.4ex>@/_/[u]_-{-\otimes A}
}
\end{gathered}.
\end{equation*}
Moreover, the bijection
$
\Hom{A}{}{}{}{{_\bullet M}\otimes {_\bullet N}}{{_\bullet P}}\cong \Hom{A}{}{}{}{{_\bullet M}}{\Hom{A}{}{}{}{{_\bullet A}\otimes {_\bullet N}}{{_\bullet P}}}
$
from \eqref{eq:lmodclosed} induces a natural bijection 
\begin{equation*}
\Hom{A}{}{A}{A}{{_\bullet M}\otimes {{}^{\phantom{\bullet}}_\bullet N_{\bullet}^{\bullet}}}{{{}^{\phantom{\bullet}}_\bullet P_{\bullet}^{\bullet}}}\cong \Hom{A}{}{}{}{{_\bullet M}}{\Hom{A}{}{A}{A}{{_\bullet A}\otimes {{}^{\phantom{\bullet}}_\bullet N_{\bullet}^{\bullet}}}{{{}^{\phantom{\bullet}}_\bullet P_{\bullet}^{\bullet}}}}
\end{equation*}
that makes of $\Hom{A}{}{A}{A}{A\otimes A}{-}$ the right adjoint of the free quasi-Hopf bimodule functor $-\otimes A$. Therefore we have an adjoint triple
\begin{equation}\label{eq:adjtrQuasi}
\overline{\left( -\right) } \ \dashv \ - \otimes A \ \dashv \ \Hom{A}{}{A}{A}{A\otimes A}{-}
\end{equation}
between $\Lmod{A}$ and $\quasihopf{A}$, with units and counits given by
\begin{equation}\label{eq:unitscounitsquasi}
\begin{gathered}
\eta_M:M\to \overline{M}\otimes A, \quad m\mapsto \overline{m_0}\otimes m_1, \qquad \epsilon_N: \overline{(N\otimes A)} \to N, \quad \overline{n\otimes a}\mapsto n\varepsilon(a), \\
\gamma_N: N \to \Hom{A}{}{A}{A}{A\otimes A}{N\otimes A}, \quad n\mapsto \left[a\otimes b\mapsto a\cdot n\otimes b\right], \\
\theta_M:\Hom{A}{}{A}{A}{A\otimes A}{M}\otimes A \to M, \quad f\otimes a \mapsto f(1\otimes 1)\cdot a.
\end{gathered}
\end{equation}

\begin{remark}\label{rem:fullyfaith}
By \cite[Proposition 3.6]{Schauenburg-HopfDouble}, the functor $-\otimes A:\Lmod{A}\to \quasihopf{A}$ is fully faithful. Even if therein $\K$ is assumed to be a field, this hypothesis is not used in the proof, which therefore can be adapted to our context without additional effort. As a consequence, by \cite[Theorem IV.3.1]{MacLane} both $\epsilon$ and $\gamma$ are natural isomorphisms. The interested reader may check directly that, for every $N$ in $\Lmod{A}$,
\begin{gather*}
\epsilon_N^{-1}:N \to \overline{(N \otimes A)}, \quad n \mapsto \overline{n \otimes 1}, \qquad \text{and} \\
\gamma_N^{-1} : \Hom{A}{}{A}{A}{A\otimes A}{N\otimes A} \to N, \quad f \mapsto (N\otimes \varepsilon)(f(1\otimes 1)),
\end{gather*}
as claimed in \cite[Propositions 2.2.1 and 2.2.3]{PhD}.
\end{remark}

Since we are in the situation of \S\ref{sec:adjtriples}, we may consider the natural transformation $\sigma$ whose component at $M\in\quasihopf{A}$ is the $A$-linear map
\begin{equation}\label{eq:sigma1}
\sigma _{M}: \Hom{A}{}{A}{A}{A\otimes A}{M} \rightarrow \overline{M};\quad f\mapsto \overline{f\left( 1\otimes 1\right) }.
\end{equation}

\begin{remark}\label{rem:Hopfclassic}
Three things deserve to be observed before proceeding.
\begin{enumerate}[label=(\alph*), ref={(\alph*)}, leftmargin=0.85cm, labelsep=0.15cm]
\item\label{item1:Hclassic} $A$ admits a preantipode $S$ if and only if either the left-most or the right-most adjunction in \eqref{eq:adjtrQuasi} is an equivalence (whence both are). See \cite[Theorem 2.2.7]{PhD} and \cite[Theorem 4]{Saracco} for further details. In particular, an inverse for $\sigma_M$ in this case is given by
\begin{equation*}
\sigma_M^{-1}:\cl{M} \to \Hom{A}{}{A}{A}{A\otimes A}{M}, \ \overline{m}\mapsto \left[ \left( a\otimes b\right) \mapsto \Phi ^{1}a_{1}\cdot m_{0}\cdot S\left( \Phi ^{2}a_{2}m_{1}\right) \Phi ^{3}b \right].
\end{equation*}
\item\label{item2:Hclassic} In light of equation \eqref{eq:sigmaF} with $\cF=-\otimes A$ and since $A \cong \K\otimes A$ in $\quasihopf{A}$, the component $\sigma_A:\Hom{A}{}{A}{A}{A\otimes A}{A}\to\cl{A}$ is always an isomorphism with inverse given by $\K\to\Hom{A}{}{A}{A}{A\otimes A}{A}, 1_\K\mapsto \left[x\otimes y\mapsto \varepsilon(x)y\right]$.
\item\label{item3:Hclassic} For every $M\in\quasihopf{A}$, the relation $\cl{m\cdot a} = \cl{m}\,\varepsilon(a)$ holds in $\cl{M}$ for all $a\in A$, $m\in M$. We will make often use of it in what follows.
\end{enumerate}
\end{remark}

By Lemma \ref{lemma:frobenius}, the functor $-\otimes A$ is Frobenius if and only if $\sigma$ of \eqref{eq:sigma1} is a natural isomorphism. Thus, let us start by having a closer look at $\Hom{A}{}{A}{A}{A\otimes A}{M}$.

\begin{remark}\label{rem:tau}
Let $M\in\quasihopf{A}$ and consider $f\in \Hom{A}{}{A}{A}{A\otimes A}{M}$. Due to right $A$-linearity, $f$ is uniquely determined by the elements $f(a\otimes 1)$ for $a\in A$. Consider the assignment $T_f:A\to M, \, a\mapsto f(a\otimes 1)$, so that $f(a\otimes b)=T_f(a)\cdot b$ for all $a,b\in A$. From $A$-colinearity of $f$ it follows that 
\begin{equation}\label{eq:deltaf}
\delta_M(T_f(a)) \stackrel{\eqref{eq:quasihopfbim}}{=} f\left(\varphi^1a \otimes \varphi^2\right) \otimes \varphi^3 = T_f(\varphi^1a)\cdot \varphi^2\otimes \varphi^3
\end{equation}
and from left $A$-linearity it follows that
\begin{equation}\label{eq:Tpreant}
T_f(a_1b)\cdot a_2 = f\left(a_1b\otimes a_2\right) \stackrel{\eqref{eq:quasihopfbim}}{=} a\cdot T_f(b)
\end{equation}
for all $a,b\in A$. Denote by $\Hom{}{\dagger}{}{}{A}{M}$ the $\K$-submodule of $\Hom{}{}{}{}{A}{M}$ of those $\K$-linear maps satisfying \eqref{eq:deltaf} and \eqref{eq:Tpreant}. It is an $A$-submodule with respect to the action $(a\triangleright g)(b)\coloneqq g(ba)$ for $a,b\in A$, $g\in \Hom{}{\dagger}{}{}{A}{M}$. The assignment
\[
\Hom{A}{}{A}{A}{A\otimes A}{M} \to \Hom{}{\dagger}{}{}{A}{M}, \quad f\mapsto T_f
\]
is an isomorphism of left $A$-modules. Let now $N$ be any right $A$-module and let $N\cmptens{}A$ be the quasi-Hopf bimodule ${N_\bullet^{\phantom{\bullet}}}\otimes \quasihopfmod{A}$. In light of Remark \ref{rem:quasihopf}, the coaction is given by the composition $\calpha_{N,A,A}^{-1}\circ (N\otimes \Delta)$, which means that
\[
\delta_{N\cmptens{}A}(n\otimes a) = n\cdot \Phi^1\otimes a_1\Phi^2 \otimes a_2\Phi^3.
\]
In light of \eqref{eq:deltaf}, for every $f\in \Hom{A}{}{A}{A}{A\otimes A}{N\cmptens{}A}$ we have
\begin{equation}\label{eq:Tf}
(N\otimes \Delta)\left(T_f(a)\right)\cdot \Phi = \delta_{N\cmptens{}A}\left(T_f(a)\right) = T_f(\varphi^1a)\cdot \varphi^2\otimes \varphi^3
\end{equation}
and, by applying $N\otimes \varepsilon\otimes A$ to both sides of \eqref{eq:Tf},
\begin{gather*}
T_f(a) \stackrel{\eqref{eq:PhiCounital}}{=} \left(N\otimes \varepsilon\otimes A\right)\left((N\otimes \Delta)\left(T_f(a)\right)\cdot \Phi\right) = \left(N\otimes \varepsilon\otimes A\right)\left(T_f(\varphi^1a)\cdot \varphi^2\otimes \varphi^3\right) \\
\stackrel{(*)}{=} \left(N\otimes \varepsilon\right)\left(T_f(\varphi^1a)\right)\cdot \varphi^2\otimes \varphi^3
\end{gather*}
where $(*)$ follows from the fact that the right $A$-action on $N\cmptens{}A$ is given by $(n\otimes a)\cdot b = n\cdot b_1\otimes ab_2$ for all $a,b\in A$, $n\in N$.
If we define $\tau_f:A\to N$ by $\tau_f(a)\coloneqq (N\otimes \varepsilon)T_f(a)$ for all $a\in A$, then it follows that
\begin{equation}\label{eq:ftau}
f(a\otimes b) = T_f(a)\cdot b = \tau_f(\varphi^1a)\cdot\varphi^2b_1\otimes \varphi^3b_2
\end{equation}
for all $a,b\in A$, $f\in \Hom{A}{}{A}{A}{A\otimes A}{N\cmptens{}A}$. Moreover, in view of \eqref{eq:epsivarphi} and since the left $A$-action on $N\cmptens{}A$ is given by $a\cdot(n\otimes b) = n\otimes ab$, applying $N\otimes \varepsilon$ to both sides of \eqref{eq:Tpreant} gives 
\begin{equation}\label{eq:quasiantipode}
\tau_f(a_1b)\cdot a_2 = \varepsilon(a)\tau_f(b)
\end{equation}
for all $a,b\in A$. Denote by $\Hom{}{\star}{}{}{A}{N}$ the family of $\K$-linear morphisms $g:A\to N$ that satisfy \eqref{eq:quasiantipode}. Then we have an isomorphism of left $A$-modules
\begin{equation}\label{eq:tau}
\begin{gathered}
\xymatrix @R=0pt {
{\tau:\Hom{A}{}{A}{A}{A\otimes A}{N\cmptens{}A}} \ar@{<->}[r] & \Hom{}{\star}{}{}{A}{N} \\
{\phantom{\tau:}f} \ar@{|->}[r] & \tau_f \\
{\phantom{\tau:}\big[a\otimes b\mapsto g(\varphi^1a)\cdot \varphi^2b_1\otimes \varphi^3b_2\big]} & g \ar@{|->}[l]
}
\end{gathered}
\end{equation}
where the $A$-module structure on $\Hom{}{\star}{}{}{A}{N}$ is the one induced by $\Hom{}{}{}{}{A}{N}$, that is, $(a\triangleright g)(b)=g(ba)$ for all $a,b\in A$ and $g\in \Hom{}{}{}{}{A}{N}$.
\end{remark}

Our first aim is to show that if $\sigma_M$ is invertible for every $M\in\quasihopf{A}$, then $A$ admits a preantipode. Let us keep the notation introduced in Remark \ref{rem:tau} and consider the distinguished quasi-Hopf bimodule $\widehat{A}=A\cmptens{}A\coloneqq {A_{\bullet }^{}} \otimes { _{\bullet }^{\phantom{\bullet}}A_{\bullet }^{\bullet}}$ and the component
\begin{equation}\label{eq:sigmahat}
\sigma_{\widehat{A }}: \Hom{A}{}{A}{A}{A\otimes A}{A\cmptens{}A} \rightarrow \overline{A\cmptens{}A}, \qquad f\mapsto \overline{\tau_f(\varphi^1)\varphi^2\otimes \varphi^3}.
\end{equation}
Observe that, in light of the structures on $A\otimes A$ and $A\cmptens{}A$, bilinearity and colinearity of $f$ can be expressed explicitly by
\begin{equation}\label{eq:alllin}
\begin{gathered}
f\left(a_1x\otimes a_2yb\right) = (1\otimes a)f\left(x\otimes y\right)\Delta(b) \quad \text{and} \\
 (A\otimes \Delta)\left(f\left(a\otimes b\right)\right)\cdot \Phi = f\left(\varphi^1a\otimes \varphi^2b_1\right) \otimes \varphi^3b_2
\end{gathered}
\end{equation}
for every $f\in \Hom{A}{}{A}{A}{A\otimes A}{A\cmptens{}A}$ and $a,b,x,y\in A$.

\begin{remark}\label{rem:sigmabal}
For all $a,b,x,y\in A$, $f\in\Hom{A}{}{A}{A}{A\otimes A}{A\cmptens{} A}$, the rules
\begin{equation}\label{eq:triang}
(a\otimes b)\triangleright \overline{x\otimes y}\coloneqq \overline{ax\otimes by} \quad \text{and} \quad ((a\otimes b)\triangleright f)(x\otimes y)\coloneqq \left( a\otimes 1\right) f\left( xb\otimes y\right)
\end{equation}
provide actions of $A\otimes A$ on $\overline{A \cmptens{} A}$ and $\Hom{A}{}{A}{A}{A\otimes A}{A\cmptens{} A}$, respectively, and the ordinary left $A$-action on $\overline{A\cmptens{}A}$ satisfies $a\cdot\left(\cl{x\otimes y}\right) = (1\otimes a)\triangleright\cl{x\otimes y}$.
\begin{personal}
The first one is well-defined. Let us check the second is well-defined
\begin{align*}
\left( (a\otimes b)\triangleright f\right) \left( c_{1}x\otimes c_{2}yd\right) & = \left( a\otimes 1\right) f\left( c_{1}xb\otimes c_{2}yd\right) =\left( a\otimes 1\right) \left( 1\otimes c\right) f\left( xb\otimes y\right) \Delta \left( d\right) \\
& =\left( 1\otimes c\right) \Big( \left( a\otimes 1\right) f\left( xb\otimes y\right) \Big) \Delta \left( d\right) = c\cdot \left( (a\otimes b)\triangleright f\right) \left( x\otimes y\right) \cdot d 
\end{align*}
whence it is still bilinear and
\begin{align*}
\big( & \left( (a\otimes b)\triangleright f\right) \left( x\otimes y\right)\big) _{0} \otimes \big( \left( (a\otimes b)\triangleright f\right) \left( x\otimes y\right) \big) _{1} = \big( \left( a\otimes 1\right) f\left( xb\otimes y\right) \big) _{0}\otimes \big( \left( a\otimes 1\right) f\left( xb\otimes y\right) \big) _{1} \\
&=\left( A\otimes \Delta \right) \left( \left( a\otimes 1\right) f\left( xb\otimes y\right) \right) \cdot \Phi = \left( a\otimes 1\otimes 1\right) \Big( \left( A\otimes \Delta \right) \left( f\left( xb\otimes y\right) \right) \cdot \Phi \Big) \\
&= \left( a\otimes 1\right) f\left( xb\otimes y\right) _{0}\otimes f\left( xb\otimes y\right) _{1} = \left( a\otimes 1\right) f\left( \left( xb\otimes y\right) _{0}\right) \otimes \left( xb\otimes y\right) _{1} \\
& = \left( a\otimes 1\right) f\left( \varphi^1xb\otimes \varphi^2y_1\right) \otimes \varphi^3y_2 = \left( (a\otimes b)\triangleright f\right) \left( \left( x\otimes y\right) _{0}\right) \otimes \left( x\otimes y\right) _{1}
\end{align*}
so that it is still colinear.
\end{personal}
Since
\begin{gather*}
\sigma_{\widehat{A }}\left( (a\otimes b)\triangleright f\right) \stackrel{\eqref{eq:sigma1}}{=} \overline{\left( (a\otimes b)\triangleright f\right)(1\otimes 1)} \stackrel{\eqref{eq:triang}}{=} \overline{\left( a\otimes 1\right) f\left( b\otimes 1\right) } \stackrel{\eqref{eq:actionhom}}{=} \left( a\otimes 1\right) \triangleright \left(\overline{ \left(b\cdot f\right)\left( 1\otimes 1\right) }\right) \\
\stackrel{\eqref{eq:sigma1}}{=} \hspace{-1pt} ( a\otimes 1) \hspace{-1pt}  \triangleright \hspace{-1pt}  \sigma_{\what{A}}(b\cdot f)  \hspace{-1pt} \stackrel{(*)}{=} \hspace{-1pt}  ( a\otimes 1) \hspace{-1pt}  \triangleright  \hspace{-1pt} \left(b\cdot \sigma_{\what{A}}(f)\right)  \hspace{-1pt} = \hspace{-1pt}  ( a\otimes 1) \triangleright\left( (1\otimes b)\triangleright\sigma_{\what{A}}(f)\right)  \hspace{-1pt}  =  \hspace{-1pt} (a\otimes b)\triangleright \sigma_{\what{A}}(f)
\end{gather*}
for all $a,b\in A$ (where $(*)$ follows by $A$-linearity of $\sigma_{\widehat{A }}$), we have that $\sigma_{\widehat{A }}$ is $A\otimes A$-linear with respect to these action. If $\sigma_{\widehat{A }}$ is invertible, then $\sigma_{\widehat{A }}^{-1}$ is $A\otimes A$-linear as well and hence
\begin{equation}\label{eq:sigmatriang}
\sigma_{\widehat{A }}^{-1}(\cl{a\otimes b})=(a\otimes b)\triangleright \sigma_{\widehat{A }}^{-1}(\cl{1\otimes 1}).
\end{equation}
As a consequence, and by right $A$-linearity of $\sigma_{\widehat{A }}^{-1}\left( \overline{1\otimes 1}\right)$,
\begin{equation}\label{eq:sigmahatbalanced0}
\begin{gathered}
\sigma_{\widehat{A }}^{-1}\left( \overline{a\otimes b}\right) \left(x\otimes y\right) \stackrel{\eqref{eq:sigmatriang}}{=} \left((a\otimes b)\triangleright \sigma_{\widehat{A }}^{-1}(\cl{1\otimes 1})\right) \left(x\otimes y\right) \\
\stackrel{\eqref{eq:triang}}{=} (a\otimes 1)\sigma_{\widehat{A }}^{-1}(\cl{1\otimes 1}) \left(xb\otimes y\right) \stackrel{\eqref{eq:alllin}}{=} (a\otimes 1)\sigma_{\widehat{A }}^{-1}(\cl{1\otimes 1}) \left(xb\otimes 1\right)\Delta \left(y\right)
\end{gathered}
\end{equation}
for all $a,b,x,y\in A$. In particular, we have
\begin{equation}\label{eq:sigmahatbalanced}
\sigma_{\widehat{A }}^{-1}\left( \overline{1\otimes a}\right) \left(1\otimes 1\right) = \sigma_{\widehat{A }}^{-1}\left( \overline{1\otimes 1}\right) \left( a\otimes 1\right) .  
\end{equation}
\end{remark}


\begin{proposition}\label{prop:prenatipode}
If $\sigma_{\widehat{A}}$ is invertible, then $S\coloneqq \tau\left( \sigma_{\widehat{A }}^{-1}\left( \overline{1\otimes 1}\right)\right)$, given by 
\begin{equation}\label{eq:S}
S(a)=(A\otimes \varepsilon)\left(\sigma_{\widehat{A }}^{-1}\left( \overline{1\otimes 1}\right)(a\otimes 1)\right)
\end{equation}
for every $a\in A$, satisfies $S(a_1b)a_2 = \varepsilon(a)S(b) = a_1S(ba_2)$ for all $a,b\in A$.
\end{proposition}

\begin{proof}
Since $\sigma_{\widehat{A }}^{-1}\left( \overline{1\otimes 1}\right)$ belongs to $\Hom{A}{}{A}{A}{A\otimes A}{A\cmptens{}A}$, it follows from relation \eqref{eq:ftau} that $\sigma_{\widehat{A }}^{-1}\left(\overline{1\otimes 1}\right)\left(a\otimes 1\right) = S(\varphi^1a)\varphi^2\otimes \varphi^3 $ and hence
\begin{equation}\label{eq:sigma-1}
\sigma_{\widehat{A }}^{-1}\left(\overline{a\otimes b}\right)\left(x\otimes y\right)\stackrel{\eqref{eq:sigmahatbalanced0}}{=} (a\otimes 1)\sigma_{\widehat{A }}^{-1}\left( \overline{1\otimes 1}\right)(xb\otimes 1)\Delta(y) = aS(\varphi^1xb)\varphi^2y_1\otimes \varphi^3y_2
\end{equation}
for all $a,b,x,y\in A$. Now, $S\left( a_{1}b\right) a_{2} = \varepsilon(a)S(b)$ is relation \eqref{eq:quasiantipode} for $f=\sigma_{\widehat{A}}^{-1}(\cl{1\otimes 1})$. Moreover, since $\cl{a_1\otimes a_2} = \cl{\left(1\otimes 1\right)\cdot a} = \left(\cl{1\otimes 1}\right)\varepsilon(a)$ by definition of $\cl{A\cmptens{}A}$, we have
\begin{gather*}
a_{1}S\left(b a_{2} \right) \stackrel{\eqref{eq:epsivarphi}}{=} (A\otimes \varepsilon)\Big(a_1S(\varphi^1ba_2)\varphi^2\otimes \varphi^3\Big) \stackrel{\eqref{eq:sigma-1}}{=} (A\otimes \varepsilon)\Big(\sigma_{\widehat{A}}^{-1}\left(\overline{a_1\otimes a_2}\right)\left(b\otimes 1\right)\Big) \\
 = (A\otimes \varepsilon)\Big(\sigma_{\widehat{A}}^{-1}\left(\overline{1\otimes 1}\right)\left(b\otimes 1\right)\Big)\varepsilon(a) = \varepsilon(a)S(b)
\end{gather*}
for all $a,b\in A$ and the proof is complete.
\end{proof}

It view of Proposition \ref{prop:prenatipode}, the endomorphism $S=\tau\left( \sigma_{\widehat{A }}^{-1}\left( \overline{1\otimes 1}\right)\right)$ is a preantipode if and only if $\Phi^1S(\Phi^2)\Phi^3=1$ (see \eqref{eq:defpreantipode}). The forthcoming lemmata are intermediate steps toward the proof of this latter identity.

\begin{lemma}\label{lemma:cltensor}
For $M\in\Lmod{A}$ and $N\in\Bim{A}{A}$ we have $\cl{M\otimes N}\cong M\otimes \cl{N}$ in $\Lmod{A}$ via the assignment $\cl{m\otimes n}\mapsto m\otimes \cl{n}$.
\end{lemma}

\begin{proof}
Since $\cl{N}=N/NA^+$ and $A/A^+\cong \K$, the thesis follows from the isomorphisms
\[
\cl{M\otimes N} \cong (M\otimes N)\tensor{A}\K\cong M\otimes (N\tensor{A}\K)\cong M\otimes \cl{N}. \qedhere
\]
\end{proof}

\begin{lemma}\label{lemma:sigmagen}
If $\sigma$ is a natural isomorphism, then for any $M\in\quasihopf{A}$, $m\in M$ and $x,y\in A$,
\begin{equation}\label{eq:sigmainv}
\sigma_M^{-1}(\cl{m})(x\otimes y) = \Phi^1x_1\cdot m_0 \cdot S(\Phi^2x_2m_1)\Phi^3y.
\end{equation}
\end{lemma}

\begin{proof}
Set $A_2 \coloneqq  \lmod{A}\otimes \rmod{A} \otimes \quasihopfmod{A} \in \quasihopf{A}$ with explicit structures
\begin{equation}\label{eq:strA2}
\begin{gathered}
a\cdot (u \otimes v\otimes w) \cdot b = a_1 u \otimes v b_1 \otimes a_2 w b_2 \qquad \text{and} \\
 \delta(u\otimes v\otimes w) = \varphi^1u \otimes v\Phi^1 \otimes \varphi^2w_1\Phi^2 \otimes \varphi^3w_2\Phi^3
\end{gathered}
\end{equation}
for all $a,b,u,v,w\in A$. 
Denote by $\iota: A\otimes \cl{A\cmptens{}A} \to \cl{A_2}$ the isomorphism of Lemma \ref{lemma:cltensor}. Consider also the left $A$-linear morphism
\begin{gather}\label{eq:homtensor}
\begin{gathered}
\tilde{\iota}: A \otimes \Hom{A}{}{A}{A}{A\otimes A}{A\cmptens{}A} \to \Hom{A}{}{A}{A}{A\otimes A}{A_2} \\
a\otimes f \longmapsto \left[x\otimes y\mapsto \Phi^1x_1a\otimes f(\Phi^2x_2\otimes \Phi^3y)\right].
\end{gathered}
\end{gather}
It is well-defined because the following direct computation 
\begin{align*}
\tilde{\iota}(a\otimes f)(b_1x\otimes b_2yc) & = \Phi^1b_{1_1}x_1a\otimes f(\Phi^2b_{1_2}x_2\otimes \Phi^3b_2yc) \stackrel{\eqref{eq:quasi-coassociativity}}{=} b_{1}\Phi^1x_1a\otimes f(b_{2_1}\Phi^2x_2\otimes b_{2_2}\Phi^3yc) \\
 & \hspace{-3pt}\stackrel{\eqref{eq:alllin}}{=} b_1\Phi^1x_1a\otimes (1\otimes b_2)f(\Phi^2x_2\otimes \Phi^3y)\Delta(c) \stackrel{\eqref{eq:strA2}}{=} b\cdot (\tilde{\iota}(a\otimes f)(x\otimes y))\cdot c
\end{align*}
entails that $\tilde{\iota}(a\otimes f)$ is $A$-bilinear and 
\begin{align*}
\delta(\tilde{\iota}(a\otimes f)(x\otimes y)) & \stackrel{\eqref{eq:homtensor}}{=} \delta\left(\Phi^1x_1a\otimes f(\Phi^2x_2\otimes \Phi^3y)\right) \\ 
& \stackrel{\eqref{eq:strA2}}{=} \varphi^1\Phi^1x_1a\otimes (1\otimes \varphi^2\otimes \varphi^3)(A\otimes \Delta)f(\Phi^2x_2\otimes \Phi^3y)\cdot \Phi \\
& \stackrel{\eqref{eq:alllin}}{=} \varphi^1\Phi^1x_1a\otimes (1\otimes \varphi^2\otimes \varphi^3)\left(f(\phi^1\Phi^2x_2\otimes \phi^2\Phi^3_1y_1)\otimes \phi^3\Phi^3_2y_2\right) \\
& \stackrel{\eqref{eq:alllin}}{=} \varphi^1\Phi^1x_1a\otimes f(\varphi^2_1\phi^1\Phi^2x_2\otimes \varphi^2_2\phi^2\Phi^3_1y_1)\otimes \varphi^3\phi^3\Phi^3_2y_2 \\
& \hspace{2pt}\stackrel{\eqref{eq:PhiCocycle}}{=} \Phi^1\varphi^1_1x_1a\otimes f(\Phi^2\varphi^1_2x_2\otimes \Phi^3\varphi^2y_1)\otimes \varphi^3y_2 \\
& \stackrel{\eqref{eq:homtensor}}{=} \tilde{\iota}(a\otimes f)(\varphi^1x\otimes \varphi^2y_1)\otimes \varphi^3y_2
\end{align*}
implies that it is colinear for all $a\in A$, $f\in \Hom{A}{}{A}{A}{A\otimes A}{A\cmptens{}A}$. The $A$-linearity of $\tilde{\iota}$ follows from
\begin{gather*}
\tilde{\iota}(b_1a\otimes b_2\cdot f)(x\otimes y) \stackrel{\eqref{eq:homtensor}}{=} \Phi^1x_1b_1a\otimes (b_2\cdot f)(\Phi^2x_2\otimes \Phi^3y) \\
\stackrel{\eqref{eq:actionhom}}{=} \Phi^1x_1b_1a\otimes f(\Phi^2x_2b_2\otimes \Phi^3y) \stackrel{\eqref{eq:homtensor}}{=} \tilde{\iota}(a\otimes f)(xb\otimes y) \stackrel{\eqref{eq:actionhom}}{=} (b\cdot\tilde{\iota}(a\otimes f))(x\otimes y)
\end{gather*}
for all $a,b,x,y\in A$ and $f\in \Hom{A}{}{A}{A}{A\otimes A}{A\cmptens{}A}$.
\begin{invisible}
Concretely, the idea is the following: if $\lmod{V}\in\Lmod{A}$ and $\quasihopfmod{M}\in\quasihopf{A}$ then $\lmod{V}\otimes \quasihopfmod{M}\in\quasihopf{A}$. If $\lmod{V}\in\Lmod{A}$ and $f:\quasihopfmod{M}\to\quasihopfmod{N}\in\quasihopf{A}$  then $\lmod{V}\otimes f:\lmod{V}\otimes\quasihopfmod{M}\to\lmod{V}\otimes \quasihopfmod{N}\in\quasihopf{A}$. Now, $\Delta:\lmod{A}\to\lmod{A}\otimes \lmod{A}\in\Lmod{A}$, $\calpha_{A,A,A}:\left(\lmod{A}\otimes \lmod{A}\right)\otimes \quasihopfmod{A} \to \lmod{A}\otimes \left(\lmod{A}\otimes \quasihopfmod{A}\right)\in\quasihopf{A}$ and the above map is given as the following composition
\begin{gather*}
A \otimes \Hom{A}{}{A}{A}{A\otimes A}{A\cmptens{}A} \to \Hom{A}{}{}{}{\lmod{A}}{\lmod{A}} \otimes \Hom{A}{}{A}{A}{\lmod{A}\otimes \quasihopfmod{A}}{\rmod{A}\otimes\quasihopfmod{A}} \to \\
\to \Hom{A}{}{A}{A}{\lmod{A}\otimes \left(\lmod{A}\otimes \quasihopfmod{A}\right)}{\lmod{A}\otimes \left(\rmod{A}\otimes\quasihopfmod{A}\right)} \to \\
\to \Hom{A}{}{A}{A}{\lmod{A}\otimes \quasihopfmod{A}}{(\lmod{A}\otimes \rmod{A})\otimes \quasihopfmod{A}}
\end{gather*}
that is to say, since $(\lmod{A}\otimes \rmod{A})\otimes \quasihopfmod{A} = \lmod{A}\otimes (\rmod{A}\otimes \quasihopfmod{A})$ in $\quasihopf{A}$,
\[
a\otimes f \mapsto \left(\lmod{A}\otimes f\right) \circ \left(f^r_a\otimes \left(\lmod{A}\otimes \quasihopfmod{A}\right)\right) \circ \calpha_{A,A,A} \circ \left(\Delta\otimes \quasihopfmod{A}\right)
\]
It can also be checked directly. 
\end{invisible}

Let us show that the diagram
\begin{equation}\label{eq:diagram}
\begin{gathered}
\xymatrix @C=40pt @R=15pt {
A \otimes \Hom{A}{}{A}{A}{A\otimes A}{A\cmptens{}A} \ar[r]^-{A\otimes \sigma_{\what{A}}} \ar[d]_-{\tilde{\iota}} & A \otimes \cl{A\cmptens{}A} \ar[d]^-{\iota} \\
\Hom{A}{}{A}{A}{A\otimes A}{A_2} \ar[r]_-{\sigma_{A_2}} & \cl{A_2}
}
\end{gathered}
\end{equation}
 is commutative. For every $a\in A$ and $f\in \Hom{A}{}{A}{A}{A\otimes A}{A\cmptens{}A}$ compute
\begin{gather*}
\sigma_{A_2}\left(\tilde{\iota}(a\otimes f)\right) \stackrel{\eqref{eq:sigma1}}{=} \cl{\tilde{\iota}(a\otimes f)(1\otimes 1)} \stackrel{\eqref{eq:homtensor}}{=} \cl{\Phi^1a\otimes f(\Phi^2\otimes \Phi^3)} \\
 \stackrel{\eqref{eq:alllin}}{=} \cl{\Phi^1a\otimes f(\Phi^2\otimes 1)\Delta(\Phi^3)} \stackrel{\eqref{eq:strA2}}{=} \cl{\left(\Phi^1a\otimes f(\Phi^2\otimes 1)\right) \cdot \Phi^3} \\
\stackrel{(*)}{=} \cl{a\otimes f(1\otimes 1)} = \iota((A\otimes \sigma_{\what{A}})(a\otimes f)),
\end{gather*}
where $(*)$ follows from the fact that $(A\otimes A\otimes \varepsilon)(\Phi)=1\otimes 1$ and the definition of $\cl{A_2}$. Therefore, in light of \eqref{eq:sigma-1}, for all $a,b,c,x,y\in A$ we have
\begin{equation}\label{eq:sigmaA2}
\begin{gathered}
\sigma_{A_2}^{-1}\left(\cl{a\otimes b\otimes c}\right)(x\otimes y) \stackrel{\eqref{eq:diagram}}{=} \left(\tilde{\iota} \circ \left(A\otimes \sigma_{\what{A}}^{-1}\right)\circ\iota^{-1}\right)\left(\cl{a\otimes b\otimes c}\right)(x\otimes y) \\
= \tilde{\iota}\left(a\otimes \sigma_{\what{A}}^{-1}\left(\cl{b\otimes c}\right)\right)(x\otimes y) \stackrel{\eqref{eq:homtensor}}{=} \Phi^1x_1a\otimes \sigma_{\what{A}}^{-1}\left(\cl{b\otimes c}\right)(\Phi^2x_2\otimes \Phi^3y) \\
\stackrel{\eqref{eq:sigma-1}}{=} \Phi^1x_1a\otimes bS(\varphi^1\Phi^2x_2c)\varphi^2\Phi^3_1y_1\otimes \varphi^3\Phi^3_2y_2.
\end{gathered}
\end{equation}
Now, for any $N\in\Bim{A}{A}$ consider $N\otimes A = \bimod{N}\otimes \quasihopfmod{A} \in\quasihopf{A}$. For every $n\in N$, the assignment $f_n:\lmod{A}\otimes \rmod{A}\to N, a\otimes b\mapsto a\cdot n\cdot b$, is a well-defined $A$-bilinear morphism. Naturality of $\sigma^{-1}$ implies then that
\begin{gather}\label{eq:sigmaN}
\begin{gathered}
\sigma_{N\otimes A}^{-1}(\cl{n\otimes b})(x\otimes y) \stackrel{(\text{nat})}{=} \left(f_n\otimes A\right)\left(\sigma_{A_2}^{-1}\left(\cl{1\otimes 1\otimes b}\right)(x\otimes y)\right) \\
 \stackrel{\eqref{eq:sigmaA2}}{=} \left(f_n\otimes A\right)\left(\Phi^1x_1\otimes S(\varphi^1\Phi^2x_2b)\varphi^2\Phi^3_1y_1\otimes \varphi^3\Phi^3_2y_2\right) \\
 = \Phi^1x_1\cdot n\cdot S(\varphi^1\Phi^2x_2b)\varphi^2\Phi^3_1y_1\otimes \varphi^3\Phi^3_2y_2
\end{gathered}
\end{gather}
for all $n\in N$, $b,x,y\in A$. Finally, the coaction $\delta_M:\quasihopfmod{M}\to \bimod{M}\otimes \quasihopfmod{A}$ is a well-defined morphism in $\quasihopf{A}$ and hence we may resort again to naturality of $\sigma^{-1}$ to get that
\begin{gather*}
\delta_M\left(\sigma_M^{-1}(\cl{m})(x\otimes y)\right) \stackrel{(\text{nat})}{=} \sigma_{M\otimes A}^{-1}(\cl{m_0\otimes m_1})(x\otimes y) \\
\stackrel{\eqref{eq:sigmaN}}{=} \Phi^1x_1\cdot m_0\cdot S(\varphi^1\Phi^2x_2m_1)\varphi^2\Phi^3_1y_1\otimes \varphi^3\Phi^3_2y_2
\end{gather*}
for all $m\in M$, $x,y\in A$. Applying $M\otimes \varepsilon$ to both sides and recalling \eqref{eq:epsivarphi} give the result.
\end{proof}

\begin{proposition}\label{prop:Phi}
If $\sigma$ is a natural isomorphism, then $\Phi^1S(\Phi^2)\Phi^3=1$.
\end{proposition}

\begin{proof}
By Lemma \ref{lemma:sigmagen}, for every $M\in\quasihopf{A}$ and for all $m\in M$, $x,y\in A$ we have $\sigma_M^{-1}(\cl{m})(x\otimes y) = \Phi^1x_1\cdot m_0 \cdot S(\Phi^2x_2m_1)\Phi^3y$. For $M=A$ and $m=x=y=1$ this implies
\[
1\stackrel{(*)}{=}\sigma_A^{-1}(\cl{1})(1\otimes 1) = \Phi^1S(\Phi^2)\Phi^3
\]
where $(*)$ follows by \ref{item2:Hclassic} of Remark \ref{rem:Hopfclassic}.
\end{proof}

Summing up, we have the following central result.

\begin{theorem}\label{thm:mainThm}
The following are equivalent for a quasi-bialgebra $A$:
\begin{enumerate}[label=(\arabic*), ref=\emph{(\arabic*)}, leftmargin=1cm, labelsep=0.3cm]
\item\label{item:main1} $A$ admits a preantipode;
\item\label{item:main2} $-\otimes A:\left(\Lmod{A},\otimes,\K\right)\to \left(\quasihopf{A},\tensor{A},A\right)$ is a monoidal equivalence of categories;
\item\label{item:main3} $-\otimes A:\Lmod{A}\to\quasihopf{A}$ is Frobenius;
\item\label{item:main4} $\sigma_M:\Hom{A}{}{A}{A}{A\otimes A}{M}\to \cl{M}, f\mapsto \cl{f(1\otimes 1)}$ is an isomorphism for every $M\in\quasihopf{A}$.
\end{enumerate}
\end{theorem}

\begin{proof}
The proof of the equivalence between \ref{item:main1} and \ref{item:main2} is contained in \cite[Theorem 3 and subsequent discussion]{Saracco}, but without explicit mention to the monoidality of the functor $-\otimes A:\left(\Lmod{A},\otimes,\K\right)\to \left(\quasihopf{A},\tensor{A},A\right)$. A more exhaustive proof can be found in \cite[Theorem 2.2.7]{PhD}. The implication from \ref{item:main2} to \ref{item:main3} is clear and the equivalence between \ref{item:main3} and \ref{item:main4} follows from Lemma \ref{lemma:frobenius}. Finally, the implication \ref{item:main4} $\Rightarrow$ \ref{item:main1} follows from Proposition \ref{prop:prenatipode} and Proposition \ref{prop:Phi}.
\end{proof}

The subsequent corollary improves considerably \cite[Proposition A.3]{Saracco-Tannaka}.

\begin{corollary}
Let $A$ be a quasi-bialgebra with preantipode $S$. For all $a,b\in A$ we have $S(ab)=S(\varphi^1b)\varphi^2S(a\varphi^3)$.
\end{corollary}

\begin{proof}
For every $f\in\Hom{A}{}{A}{A}{A\otimes A}{A\cmptens{}A}$ and $a,b,c\in A$ we have
\begin{align*}
\tau_f(\varphi^1ab) & \varphi^2c_1\otimes \varphi^3c_2 \stackrel{\eqref{eq:ftau}}{=} f(ab\otimes c) = \sigma_{\what{A}}^{-1}\left(\sigma_{\what{A}}(f)\right)(ab\otimes c) = \sigma_{\what{A}}^{-1}\left(\cl{f(1\otimes 1)}\right)(ab\otimes c) \\
  & \stackrel{\eqref{eq:triang}}{=} \left((1\otimes b)\triangleright\sigma_{\what{A}}^{-1}\left(\cl{f(1\otimes 1)}\right)\right)(a\otimes c) \stackrel{\eqref{eq:sigmatriang}}{=} \sigma_{\what{A}}^{-1}\left((1\otimes b)\triangleright\cl{f(1\otimes 1)}\right)(a\otimes c) \\
 & \stackrel{\eqref{eq:triang}}{=} \sigma_{\what{A}}^{-1}\left(\cl{(1\otimes b)f(1\otimes 1)}\right)(a\otimes c) \stackrel{\eqref{eq:alllin}}{=} \sigma_{\what{A}}^{-1}\left(\cl{f(b_1\otimes b_2)}\right)(a\otimes c) \\
 & \stackrel{\eqref{eq:alllin}}{=} \sigma_{\what{A}}^{-1}\left(\cl{f(b_1\otimes 1) \cdot b_2}\right)(a\otimes c) = \sigma_{\what{A}}^{-1}\left(\cl{f(b\otimes 1)}\right)(a\otimes c) \\
 & \stackrel{\eqref{eq:ftau}}{=} \sigma_{\what{A}}^{-1}\left(\cl{\tau_f(\psi^1b)\psi^2\otimes \psi^3}\right)(a\otimes c) \stackrel{\eqref{eq:sigma-1}}{=} \tau_f(\psi^1b)\psi^2S(\varphi^1a\psi^3)\varphi^2c_1\otimes \varphi^3c_2,
\end{align*}
so that, by applying $A\otimes \varepsilon$ to both sides and taking $c=1$, $\tau_f(ab)=\tau_f(\varphi^1b)\varphi^2S(a\varphi^3)$. Since $\tau$ is bijective and $S\in \Hom{}{\star}{}{}{A}{A}$, there exists $f\in \Hom{A}{}{A}{A}{A\otimes A}{A\cmptens{}A}$ such that $\tau_f=S$ and so $S(ab) = S(\varphi^1b)\varphi^2S(a\varphi^3)$ for all $a,b\in A$.
\end{proof}

\begin{remark}
At the present moment it is not clear to us if there exists a quasi-Hopf bimodule $M$ such that $\sigma$ is a natural isomorphism if and only if $\sigma_M$ is an isomorphism. 
\begin{invisible}
Our candidate would have been $A\cmptens{}A$, but proving that $\Phi^1S(\Phi^2)\Phi^3=1$ out of the invertibility of $\sigma_{\what{A}}$ seems to be more involved than expected. 
In fact, consider for example the group $\C$-bialgebra $A=\C G$ over a finite group $G$ with $\Phi\coloneqq 1\otimes 1\otimes 1$. Fix $t\coloneqq |G|^{-1}\sum_{g\in G}g$ a total integral in $A$. The assignment $s(a)=at=\varepsilon(a)t$ for all $a\in B$ satisfies
\[
a_1s(ba_2) = a_1\varepsilon(a_2)\varepsilon(b)t = \varepsilon(a)s(b) = \varepsilon(a_1)\varepsilon(b)ta_2 = s(a_1b)a_2
\]
for all $a,b\in B$, but obviously $s(1)=t\neq 1$.
By running through the whole proof of Theorem \ref{thm:mainThm} again, one may observe that a necessary and sufficient condition for having that $S$ is a preantipode would have been the invertibility of $\sigma_{\what{A}}$, $\sigma_{A_2}$ and $\sigma_{A\otimes A}$. Nevertheless, we don't know if there is some redundancy between them or not.
\end{invisible}
\end{remark}

\begin{personal}
[The following seemed surprising at a first sight, but in fact it is nothing new: already $\eta_M^{-1}(\cl{m}\otimes a) = \Phi^1m_0S(\Phi^2m_1)\Phi^3a = (\theta_M\circ (\sigma_M^{-1}\otimes A))(\cl{m}\otimes a)$ was colinear at the time of the structure theorem, but this didn't help directly to find the anti-comultiplicativity formula.]

{\color{blue}
Since for every $M\in\quasihopf{A}$ and every $m\in M$, $\sigma_M^{-1}(\cl{m})\in\Hom{A}{}{A}{A}{A\otimes A}{M}$, we have that
\begin{gather*} 
\Phi^1_1x_{1_1}\cdot m_{0_0} \cdot S(\Phi^2x_2m_1)_1\Phi^3_1y_1\otimes \Phi^1_2x_{1_2}\cdot m_{0_1} \cdot S(\Phi^2x_2m_1)_2\Phi^3_2y_2 =\delta_M\left(\sigma_M^{-1}(\cl{m})(x\otimes y)\right) \\
= \sigma_M^{-1}(\cl{m})(\varphi^1x\otimes \varphi^2y_1)\otimes \varphi^3y_2 = \Phi^1\varphi^1_1x_1\cdot m_0 \cdot S(\Phi^2\varphi^1_2x_2m_1)\Phi^3\varphi^2y_1\otimes \varphi^3y_2 \\
 \stackrel{\eqref{eq:PhiCocycle}}{=} \Phi^1x_1\cdot m_0\cdot S(\varphi^1\Phi^2x_2m_1)\varphi^2\Phi^3_1y_1\otimes \varphi^3\Phi^3_2y_2
\end{gather*}
In particular,
\[
\Phi^1_1\cdot m_{0_0} \cdot S(\Phi^2m_1)_1\Phi^3_1\otimes \Phi^1_2\cdot m_{0_1} \cdot S(\Phi^2m_1)_2\Phi^3_2 = \Phi^1\cdot m_0\cdot S(\varphi^1\Phi^2m_1)\varphi^2\Phi^3_1\otimes \varphi^3\Phi^3_2.
\]
This should encode somehow a kind of anti-comultiplicativity of the preantipode. Check if the anti-comultiplicativity known works for this case.
}
\end{personal}

Recall that a bialgebra $B$ is in particular a quasi-bialgebra with $\Phi=1\otimes 1\otimes 1$. 
Moreover, $B$ is a Hopf algebra if and only if, as a quasi-bialgebra, it admits a preantipode. Therefore, from Theorem \ref{thm:mainThm} descends the following result.

\begin{theorem}\label{thm:mainThmHopf}
The following are equivalent for a bialgebra $B$:
\begin{enumerate}[label=(\arabic*), ref=\emph{(\arabic*)}, leftmargin=1cm, labelsep=0.3cm]
\item $B$ is a Hopf algebra;
\item $-\otimes B:\left(\Lmod{B},\otimes,\K\right)\to \left(\quasihopf{B},\tensor{B},B\right)$ is a monoidal equivalence of categories;
\item $-\otimes B:\Lmod{B}\to\quasihopf{B}$ is Frobenius;
\item $\sigma_M:\Hom{B}{}{B}{B}{B\otimes B}{M}\to \cl{M}, f\mapsto \cl{f(1\otimes 1)}$ is an isomorphism for all $M\in\quasihopf{B}$.
\end{enumerate}
\end{theorem}

\begin{personal}[explicit proof for the bialgebra case]


\section{Hopf algebras, Frobenius property and the free two-sided Hopf module functor}\label{sec:HopfBimod}

Let $\left( B,m,u,\Delta ,\varepsilon \right) $ be a $\K$-bialgebra over a commutative ring $\K$. It is straightforward to check that the category of left $B$-modules is not only monoidal, but in fact a (right) closed monoidal category with internal hom-functor $\Hom{B}{}{}{}{B\otimes N}{-}$ for all $N\in \Lmod{B}$ (a right-handed analogue of this result can be found in \cite[Lemma 3.5]{Saracco-Frobenius}). 

\begin{lemma}\label{lemma:modclosed2}
Let $B$ be a bialgebra. Then the category ${_B\M}$ of left $B$-modules is left and right-closed. Namely, we have bijections
\begin{gather}
\xymatrix{
\Hom{B}{}{}{}{M\otimes N}{P} \ar@<+0.5ex>[r]^-{\varphi} & \Hom{B}{}{}{}{M}{\Hom{B}{}{}{}{B\otimes N}{P}}  \ar@<+0.5ex>[l]^-{\psi},
} \label{eq:lmodclosed2}\\
\xymatrix{
\Hom{B}{}{}{}{N\otimes M}{P} \ar@<+0.5ex>[r]^-{\varphi'} & \Hom{B}{}{}{}{M}{\Hom{B}{}{}{}{N\otimes B}{P}}  \ar@<+0.5ex>[l]^-{\psi'},
}\notag
\end{gather}
natural in $M$ and $P$, given explicitly by
\begin{gather*}
\varphi(f)(m):a\otimes n \mapsto f(a\cdot m\otimes n), \qquad \psi(g):m\otimes n \mapsto g(m)(1\otimes n), \\
\varphi'(f)(m):n\otimes a \mapsto f(n\otimes a\cdot m), \qquad \psi'(g):n\otimes m \mapsto g(m)(n\otimes 1),
\end{gather*}
where the left $B$-module structures on $\Hom{B}{}{}{}{N\otimes B}{P}$ and $\Hom{B}{}{}{}{B\otimes N}{P}$ are induced by the right $B$-module structure on $B$ itself.
\end{lemma}

\begin{invisible}[Proof.]
Let $M,N,P$ be left $B$-modules. To make the exposition clearer, we will denote them via ${_\bullet M}$, ${_\bullet N}$ and ${_\bullet P}$ to underline the given actions. We are claiming that there is a bijection
\begin{equation*}
\xymatrix{
\Hom{B}{}{}{}{{_\bullet M}\otimes {_\bullet N}}{{_\bullet P}} \ar@<+0.5ex>[r]^-{\varphi} & \Hom{B}{}{}{}{{_\bullet M}}{\Hom{B}{}{}{}{{{_\bullet B}}\otimes {_\bullet N}}{{_\bullet P}}}   \ar@<+0.5ex>[l]^-{\psi}
}
\end{equation*}
natural in $M$ and $P$. Consider a generic $f\in \Hom{B}{}{}{}{{_\bullet M}\otimes {_\bullet N}}{{_\bullet P}}$. For all $m\in M$, $n\in N$ and $a,b,c\in B$ we have that
\begin{align*}
\big(\varphi(f)(c \cdot m)\big)\big(b\cdot(a\otimes n)\big) & = \sum f\Big(\big(b_1a\cdot (c\cdot m)\big)\otimes (b_2\cdot n)\Big) = f\Big(b\cdot \big((ac\cdot m)\otimes n\big)\Big) \\
 & =b \cdot \big(\varphi(f)(m)\big)\left(ac\otimes n\right)= b \cdot \big(c\cdot \varphi(f)(m)\big)\left(a\otimes n\right).
\end{align*}
Taking $c=1$ gives the left $B$-linearity of $\varphi(f)(m)$ while taking $b=1$ gives the left $B$-linearity of $\varphi(f)$, whence $\varphi$ is well-defined. On the other hand, for all $g\in \Hom{B}{}{}{}{{_\bullet M}}{\Hom{B}{}{}{}{{{_\bullet B}}\otimes {_\bullet N}}{{_\bullet P}}} $, $m\in M$, $n\in N$ and $a\in B$ we have also
\begin{align*}
\psi(g) & (a\cdot (m\otimes n)) = \sum g(a_1\cdot m)(1\otimes (a_2\cdot n)) = \sum \left(a_1\cdot g(m)\right)(1\otimes (a_2\cdot n)) \\
 &  = \sum g(m)(a_1\otimes (a_2\cdot n)) = g(m)(a\cdot (1\otimes n)) = a\cdot g(m)(1\otimes n) = a\cdot \psi(g)(m\otimes n),
\end{align*}
which implies that $\psi(g)$ is left $B$-linear and $\psi$ is well-defined as well. To check the naturality in $M$ and $P$ consider two left $B$-linear morphisms $h:M'\to M$ and $l:P\to P'$. Then
\begin{align*}
\Big(\big(\Hom{B}{}{}{}{B\otimes N}{l}\circ\varphi(f)\circ h\big)(m)\Big)(a\otimes n) & = \big(l\circ \varphi(f)(h(m))\big)(a\otimes n)= l\Big(\big(\varphi(f)(h(m))\big)(a\otimes n)\Big) \\
 & = l\big(f(a\cdot h(m)\otimes n)\big) = l\Big(f\big(h(a\cdot m)\otimes n\big)\Big) \\
 & = \Big(\varphi\big(l\circ f\circ (h\otimes N)\big)(m)\Big)(a\otimes n)
\end{align*}
To conclude, it is enough to check that $\varphi$ and $\psi$ are inverses each other. To this aim, we may compute directly
\begin{gather*}
\left(\varphi\psi(g)(m)\right)(a\otimes n) = \psi(g)(a\cdot m\otimes n) = g(a\cdot m)(1\otimes n) = g(m)(a\otimes n), \\
\left(\psi\varphi(f)\right)(m\otimes n) = \left(\varphi(f)(m)\right)(1\otimes n) = f(m\otimes n)
\end{gather*}
for all $m\in N$, $n\in N$, $a\in B$, $f\in \Hom{B}{}{}{}{{_\bullet M}\otimes {_\bullet N}}{{_\bullet P}}$ and $g\in \Hom{B}{}{}{}{{_\bullet M}}{\Hom{B}{}{}{}{{{_\bullet B}}\otimes {_\bullet N}}{{_\bullet P}}}$. Therefore, the first claim holds. The second claim can be proved analogously or may be deduced as follows. Consider the $\K$-modules $(N\otimes B)\tensor{B}M$ and $N\otimes (B\tensor{B}M)$ endowed with the $B$-actions
\begin{gather*}
a\cdot\big(\left(n\otimes b\right)\tensor{B}m\big) \coloneqq  \sum \left(\left(a_1\cdot n\right)\otimes a_2b\right)\tensor{B}m, \\
a\cdot\big(n\otimes \left(b\tensor{B}m\right)\big) \coloneqq  \sum \left(a_1\cdot n\right)\otimes \left(a_2b\tensor{B}m\right),
\end{gather*}
for all $a,b\in B$, $m\in M$ and $n\in N$. The canonical isomorphism $(N\otimes B)\tensor{B}M \cong N\otimes (B\tensor{B}M)$ (``the identity'') is a morphism in ${_B\M}$. Therefore, by the classical hom-tensor adjunction (see \cite[\S3]{Pareigis} for a very general approach), we have a chain of natural isomorphisms
\begin{align*}
\Hom{B}{}{}{}{{_\bullet N}\otimes {_\bullet M}}{{_\bullet P}} & \cong \Hom{B}{}{}{}{{_\bullet N}\otimes \left( {_\bullet B} \tensor{B} {M}\right)}{{_\bullet P}} \cong \Hom{B}{}{}{}{\left( {_\bullet N}\otimes {_\bullet B}\right) \tensor{B} {M}}{{_\bullet P}} \\
 & \cong \Hom{B}{}{}{}{{_\bullet M}}{\Hom{B}{}{}{}{{_\bullet N}\otimes {_\bullet B}}{_\bullet P}}
\end{align*}
whose composition gives exactly $\varphi'$ and $\psi'$.
\end{invisible}

Consider now the category $\quasihopf{B}=\left(\Bim{B}{B}\right)^B$, whose objects will be called \emph{two-sided Hopf modules} (see \cite[Definition 3.2]{Schauenburg-YDHopf}). As in \S\ref{sec:HopfMod}, we have an adjunction
\begin{equation*}
\xymatrix @R=15pt{
\quasihopf{B} \ar@<-0.5ex>[d]_-{\overline{(-)}} \\
\Lmod{B} \ar@<-0.5ex>[u]_-{-\otimes B}
}
\end{equation*}
where for every two-sided Hopf module $M$, $\overline{M}=M/MB^+$ is a $B$-module with $a\cdot \overline{m}=\overline{a\cdot m}$ and for every $B$-module $N$, $N\otimes B$ is a two-sided Hopf module with $a\cdot (n\otimes b) \cdot c = a_1\cdot n \otimes a_2bc$ and $\rho(n\otimes b) = \left(n \otimes b_1\right) \otimes b_2$ for all $m\in M, n\in N$ and $a,b,c\in B$.
Moreover, the bijection \eqref{eq:lmodclosed2} induces a bijection 
\begin{equation*}
\Hom{B}{}{B}{B}{{_\bullet M}\otimes {{}^{\phantom{\bullet}}_\bullet N_{\bullet}^{\bullet}}}{{{}^{\phantom{\bullet}}_\bullet P_{\bullet}^{\bullet}}}\cong \Hom{B}{}{}{}{{_\bullet M}}{\Hom{B}{}{B}{B}{{_\bullet B}\otimes {{}^{\phantom{\bullet}}_\bullet N_{\bullet}^{\bullet}}}{{{}^{\phantom{\bullet}}_\bullet P_{\bullet}^{\bullet}}}}
\end{equation*}
that makes of $\Hom{B}{}{B}{B}{B\otimes B}{-}$ the right adjoint of the functor $-\otimes B$.
Therefore we have another adjoint triple
\begin{equation}\label{eq:adjbimod2}
\overline{\left( -\right) } \ \dashv \ - \otimes B \ \dashv \ \Hom{B}{}{B}{B}{B\otimes B}{-},
\end{equation}
now between $\Lmod{B}$ and $\quasihopf{B}$, with units and counits given by
\begin{equation}\label{eq:unitscounitsquasi2}
\begin{gathered}
\eta_M:M\to \overline{M}\otimes B; \quad m\mapsto \overline{m_0}\otimes m_1, \quad \epsilon_N: \overline{(N\otimes B)}\stackrel{\cong}{\longrightarrow} N; \quad \overline{n\otimes a}\mapsto n\varepsilon(a), \\
\gamma_N: N \stackrel{\cong}{\longrightarrow} \Hom{B}{}{B}{B}{B\otimes B}{N\otimes B}; \quad n\mapsto \left[a\otimes b\mapsto a\cdot n\otimes b\right], \\
\theta_M:\Hom{B}{}{B}{B}{B\otimes B}{M}\otimes B \to M; \quad f\otimes a \mapsto f(1\otimes 1)\cdot a.
\end{gathered}
\end{equation}
What we are going to show now is that, differently from what happened in \S\ref{sec:HopfMod}, being Frobenius for the functor $-\otimes B:\Lmod{B}\to \quasihopf{B}$ is in fact equivalent to the existence of an ordinary antipode for $B$. As in \S\ref{sec:adjtriples}, we have the canonical morphism
\begin{equation*}
\sigma _{M}: \Hom{B}{}{B}{B}{B\otimes B}{M} \rightarrow \overline{M};\quad f\mapsto \overline{f\left(1\otimes 1\right) }.
\end{equation*}

\begin{remark}\label{rem:Hopfclassic2}
Two things deserve to be observed before proceeding.
\begin{enumerate}[label=(\arabic*), ref={(\arabic*)}, leftmargin=1cm, labelsep=0.3cm]
\item\label{item1:Hclassic2} Of course, $B$ admits an antipode $S$ if and only if one of the two adjunctions in \eqref{eq:adjbimod2} is an equivalence (whence both are). This provides a coassociative analogue of what is called the Structure Theorem for quasi-Hopf bimodules in the framework of quasi-bialgebras with preantipodes (see \cite{Saracco} for the one involving only the left-most adjunction and \cite{PhD} for the complete one). In particular, an inverse for $\sigma_M$ is given by $\sigma_M^{-1}(\cl{m})(x\otimes y) = x_1m_0S(x_2m_1)y$, for all $m\in M$, $x,y\in B$.
\item\label{item2:Hclassic2} Since $B \cong \K\otimes B \in\quasihopf{B}$, the component $\sigma_B:\Hom{B}{}{B}{B}{B\otimes B}{B}\to\cl{B}$ is always an isomorphism with inverse given by $\K\to\Hom{B}{}{B}{B}{B\otimes B}{B}, 1_\K\mapsto \left[x\otimes y\mapsto \varepsilon(x)y\right]$, in light of equation \eqref{eq:sigmaF2}.
\end{enumerate}
\end{remark}

\begin{invisible}
Checking \ref{item1:Hclassic2}: If $S$ exists, then $\eta_M^{-1}(\cl{m}\otimes b) = m_0S(m_1)b$. If the left-most adjunction in \eqref{eq:adjbimod2} is an equivalence, then the structure theorem for quasi-Hopf bimodules provides a linear endomorphism $S$ of $B$ such that $S(a_1b)a_2 = \varepsilon(a)S(b) = a_1S(ba_2)$ and $1S(1)1 = 1$, \ie an antipode. Finally $\cl{\sigma_M^{-1}(\cl{m})(1\otimes 1)} = \cl{m_0S(m_1)} = \cl{m}$ and
\[
x_1f(1\otimes 1)_0S(x_2f(1\otimes 1)_2)y = x_1f(1\otimes 1)S(x_2)y = f(x_1\otimes x_2S(x_3)y) = f(x\otimes y).
\]
\end{invisible}

Consider the distinguished two-sided Hopf module $B\cmptens{}B\coloneqq {B_{\bullet }^{\phantom{\bullet}}} \otimes { _{\bullet }^{\phantom{\bullet}}B_{\bullet }^{\bullet}} $ (also denoted simply by $\what{B}$). Both $\overline{B\cmptens{}B}$ and $\Hom{B}{}{B}{B}{B\otimes B}{B\cmptens{}B}$ turn out to be $B\otimes B$-modules via
\begin{equation}\label{eq:BBmodule2}
(a\otimes b)\triangleright \overline{x\otimes y} = \overline{ax\otimes by} \quad \text{and} \quad \big((a\otimes b)\triangleright f\big)(x\otimes y) = (a\otimes 1)f(xb\otimes y)
\end{equation}
respectively, for $a,b,x,y\in B$ and $f\in \Hom{B}{}{B}{B}{B\otimes B}{B\cmptens{}B}$, and the associated component $\sigma_{\what{B}}$ is $B\otimes B$-linear. 
\begin{invisible}
Perhaps, the only non immediate check is that $(a\otimes 1)\triangleright f$ is a well-defined element of $\Hom{B}{}{B}{B}{B\otimes B}{B\cmptens{}B} $ for all $a\in B$ and $f\in \Hom{B}{}{B}{B}{B\otimes B}{B\cmptens{}B}$. To this aim, observe that
\begin{align*}
\left( a\triangleright f\right) \left( b_{1}x\otimes b_{2}yc\right) & = \left( a\otimes 1\right) f\left( b_{1}x\otimes b_{2}yc\right) = \left( a\otimes 1\right) \left( 1\otimes b\right) f\left( x\otimes y\right) \Delta \left( c\right) \\
 & = \left( 1\otimes b\right) \left[ \left( a\otimes 1\right) f\left( x\otimes y\right) \right] \Delta \left( c\right) = b\cdot \left( a\triangleright f\right) \left( x\otimes y\right) \cdot c
\end{align*}
and
\begin{align*}
\big[  & \left( a\triangleright f\right) \left( x\otimes y\right) \big] _{0} \otimes \big[ \left( a\triangleright f\right) \left( x\otimes y\right) \big] _{1} \\
 & = \left[ \left( a\otimes 1\right) f\left( x\otimes y\right) \right] _{0}\otimes \left[ \left( a\otimes 1\right) f\left( x\otimes y\right) \right] _{1} = \left( A\otimes \Delta \right) \left( \left( a\otimes 1\right) f\left( x\otimes y\right) \right) \\
 & = \left( a\otimes 1\otimes 1\right) \left[ \left( A\otimes \Delta \right) \left( f\left( x\otimes y\right) \right) \right] = \left( a\otimes 1\right) f\left( x\otimes y\right) _{0}\otimes f\left( x\otimes y\right) _{1} \\
 & = \left( a\otimes 1\right) f\left( \left( x\otimes y\right)_{0}\right) \otimes \left( x\otimes y\right) _{1} = \left( a\triangleright f\right) \left( \left( x\otimes y\right) _{0}\right) \otimes \left( x\otimes y\right) _{1},
\end{align*}
as expected.
\end{invisible}

\begin{remark}\label{rem:tau2}
Let $N$ be any right $B$-module and $N\cmptens{}B$ the two-sided Hopf module ${N_\bullet^{\phantom{\bullet}}}\otimes \quasihopfmod{B}$. By right $B$-linearity, every $f\in \Hom{B}{}{B}{B}{B\otimes B}{N\cmptens{}B}$ is uniquely determined by $f(a\otimes 1)$ for $a\in B$ and, moreover, right $B$-colinearity implies that $f(a\otimes 1)\otimes 1 = (N\otimes \Delta)\left(f(a\otimes 1)\right)$.
Set $\tau_f(a)\coloneqq (N\otimes \varepsilon)f(a\otimes 1)$. This defines a linear morphism $\tau_f:B\to N$ such that
\begin{equation}\label{eq:ftau2}
f(a\otimes b) = \tau_f(a)b_1\otimes b_2
\end{equation}
for all $a,b\in B$. Left $B$-linearity now entails $\tau_f(b)\otimes a = \tau_f(a_1b)a_2\otimes a_3$ and hence
\begin{equation}\label{eq:quasiantipode2}
\tau_f(a_1b)a_2 = \varepsilon(a)\tau_f(b).
\end{equation}
Denote by ${^\star\Homk({B},{N})}$ the family of $\K$-linear morphisms $g:B\to N$ that satisfies \eqref{eq:quasiantipode2}. Then we have an isomorphism of left $B$-modules
\begin{equation}\label{eq:tau2}
\begin{gathered}
\xymatrix @R=0pt {
{\tau:\Hom{B}{}{B}{B}{B\otimes B}{N\cmptens{}B}} \ar@{<->}[r] & {^\star\Homk({B},{N})} \\
{\phantom{\tau:}f} \ar@{|->}[r] & \tau_f \\
{\phantom{\tau:}\big[a\otimes b\mapsto g(a)b_1\otimes b_2\big]} & g \ar@{|->}[l]
}
\end{gathered}
\end{equation}
where the $B$-module structure on ${^\star\Homk({B},{N})}$ is the one induced by ${\Homk({B},{N})}$, that is, $(a\triangleright g)(x)=g(xa)$ for all $a,x\in B$ and $g\in {\Homk({B},{N})}$.
The isomorphism $\tau$ fits into a commutative diagram of $\K$-modules
\begin{equation}\label{eq:bigdiagram2}
\begin{gathered}
\xymatrix{
\Hom{B}{}{B}{B}{B\otimes B}{N\cmptens{}B} \ar@{<->}[r]^-{\tau} \ar[d]_-{\sigma_{\what{N}}} & {^\star\Homk({B},{N})} \ar[d]^-{\mathsf{ev}_1} \\
\overline{N\cmptens{} B} & N \ar[l]_-{\jmath}
}
\end{gathered}
\end{equation}
where $\jmath(n)=\overline{n\otimes 1}$ and $\mathsf{ev}_1(g) = g(1)$. In the particular case of $N=B$, relation \eqref{eq:quasiantipode2} gives $\tau_f*\id_B = \tau_f(1) u \circ \varepsilon$ and $\tau$ of \eqref{eq:tau2} becomes an isomorphism of $B\otimes B$-modules, where the $B\otimes B$-module structure on ${^\star\End{\K}{B}}$ is the one induced by $\End{\K}{B}$, that is, $((a\otimes b)\triangleright g)(x)=ag(xb)$ for all $a,b,x\in B$ and $g\in \End{\K}{B}$.
Moreover, diagram \eqref{eq:bigdiagram2} becomes a diagram of left $B$-modules, where the left $B$-module structure on the top row is the one induced by $B\to B\otimes B, a\mapsto a\otimes 1$.
\begin{invisible}
Then we have an isomorphism of $B\otimes B$-modules
\begin{equation}\label{eq:tau2}
\begin{gathered}
\xymatrix @R=0pt {
{\tau:\Hom{B}{}{B}{B}{B\otimes B}{B\cmptens{}B}} \ar@{<->}[r] & {^\star\End{\K}{B}} \\
{\phantom{\tau:}f} \ar@{|->}[r] & \tau_f \\
{\phantom{\tau:}\big[a\otimes b\mapsto g(a)b_1\otimes b_2\big]} & g \ar@{|->}[l]
}
\end{gathered}
\end{equation}
where the $B\otimes B$-module structure on ${^\star\End{\K}{B}}$ is the one induced by $\End{\K}{B}$, \ie for all $a,b,x\in B$ and $g\in \End{\K}{B}$
\begin{equation*}
\big((a\otimes b)\triangleright g\big)(x)=ag(xb).
\end{equation*}
The latter isomorphism fits into a commutative diagram of left $B$-modules
\begin{equation*}
\begin{gathered}
\xymatrix{
\Hom{B}{}{B}{B}{B\otimes B}{B\cmptens{}B} \ar@{<->}[r]^-{\tau} \ar[d]_-{\sigma_{\what{B}}} & {^\star\End{\K}{B}} \ar[d]^-{\mathsf{ev}_1} \\
\overline{B\cmptens{} B} & B \ar[l]_-{\jmath}
}
\end{gathered}
\end{equation*}
where $\jmath(a)=\overline{a\otimes 1}$, $\mathsf{ev}_1(g) = g(1)$ and the left $B$-module structure on the top row is the one induced by $B\to B\otimes B, a\mapsto a\otimes 1$.
\end{invisible}
\end{remark}

Let us keep the notation introduced in Remark \ref{rem:tau2} and assume that
\begin{equation*}
\sigma_{\widehat{B }}: \Hom{B}{}{B}{B}{B\otimes B}{B\cmptens{}B} \rightarrow \overline{B\cmptens{}B}, \ f\mapsto \overline{\tau_f(1)\otimes 1},
\end{equation*}
is an isomorphism.

\begin{proposition}\label{prop:stillcomplete2}
If $\sigma_{\widehat{B }}$ is invertible, then the linear endomorphism $s=\tau\left(\sigma_{\widehat{B }}^{-1}\left( \overline{1\otimes 1}\right)\right)$ of $B$ satisfies, for all $a,b\in B$,
\begin{equation}\label{eq:convinv2}
a_1s(ba_2) = \varepsilon(a)s(b) = s(a_1b)a_2.
\end{equation}
\end{proposition}

\begin{proof}
Recall from Remark \ref{rem:tau2} that $s(a) = (B\otimes \varepsilon)\left(\sigma_{\widehat{B }}^{-1}\left( \overline{1\otimes 1}\right)(a\otimes 1)\right)$ for all $a\in A$.
The $B\otimes B$-linearity of $\sigma_{\widehat{B }}$ entails $\sigma_{\widehat{B }}^{-1}( \overline{a\otimes b}) = (a\otimes b)\triangleright \sigma_{\widehat{B} }^{-1}( \overline{1\otimes 1}) $ and so
\begin{equation}\label{eq:sigmaBBlin2}
\sigma_{\widehat{B} }^{-1}\left( \overline{a\otimes b}\right) \left(x\otimes y\right) = \left( a\otimes 1\right) \left( \sigma_{\widehat{B }}^{-1}\left( \overline{1\otimes 1}\right) \left( xb\otimes y\right)\right)
\end{equation}
for all $a,b,x,y\in B$. In particular,
\begin{equation}\label{eq:sigmabalbialg2}
\sigma_{\widehat{B }}^{-1}\left( \overline{1\otimes a}\right) \left( 1\otimes 1\right) =\sigma_{\widehat{B }}^{-1}\left( \overline{1\otimes 1} \right) \left( a\otimes 1\right).
\end{equation}
Now, by rewriting equation \eqref{eq:ftau2} for $\sigma_{\widehat{B }}^{-1}\left( \overline{1\otimes 1}\right)$ we get
\begin{equation} \label{eq:sigmainvsmallS32}
\sigma_{\widehat{B} }^{-1}\left( \overline{a\otimes b}\right) \left(x\otimes y\right) \stackrel{\eqref{eq:sigmaBBlin2}}{=} \left( a\otimes 1\right) \sigma_{\widehat{B }}^{-1}\left( \overline{1\otimes 1}\right) \left( xb\otimes y\right) = as\left( xb\right) y_{1}\otimes y_{2}
\end{equation}
and the right-most equality in \eqref{eq:convinv2} is \eqref{eq:quasiantipode2}.
Moreover, by definition of $\overline{B \cmptens{} B}$,
\begin{equation*}
\varepsilon \left( a\right) \left( \sigma_{\widehat{B} }^{-1}\left( \overline{1\otimes b}\right) \left( 1\otimes 1\right) \right)  =\sigma_{\widehat{B} }^{-1}\left( \overline{a_{1}\otimes ba_{2}}\right) \left( 1\otimes 1\right) =a_{1}s\left( ba_{2}\right) \otimes 1
\end{equation*}
from which it follows that $\varepsilon \left( a\right) s\left( b\right) =a_{1}s\left( ba_{2}\right) $ for all $a,b\in B$.
\end{proof}

\begin{remark}\label{rem:sigmaepsi2}
Since $s\left( b\right) \otimes 1=\sigma_{\widehat{B } }^{-1}( \overline{1\otimes b}) \left( 1\otimes 1\right)$ for all $b\in B$, we have that
\begin{equation}\label{eq:sigmainvsmallS42}
\overline{s\left( b\right) \otimes 1}=\overline{\sigma_{\widehat{B } }^{-1}\left( \overline{1\otimes b}\right) \left( 1\otimes 1\right) }=\sigma_{\widehat{B } }\left(\sigma_{\widehat{B } }^{-1}\left( \overline{1\otimes b}\right)\right) = \overline{1\otimes b}.
\end{equation}
Thus $s\left( b\right) -\varepsilon \left( b\right) 1\in \ker\left(\varepsilon \right) $ and so $\varepsilon \circ s=\varepsilon $. 
\end{remark}

It is self-evident now that $s$ is an antipode if and only if $s(1)=1$. The forthcoming results are all intermediate steps toward the proof of this latter identity.

\begin{lemma}\label{lemma:cltensor2}
For $M\in\Lmod{B}$ and $N\in\Bim{B}{B}$ we have $\cl{M\otimes N}\cong M\otimes \cl{N}$ in $\Lmod{B}$.
\end{lemma}

\begin{proof}
Since $\cl{N}=N/NB^+$ and $B/B^+\cong \K$, the thesis follows from the isomorphisms
\[
\cl{M\otimes N} \cong (M\otimes N)\tensor{B}\K\cong M\otimes (N\tensor{B}\K)\cong M\otimes \cl{N}. \qedhere
\]
\end{proof}

\begin{lemma}
If $\sigma$ is a natural isomorphism, then for any $M\in\quasihopf{B}$, $m\in M$ and $x,y\in B$,
\begin{equation}\label{eq:sigmainv2}
\sigma_M^{-1}(\cl{m})(x\otimes y) = x_1m_0s(x_2m_1)y.
\end{equation}
\end{lemma}

\begin{proof}
Set $B_2 \coloneqq  \lmod{B}\otimes \rmod{B} \otimes \quasihopfmod{B} \in \quasihopf{B}$ and denote by $\varphi: B\otimes \cl{B\cmptens{}B} \to \cl{B_2}$ the isomorphism of Lemma \ref{lemma:cltensor2}. Consider also the left $B$-linear morphism
\begin{gather}\label{eq:homtensor2}
\begin{gathered}
\psi: B \otimes \Hom{B}{}{B}{B}{B\otimes B}{B\cmptens{}B} \to \Hom{B}{}{B}{B}{B\otimes B}{B_2} \\
a\otimes f \longmapsto \left[x\otimes y\mapsto x_1a\otimes f(x_2\otimes y)\right].
\end{gathered}
\end{gather}
\begin{invisible}
It is well-defined since 
\[
\psi(a\otimes f)(b_1x\otimes b_2yc) = b_1x_1a\otimes f(b_2x_2\otimes b_3yc) = b_1x_1a\otimes (1\otimes b_2)f(x_2\otimes y)\Delta(c) = b(\psi(a\otimes f)(x\otimes y))c
\]
and 
\[
\delta(\psi(a\otimes f)(x\otimes y)) = x_1a\otimes (B\otimes \Delta)f(x_2\otimes y) = x_1a\otimes f(x_2\otimes y_1)\otimes y_2 = \psi(a\otimes f)(x\otimes y_1)\otimes y_2.
\]
$B$-linearity follows from
\[
\psi(b_1a\otimes b_2f)(x\otimes y) = x_1b_1a\otimes (b_2f)(x_2\otimes y) = x_1b_1a\otimes f(x_2b_2\otimes y) = \psi(a\otimes f)(xb\otimes y) = (b\psi(a\otimes f))(x\otimes y).
\]
\end{invisible}
Then the following diagram is easily seen to be commutative
\[
\xymatrix @C=40pt @R=15pt {
B \otimes \Hom{B}{}{B}{B}{B\otimes B}{B\cmptens{}B} \ar[r]^-{B\otimes \sigma_{\what{B}}} \ar[d]_-{\psi} & B \otimes \cl{B\cmptens{}B} \ar[d]^-{\varphi} \\
\Hom{B}{}{B}{B}{B\otimes B}{B_2} \ar[r]_-{\sigma_{B_2}} & \cl{B_2}
}
\]
and so, in light of \eqref{eq:sigmainvsmallS32}, for all $a,b,c,x,y\in B$ we have
\begin{equation}\label{eq:sigmaB22}
\sigma_{B_2}^{-1}\left(\cl{a\otimes b\otimes c}\right)(x\otimes y) = x_1a\otimes bs(x_2c)y_1\otimes y_2.
\end{equation}
Now, for any $N\in\Bim{B}{B}$ consider $N\otimes B = \bimod{N}\otimes \quasihopfmod{B} \in\quasihopf{B}$. For every $n\in N$, the assignment $\tilde{n}:\lmod{B}\otimes \rmod{B}\to N, a\otimes b\mapsto anb$, is a well-defined $B$-bilinear morphism. Naturality of $\sigma$ implies then that
\begin{gather}\label{eq:sigmaN2}
\begin{gathered}
\sigma_{N\otimes B}^{-1}(\cl{n\otimes b})(x\otimes y) = \left(\tilde{n}\otimes B\right)\left(\sigma_{B_2}^{-1}\left(\cl{1\otimes 1\otimes b}\right)(x\otimes y)\right) \\
 \stackrel{\eqref{eq:sigmaB22}}{=} \left(\tilde{n}\otimes B\right)\left(x_1\otimes s(x_2b)y_1\otimes y_2\right) = x_1ns(x_2b)y_1\otimes y_2
\end{gathered}
\end{gather}
for all $n\in N$, $b,x,y\in B$. Finally, the coaction $\delta_M:\quasihopfmod{M}\to \bimod{M}\otimes \quasihopfmod{B}$ is a well-defined morphism in $\quasihopf{B}$ and hence we may resort again to naturality of $\sigma$ to get that
\begin{gather*}
\delta_M\left(\sigma_M^{-1}(\cl{m})(x\otimes y)\right) = \sigma_{M\otimes B}^{-1}(\cl{m_0\otimes m_1})(x\otimes y) \stackrel{\eqref{eq:sigmaN2}}{=} x_1m_0s(x_2m_1)y_1\otimes y_2
\end{gather*}
for all $m\in M$, $x,y\in B$. Applying $M\otimes \varepsilon$ to both sides gives the result.
\end{proof}

\begin{lemma}
If $\sigma$ is a natural isomorphism, then $s(1)=1$.
\end{lemma}

\begin{proof}
For every $M\in\quasihopf{B}$ and for all $m\in M$, $x,y\in B$, colinearity of $\sigma_{M}^{-1}(\cl{m})$ and \eqref{eq:sigmainv2} entail that
$
x_1m_0s(x_3m_2)_1y_1\otimes x_2m_1s(x_3m_2)_2y_2 = \delta_M(\sigma_M^{-1}(\cl{m})(x\otimes y)) = x_1m_0s(x_2m_1)y_1\otimes y_2.
$
By rewriting it for $M=B$ and all $m,x,y$ equal $1$ one gets $s(1)_1\otimes s(1)_2 = s(1)\otimes 1$ and hence $s(1)=1$ by Remark \ref{rem:sigmaepsi2}.
\end{proof}

Summing up, we have proved the following central result.

\begin{theorem}\label{thm:mainThm2}
A bialgebra $B$ is a Hopf algebra if and only if the functor $-\otimes B$ is Frobenius.
\end{theorem}

\begin{remark}
At the present moment it is not clear to us if there exists a distinguished two-sided Hopf module $M$ such that $\sigma$ is a natural isomorphism if and only if $\sigma_M$ is an isomorphism. Our candidate would have been $B\cmptens{}B$, but proving that $s(1)=1$ out of the invertibility of $\sigma_{\what{B}}$ seems more involved than expected. By running through the whole proof of Theorem \ref{thm:mainThm2} again, one may observe that a necessary and sufficient condition for having that $s$ is an antipode would have been the invertibility of $\sigma_{\what{B}}$, $\sigma_{B_2}$ and $\sigma_{B\otimes B}$. Indeed, by \ref{item2:Hclassic2} of Remark \ref{rem:Hopfclassic2} and the relations we found in the aforementioned proof,
\begin{align*}
1\otimes 1 & = \Delta\left(\sigma_B^{-1}(\cl{1})(1\otimes 1)\right) = \sigma_{B\otimes B}^{-1}(\cl{1\otimes 1})(1\otimes 1) = (m\otimes B)\left(\sigma_{B_2}^{-1}(\cl{1\otimes 1\otimes 1})(1\otimes 1)\right) \\
 & = (m\otimes B)\left(1\otimes s(1)\otimes 1\right) = s(1)\otimes 1,
\end{align*}
where $m$ here stands again for the multiplication of $B$. Nevertheless, we don't know if there is some redundancy between them or not.
\end{remark}
\end{personal}

\begin{invisible}
\begin{remark}
As the referee correctly suggested, by looking at Theorems \ref{thm:mainThm} and \ref{thm:mainThmHopf} the reader may be curious to see if and how our results so far are connected with those obtained by D.\t Bulacu, S.\t Caenepeel and B.\t Torrecillas in \cite[\S 2.1 and \S3.2]{BulacuCaenepeelTorrecillas2}, in particular Theorems 2.4 and 3.4(iv) therein. The main reason why we see no direct correlation between them is that the above-mentioned results in \cite{BulacuCaenepeelTorrecillas2} are particular instances of the more general \cite[Theorem 5.8]{BulacuCaenepeelTorrecillas1}, which states a necessary and sufficient condition for a certain forgetful functor (always admitting a right adjoint) $F : \cC(\psi)_A^X \to \cC_A$ to be Frobenius in terms of the Frobenius property for the coalgebra $(X,\psi)$ (we refer to the original paper for notations and terminology). In our context, $\cl{(-)}$ cannot be seen as a forgetful functor in general, because it is not faithful. For instance, consider the monoid bialgebra $B:=\C[\Z_4]$ with group-like comultiplication, where the monoid structure on $\Z_4$ is given by the product. For the two-sided Hopf module $M := {_\bullet B^\bullet \otimes B_\bullet^\bullet}$ we have that $\eta_M: {_\bullet B^\bullet \otimes B_\bullet^\bullet} \to {\lmod{B} \otimes \quasihopfmod{B}}, a\otimes b\mapsto a_1\otimes a_2b,$ is not injective ($\eta_M(\cl{2}\otimes\cl{2}) = \eta_M(\cl{2}\otimes\cl{0})$), whence $\cl{(-)}$ cannot be faithful in light of (the dual version of) \cite[Theorem IV.3.1]{MacLane}. The lack of correlation is also supported by the fact that, in \cite{BulacuCaenepeelTorrecillas2}, being Frobenius for $F$ is always connected with $C$ being finite-dimensional, while in our framework the dimension of $A$ plays no role.

A bit more in detail: as mentioned before \cite[Theorem 2.4]{BulacuCaenepeelTorrecillas2}, a quasi-bialgebra $A$ can be seen as a right $A$-comodule algebra in the sense of \cite[Definition 3.3]{HausserNill}, whence we may consider the category of relative Hopf modules $\cM_A^C$ (as in the introduction to \cite[\S2]{BulacuCaenepeelTorrecillas2}) and the forgetful functor $F : \cM_A^C \to \cM_A$, where $C$ is a right $A$-module coalgebra (a coalgebra in the monoidal category of right $A$-modules). By \cite[Theorem 2.4]{BulacuCaenepeelTorrecillas2}, $F$ is a Frobenius functor if and only if $C$ is a Frobenius coalgebra in the category of right $A$-modules. This is not the same context as ours because, in general, $(A,\Delta,\varepsilon)$ is not a right $A$-module coalgebra. On the other hand, $(A,\Delta,\varepsilon)$ is a right $A\otimes \op{A}$-module coalgebra and, in fact, we may consider the category $\cM(A\otimes\op{A})_{A\otimes \op{A}}^A = \quasihopf{A}$ as in \cite[Theorem 3.4(iv)]{BulacuCaenepeelTorrecillas2}. However, in this case, the forgetful functor of \cite{BulacuCaenepeelTorrecillas2} is $F:\quasihopf{A}\to \Bim{A}{A}$, which is not the functor we are interested in so far. We will come back on this at the end of \S\ref{ssec:previous}, when the latter functor will play a role in the study of the Frobenius property for $A$ as a $\K$-algebra. 
\end{remark}
\end{invisible}

\begin{example}
This ``toy example'' is intended to show, in an easy-handled context, some of the facts and the computations presented so far. We point out that it already appeared in this setting in \cite[Example 1]{Saracco} and previously in \cite[Preliminaries 2.3]{EtingofGelaki}. Let $G\coloneqq \langle g\rangle$ be the cyclic group of order 2 with generator $g$ and let $\K$ be a field of characteristic different from 2. Consider the group algebra $A\coloneqq \K G$, which is a commutative algebra of dimension 2. An $A$-bimodule is a $\K$-vector space $V$ endowed with two distinguished commuting automorphisms $\alpha,\beta$ such that $\alpha^2=\id_V=\beta^2$ (which are left and right action by $g$ respectively). Consider the distinguished elements $t\coloneqq \frac{1}{2}(1+g)$ (total integral in $A$) and $p=\frac{1}{2}(1-g)$. They form a pair of pairwise orthogonal idempotents and $A \cong \K t \oplus \K p$ as $\K$-algebras. Moreover, with respect to this new basis, $g=t-p$ and $1=t+p$. 

Now, endow $A$ with the group-like comultiplication $\Delta(g)=g\otimes g$ and counit $\varepsilon(g)=1$ and consider the element
\[
\Phi \coloneqq  1\otimes 1\otimes 1 - 2p\otimes p\otimes p.
\]
This is invertible (with inverse itself) and it satisfies the conditions \eqref{eq:PhiCocycle}, \eqref{eq:PhiCounital} and \eqref{eq:quasi-coassociativity}. These make of $A$ a genuine quasi-bialgebra (with $A^+=\K p$), so that the category of $A$-bimodules is now a monoidal category. Observe that
\begin{gather*}
\Delta(t) = \frac{1}{2}1\otimes 1 + \frac{1}{2}g \otimes g = t \otimes g + 1 \otimes p = t \otimes t + p \otimes p, \\
\Delta(p) = \frac{1}{2}1\otimes 1 - \frac{1}{2}g \otimes g = p \otimes g + 1 \otimes p = p \otimes t + t \otimes p.
\end{gather*}

A bimodule $M$ is a quasi-Hopf bimodule if it comes endowed with an $A$-bilinear coassociative and counital $A$-coaction in $\Bim{A}{A}$. For every $m\in M$, write $\delta(m)\coloneqq m_1\otimes t + m_2\otimes p$. The counitality condition already implies that $m_1 = m$, so that we may write $\delta(m) = m\otimes t + m'\otimes p$ and $\delta(m') = m'\otimes t + m''\otimes p$. Concerning the coassociativity condition, compute
\begin{gather*}
\Phi\cdot (\delta\otimes A)(\delta(m)) = \Phi \cdot \left( m \otimes t \otimes t + m' \otimes p \otimes t + m' \otimes t \otimes p + m'' \otimes p \otimes p\right) \\
 = m \otimes t \otimes t + m' \otimes p \otimes t + m' \otimes t \otimes p + m'' \otimes p \otimes p - 2p m'' \otimes p \otimes p,
\end{gather*}
\begin{gather*}
(M\otimes \Delta)(\delta(m)) \cdot \Phi = \left( m \otimes t \otimes t + m\otimes p \otimes p + m'\otimes p \otimes t + m'\otimes t \otimes p \right) \cdot \Phi \\
 = m\otimes t \otimes t + m\otimes p \otimes p + m'\otimes p \otimes t + m'\otimes t \otimes p - 2m p\otimes p \otimes p.
\end{gather*}
By equating the right-most terms we find
\[
(1-2p) m'' \otimes p \otimes p = m'' \otimes p \otimes p - 2p m'' \otimes p \otimes p = m\otimes p \otimes p - 2m p\otimes p \otimes p = m  (1-2p) \otimes p \otimes p
\]
so that $gm'' = (1-2p) m'' = m  (1-2p) = m  g$ and hence $m'' = g m g$. This allows us to define a $\K$-linear automorphism $\nu:M\to M, m \mapsto m',$ which satisfies $\nu^2(m) = g m g$ and, since
\[
\begin{gathered}
\delta(g m) = (g\otimes g)\delta(m) = (g\otimes g)(m\otimes t + m'\otimes p) = g m\otimes t - g m'\otimes p, \\
\delta(m g) = (m\otimes t + m'\otimes p)(g\otimes g) = m g\otimes t - m' g\otimes p,
\end{gathered}
\]
it satisfies $\nu(g m) = -g  \nu(m)$ and $\nu(m g) = -\nu(m) g$ as well.
In particular, $\nu(m  t) = \nu(m)  p$ and $\nu(m  p) = \nu(m)  t$ for all $m\in M$. Thus, a quasi-Hopf bimodule over $A$ is essentially a vector space with three distinguished automorphisms $\alpha,\beta,\nu$ such that $\alpha^2=\id_V=\beta^2$, $\alpha\circ \beta = \beta\circ\alpha = \nu^2$, $\nu\circ\alpha = - \alpha\circ \nu$ and $\nu\circ \beta = -\beta\circ\nu$. 

Let $M\in\quasihopf{A}$ and pick $f\in\Hom{A}{}{A}{A}{A\otimes A}{M}$. As we have seen, such an $f$ is uniquely determined by $T_f:A\to M$ satisfying \eqref{eq:deltaf} and \eqref{eq:Tpreant}. In particular, since from \eqref{eq:Tpreant} it follows that $T_f(g)=T_f(g) g^2 = g T_f(1) g$, $f$ is uniquely determined by an element $\omega\coloneqq T_f(1)$ satisfying \eqref{eq:deltaf}. If we compute first $2T_f(p) = T_f(1-g) = \omega - g \omega  g$, then \eqref{eq:deltaf} becomes
\[
\begin{gathered}
\omega\otimes t + \nu(\omega)\otimes p = \delta(\omega) \stackrel{\eqref{eq:deltaf}}{=} T_f(1)\otimes 1 - 2T_f(p) p\otimes p = \omega \otimes 1 - \omega p \otimes p + g\omega gp \otimes p \\
 = \omega \otimes t + \omega \otimes p - \omega p \otimes p - g\omega p \otimes p= \omega \otimes t + (\omega -\omega p - g \omega p) \otimes p. 
\end{gathered}
\]
Thus, $\nu(\omega) = \omega - \omega p - g \omega p = \omega - 2t \omega p$. As a consequence, observe that $\nu(\omega) p = \omega p - 2t \omega p = -g \omega p$ and so $\omega p = - g \nu(\omega) p$. Therefore,
$
\omega = \omega t + \omega p = \omega t - g \nu(\omega) p = \omega t + \nu(g \omega) p = \omega t + \nu(g \omega t)
$ and $\omega$ is uniquely determined by $\omega t$. 
The converse is true as well: if $\omega\in M$ satisfies $\omega = \omega t + \nu(g \omega) p$ then the morphism $f_\omega:A\to M$ given by $f_\omega(a) = f_\omega(a_e1+a_gg) \coloneqq  a_e\omega +a_gg\omega g$ satisfies \eqref{eq:deltaf} and \eqref{eq:Tpreant}.\begin{invisible}If $^\dag M$ are the elements of the form $T_f(1)$, then the bijection with $Mt$ is given by $^\dag M\to Mt, m \mapsto mt$ and $Mt\to ^\dag M, mt \mapsto mt + \nu(gmt)$.\end{invisible}
This means that $T_f(1)$ is uniquely determined by its image via the projection $M\to Mt,m\mapsto mt$, which in turn induces an isomorphism of left $A$-modules $\cl{M}\cong Mt$.\begin{invisible}Since $m=m 1 = m t+ m  p$, if $m t=0$ then $m=m p\in Mp=MA^+$.\end{invisible}
Summing up, the existence of the bijective correspondence $T:\Hom{A}{}{A}{A}{A\otimes A}{M} \to Mt, f \mapsto T_f(1)t,$ shows that the canonical map $\sigma_M: \Hom{A}{}{A}{A}{A\otimes A}{M} \to \cl{M}$ is an isomorphism for every $M\in\quasihopf{A}$. Note that we explicitly have 
\begin{gather}\label{eq:T-1}
\begin{gathered}
T^{-1}(mt)(1\otimes 1) = mt+\nu(gmt) = mt-g\nu(m)p, \\
T^{-1}(mt)(g\otimes 1) = g(mt+\nu(gmt))g = gmt + \nu(mt).
\end{gathered}
\end{gather}
To see how the preantipode looks like, first we compute $\nu$ for $A\cmptens{}A$. Since the coaction on the elements of the basis behave as follows:
\begin{gather*}
1\otimes 1\mapsto 1\otimes 1\otimes 1 -2p\otimes p \otimes p = (1\otimes 1) \otimes t + (1\otimes 1 - 2p\otimes p)\otimes p, \\
g\otimes 1\mapsto g\otimes 1\otimes 1 -2gp\otimes p \otimes p = (g\otimes 1) \otimes t + (g\otimes 1 + 2p\otimes p)\otimes p, \\
1\otimes g\mapsto 1\otimes g\otimes g -2p\otimes gp \otimes gp = (1\otimes g) \otimes t - (1\otimes g + 2p\otimes p)\otimes p, \\
g\otimes g\mapsto g\otimes g\otimes g -2gp\otimes gp \otimes gp = (g\otimes g) \otimes t - (g\otimes g - 2p\otimes p)\otimes p,
\end{gather*}
we see that, by definition of $\nu:M\to M, m \mapsto m'$,
\begin{gather*}
\nu(1\otimes 1) = 1\otimes 1 - 2p\otimes p, \qquad \nu(g\otimes 1) =  g\otimes 1 + 2p\otimes p, \\
\nu(1\otimes g) = -1\otimes g - 2p\otimes p, \qquad \nu(g\otimes g) = -g\otimes g + 2p\otimes p.
\end{gather*}
Thus, by resorting to \eqref{eq:T-1} and by writing $a=a_e1+a_gg$ for all $a\in A$, we find out that
\begin{gather*}
S(a) \stackrel{\eqref{eq:S}}{=} (A\otimes \varepsilon)\left(\sigma_{\what{A}}^{-1}\left(\cl{1\otimes 1}\right)(a\otimes 1)\right) = (A\otimes \varepsilon)\left(T^{-1}\left((1\otimes 1)\cdot t\right)(a\otimes 1)\right) \\
= a_e(A\otimes \varepsilon)\left(T^{-1}\left((1\otimes 1)\cdot t\right)(1\otimes 1)\right) + a_g(A\otimes \varepsilon)\left(T^{-1}\left((1\otimes 1)\cdot t\right)(g\otimes 1)\right) \\
\stackrel{\eqref{eq:T-1}}{=} a_e(A\otimes \varepsilon)\left(t_1\otimes t_2 - g(1\otimes 1 - 2p\otimes p)p\right) + a_g(A\otimes \varepsilon)\left(g(t_1\otimes t_2) + (1\otimes 1-2p\otimes p)p\right) \\
= a_e(A \hspace{-1pt} \otimes \hspace{-1pt}  \varepsilon)\left(t_1 \hspace{-1pt} \otimes \hspace{-1pt}  t_2 - p_1 \hspace{-1pt} \otimes \hspace{-1pt}  gp_2 - 2pp_1 \hspace{-1pt} \otimes \hspace{-1pt}  gpp_2\right) + a_g(A \hspace{-1pt} \otimes \hspace{-1pt}  \varepsilon)\left(t_1 \hspace{-1pt} \otimes \hspace{-1pt}  gt_2 + p_1 \hspace{-1pt} \otimes \hspace{-1pt}  p_2-2pp_1 \hspace{-1pt} \otimes  \hspace{-1pt} pp_2\right) \\
= a_e(t-p) + a_g(t+p) = a_eg+a_g = ag,
\end{gather*}
for every $a\in A$, which coincides with the one constructed in \cite[Example 1]{Saracco} as expected.
\end{example}


\subsection{Connections with one-sided Hopf modules and Hopf algebras}\label{ssec:previous}

Given a bialgebra $B$, one can consider its category of (right) Hopf modules $\rhopf{B}$ and we always have an adjoint triple $\inv{(-)}{B} \dashv -\otimes B \dashv \coinv{(-)}{B}$ between $\M$ and $\rhopf{B}$, where $\inv{M}{B} = M/MB^+$ and $\coinv{M}{B} = \left\{m\in M \mid \delta(m) = m\otimes 1\right\}$.
In \cite[Theorem 2.7]{Saracco-Frobenius} we proved that the functor $-\otimes B:\M\to\rhopf{B}$ is Frobenius if and only if the canonical map $\varsigma_M:\coinv{M}{B}\to\inv{M}{B}, m\mapsto \cl{m},$ is an isomorphism for every $M\in\rhopf{B}$, if and only if $B$ is a right Hopf algebra (\ie it admits a right convolution inverse of the identity). 



By working with left Hopf modules instead, recall that the counit of the adjunction $\xymatrix @C=15pt{ B\otimes (-): \M \ar@<+0.4ex>[r]|{} & { _{B}^{B}\M }: {^{coB}\left( -\right)} \ar@<+0.4ex>[l]}$ on the Hopf module $\check{B}\coloneqq { _{\bullet }^{\bullet}B} \otimes { _{\bullet }^{\phantom{\bullet}}B} $ induces
\begin{equation*}
\can : \left(
\begin{gathered}
\xymatrix @R=0pt{
{{ _{\bullet }^{\bullet }B} \otimes B} \ar[r] & { _{\bullet }^{\bullet }B} \otimes { ^{coB}\left( {_{\bullet }^{\bullet }B} \otimes { _{\bullet }^{\phantom{\bullet}}B} \right)} \ar[r]^-{\epsilon_{\check{B }}} & { _{\bullet}^{\bullet }B} \otimes { _{\bullet }^{\phantom{\bullet}}B} \\
a\otimes b \ar@{|->}[r] & a\otimes \left( 1\otimes b\right) \ar@{|->}[r] & a_{1}\otimes a_{2}b
}
\end{gathered}
\right)
\end{equation*}

\begin{lemma}
Let $B$ be a bialgebra. The canonical morphism
\begin{equation*}
\can :B\otimes B\rightarrow B\otimes B:\quad a\otimes b\mapsto a_{1}\otimes a_{2}b
\end{equation*}%
can be considered as a morphism $\can :{ _{\bullet }^{\phantom{\bullet}}B^{\bullet }_{\phantom{\bullet}}} \otimes { B_{\bullet}^{\bullet }} \rightarrow { _{\bullet }^{\phantom{\bullet}}B} \otimes {_{\bullet }^{\phantom{\bullet}}B_{\bullet }^{\bullet }} $ in $\quasihopf{B}$.
\end{lemma}

\begin{proof}
The $B$-bilinearity is clear. For colinearity, we compute
\begin{align*}
\left(\can \otimes B\right)\delta\left(a\otimes b\right) & = \can \left( a_{1}\otimes b_{1}\right) \otimes a_{2}b_{2} =a_{1_{1}}\otimes a_{1_{2}}b_{1}\otimes a_{2}b_{2} =a_{1}\otimes a_{2_{1}}b_{1}\otimes a_{2_{2}}b_{2} \\
&=\left( a_{1}\otimes a_{2}b\right) _{0}\otimes \left( a_{1}\otimes a_{2}b\right) _{1} = \delta\left(\can(a\otimes b)\right). \qedhere
\end{align*}
\end{proof}

\begin{lemma}\label{lemma:adjoint}
The assignments $M_{\bullet }^{\bullet }\mapsto B^{\bullet }\otimes M_{\bullet }^{\bullet }$ and $f\mapsto B^{\bullet }\otimes f$ provide a monad $B^{\bullet }\otimes -$ on $\M _{B}^{B}$. Its Eilenberg-Moore category of algebras is $\quasihopf{B}$. In particular, the functor $B\otimes-:\rhopf{B}\to\quasihopf{B},M^\bullet_\bullet\mapsto {_\bullet^{\phantom{\bullet}}B_{\phantom{\bullet}}^\bullet\otimes M^\bullet_\bullet}$ is left adjoint to the forgetful functor ${_BU}:\quasihopf{B}\to\rhopf{B}$:
\begin{equation}\label{eq:adjoint}
\Hom{B}{}{B}{B}{ { _{\bullet }^{\phantom{\bullet}}B^{\bullet }_{\phantom{\bullet}}} \otimes { M_{\bullet }^{\bullet }}}{\quasihopfmod{N} }  \cong \Hom{}{}{B}{B}{ M_{\bullet }^{\bullet }}{N_{\bullet}^{\bullet }}
\end{equation}
\end{lemma}

\begin{proof}
In a nutshell, the comodule $B^{\bullet }$ is an algebra in the monoidal category $\left( \M ^{B},\otimes ,\K \right) $. Given any monoidal category $\cC$ and $A,A^{\prime }$ two algebras, the endofunctors $A\otimes -$ and $-\otimes A^{\prime }$ provide monads on $\cC_{A^{\prime }}$ and $_{A}\cC$ respectively. In particular, $B^{\bullet }\otimes -$ does on $\M _{B}^{B}$. A direct check is also possible: recall that the $B$-comodule structure on the tensor product of two comodules is given by the diagonal coaction, \ie $\delta(n\otimes p) = n_0\otimes p_0 \otimes n_1p_1$ for all $n\in N^\bullet$, $p\in P^\bullet$. Consider the assignments
\begin{gather*}
\rho:B^\bullet_{\phantom{\bullet}} \otimes \left( B^\bullet_{\phantom{\bullet}} \otimes M^\bullet_\bullet \right)  \to B^\bullet_{\phantom{\bullet}}\otimes M^\bullet_\bullet :a\otimes b\otimes m\mapsto ab\otimes m \\
M^\bullet_\bullet \rightarrow B^\bullet _{\phantom{\bullet}} \otimes M^\bullet_\bullet:m\mapsto 1\otimes m
\end{gather*}
for every $M$ in $\M _{B}^{B}$. They are morphism of right Hopf modules since
\begin{equation*}
\left(\rho\otimes B\right)\delta\left(a\otimes \left(b\otimes m\right)\right) = a_{1}b_{1}\otimes m_{0}\otimes a_{2}\left( b_{2}m_{1}\right) =\left( ab\right) _{1}\otimes m_{0}\otimes \left( ab\right) _{2}m_{1} = \delta(ab\otimes m)
\end{equation*}
(the other three compatibilities are trivial). Therefore $B^{\bullet }\otimes -$ is indeed a monad on $\M _{B}^{B}$. An algebra $\left( M,\mu \right) $ for this monad is an object $M$ in $\rhopf{B}$, whose underlying vector space admits a left $B$-module structure $b\triangleright m\coloneqq \mu \left( b\otimes m\right) $ which is $B$-linear and $B$-colinear:
\begin{equation*}
b\triangleright \left( ma\right) =\left( b\triangleright m\right) a,\qquad b_{1}\triangleright m_{0}\otimes b_{2}m_{1}=\left( b\triangleright m\right) _{0}\otimes \left( b\triangleright m\right) _{1},
\end{equation*}
\ie it is an object in $\quasihopf{B}$, and viceversa.
\end{proof}

\begin{remark}\label{rem:freemod}
The fact that $B^\bullet$ is an algebra in the monoidal category $\left( \M ^{B},\otimes ,\K \right)$, mentioned in the proof of Lemma \ref{lemma:adjoint}, implies also that the functor $-\otimes B:\Rcomod{B} \to \rhopf{B}$ is left adjoint to the corresponding forgetful functor $\rhopf{B} \to \Rcomod{B}$, forgetting the module structure (see, for example, \cite[\S VII.4]{MacLane}).
\end{remark}

As a consequence of Lemma \ref{lemma:adjoint} and Remark \ref{rem:freemod}, for all $M$ in $\quasihopf{B}$ we have a $\K $-linear map $\Lambda_M:\Hom{B}{}{B}{B}{B\otimes B}{M}\to \coinv{M}{B}$, natural in $M$, given by the composition of the chain of isomorphisms
\[
\Hom{B}{}{B}{B}{ { _{\bullet }^{\phantom{\bullet}}B^{\bullet }_{\phantom{\bullet}}} \otimes { B_{\bullet }^{\bullet }}}{\quasihopfmod{M} }  \stackrel{\eqref{eq:adjoint}}{\cong} \Hom{}{}{B}{B}{ B_{\bullet }^{\bullet }}{M_{\bullet}^{\bullet }} \cong \Hom{}{}{}{B}{ \K ^{\bullet}}{M^{\bullet }} \cong \coinv{M}{B}
\]
with the morphism $\Hom{B}{}{B}{B}{{ _{\bullet }^{\phantom{\bullet}}B} \otimes \quasihopfmod{B}}{\quasihopfmod{M}} \to \Hom{B}{}{B}{B}{ { _{\bullet }^{\phantom{\bullet}}B^{\bullet }_{\phantom{\bullet}}} \otimes { B_{\bullet }^{\bullet }}}{\quasihopfmod{M} }$ induced by $\can$.
It is given by the assignment $f\mapsto f\left( 1\otimes 1\right) $, whence the following diagram in $\M$ commutes
\begin{equation*}
\begin{gathered}
\xymatrix @!0 @=40pt{ 
 & \Hom{B}{}{B}{B}{ B\otimes B}{M} \ar[dl]_-{\Lambda_M} \ar[dr]^-{\sigma _{M}} & \\
\coinv{M}{B} \ar[rr]_-{\varsigma _{M}} & & \inv{M}{B}
}
\end{gathered}.
\end{equation*}

Recall that a bialgebra $B$ is a Hopf algebra if and only if $\can$ is invertible (in light, for example, of a left-handed version of \cite[Example 2.1.2]{Schauenburg-Galois}).

\begin{proposition}
The following are equivalent for a bialgebra $B$:
\begin{enumerate}[label=(\arabic*), ref=\emph{(\arabic*)}, leftmargin=1cm, labelsep=0.3cm]
\item\label{item1:useless} $B$ is a Hopf algebra;
\item\label{item2:useless} $\sigma$ is a natural isomorphism;
\item\label{item3:useless} $\Lambda$ is a natural isomorphism.
\end{enumerate}
If any of the foregoing conditions holds, then $\varsigma$ is a natural isomorphism.
\end{proposition}

\begin{proof}
The equivalence between \ref{item1:useless} and \ref{item2:useless} comes from Theorem \ref{thm:mainThmHopf}. Concerning the equivalence between \ref{item1:useless} and \ref{item3:useless}, $\Lambda$ is a natural isomorphism if and only if $\Hom{B}{}{B}{B}{\can}{-}$ is a natural isomorphism, if and only if $\can$ is an isomorphism.
\end{proof}

\begin{personal}
Since, to be precise, $\varsigma_M:\coinv{{_BU(M)}}{B}\to{U(\cl{M})}$, where ${_BU}:\quasihopf{B}\to \rhopf{B}$ and $U:\Lmod{B}\to\M$ are the forgetful functors, if it is an iso for every $M\in\rhopf{B}$ then it is also for every $M\in \quasihopf{B}$. Moreover, if $B$ is Hopf then obviously it is an iso.

Concerning the equivalence with the last point: $\varphi_M$ is an isomorphism if and only if $\Hom{B}{}{B}{B}{\can}{-}$ is a natural isomorphism, if and only if $\can$ is an isomorphism.

The latter follows by the subsequent argument. Assume $\cC$ is a category and $f:a\to b$ in $\cC$. Then $\cC(f,-):\cC(b,-)\to \cC(a,-)$ is a natural isomorphism iff exists $\tau:\cC(a,-) \to \cC(b,-)$ natural such that $\tau \circ \cC(f,-) = \id_{\cC(b,-)}$ and $\cC(f,-)\circ \tau = \id_{\cC(a,-)}$. Since $\Nat\left(\cC(a,-), \cC(b,-)\right) \cong \cC(b,a)$, there exists a unique $g:b\to a$ such that $\tau = \cC(g,-)$. Therefore, by uniqueness, $\cC(fg,-) = \cC(g,-) \circ \cC(f,-) = \id_{\cC(b,-)}$ implies $fg=\id_b$ and $\cC(gf,-) = \cC(f,-)\circ \cC(g,-) = \id_{\cC(a,-)}$ implies $gf=\id_a$.
\end{personal}

\begin{remark}
Note however that being $\varsigma_M$ an isomorphism for every $M\in\quasihopf{B}$ is not enough to have that $B$ is a Hopf algebra. In fact, denote by ${_BU}:\quasihopf{B}\to \rhopf{B}$ and by $U:\Lmod{B}\to\M$ the forgetful functors. If $B$ is a right Hopf algebra, then $\varsigma_N$ an isomorphism for every $N\in\rhopf{B}$ and hence, in particular, $\varsigma_M:\coinv{\left({_BU}(M)\right)}{B}\to U\left(\cl{M}\right)$ is an isomorphism for every $M\in\quasihopf{B}$. Since there exist right Hopf algebras that are not Hopf, the latter cannot imply that $B$ is Hopf.
\end{remark}


\subsection{Frobenius functors and unimodularity}\label{ssec:unimodularity}

It would be interesting, in light of the similarity between Theorem \ref{thm:mainThm} and \cite[Theorem 2.7]{Saracco-Frobenius}, to look for an analogue of \cite[Theorem 3.12]{Saracco-Frobenius}. Let us report briefly on some partial achievements in this direction.

For a quasi-bialgebra $A$ one can consider its space of \emph{right integrals} $\int_rA$, which is the the $\K$-module $\left\{t\in A\mid ta=t\varepsilon(a) \text{ for all } a\in A\right\}$, and its space of \emph{left integrals} $\int_lA$, which is the the $\K$-module $\left\{s\in A\mid as=\varepsilon(a)s \text{ for all }a\in A\right\}$. As in \cite[page 14]{HausserNill}, we say that $A$ is \emph{unimodular} if $\int_lA=\int_rA$. The following fact will be used in the forthcoming results.

\begin{lemma}\label{lemma:integrals}
If $\K$ is a field and $A$ is a finite-dimensional quasi-bialgebra with preantipode over $\K$, then $\dim_\K\left(\int_lA\right) = 1 = \dim_\K\left(\int_rA\right)$.
\end{lemma}

\begin{proof}
In view of Proposition \ref{prop:PreQuasiHopf}, $A$ is a finite-dimensional quasi-Hopf algebra. Thus the result follows from \cite[Theorem 2.2]{BulacuCaenepeel}.
\end{proof}

Consider the adjunctions
\begin{equation*}
\xymatrix{
\quasihopf{A} \ar@/_/[d]_-{\cl{(-)}} \\
\Lmod{A} \ar@/_/[u]_-{-\otimes A}
} \qquad\quad
\xymatrix{
\quasihopf{A} \ar@/_/[d]_-{U} \\
\Bimod{A} \ar@/_/[u]_-{-\tildetens{}A}
} \qquad
\xymatrix{
\Bimod{A} \ar@/_/[d]_-{-\tensor{A}\K} \\
\Lmod{A} \ar@/_/[u]_-{-\otimes\K}
}
\end{equation*}
where for every $A$-bimodule $N$, $N\tildetens{}A$ denotes the quasi-Hopf bimodule $\bimod{N}\otimes\quasihopfmod{A}$, $U$ is the functor forgetting the coaction and $\K$ is considered as a left or right $A$-module via $\varepsilon$.

For $V\in\Lmod{A}$, recall that we set $V_\varepsilon\coloneqq V\otimes\K\in\Bimod{A}$. An easy observation allows us to conclude that $\cl{M}=U(M)\tensor{A}\K$ and that $V\otimes A = V_\varepsilon\tildetens{}A$ for all $M\in\quasihopf{A}, V\in\Lmod{A}$. Therefore, similarly to what was proven in \cite[Proposition 3.3]{Saracco-Frobenius}, if $\left( U,-\tildetens{}A\right) $ is Frobenius and if $\Hom{A}{}{A}{}{V_\varepsilon}{U(M)} \cong \Hom{A}{}{}{}{V}{\cl{M}}$ naturally in $V\in \Lmod{A} $ and $M\in \quasihopf{A}$, then 
\begin{equation*}
\Hom{A}{}{A}{A}{V\otimes A}{M} = \Hom{A}{}{A}{A}{V_\varepsilon\tildetens{} A}{M} \cong \Hom{A}{}{A}{}{V_\varepsilon}{U(M)} \cong \Hom{A}{}{}{}{V}{\cl{M}}
\end{equation*}
and so $\left( \cl{(-)},-\otimes A\right)$ is Frobenius, which in turn implies that $A$ admits a preantipode. 

\begin{lemma}\label{lemma:unimodular}
Any bijection $\Hom{A}{}{A}{}{V_\varepsilon}{U(M)}\cong\Hom{A}{}{}{}{V}{\cl{M}}$ natural in $V\in\Lmod{A}$ and $M\in\quasihopf{A}$ is a $\K$-linear natural isomorphism
\[
\Theta_{V,M}:\Hom{A}{}{}{}{V}{\cl{M}} \to \Hom{A}{}{A}{}{V_\varepsilon}{U(M)}, \qquad f \mapsto \tau_M \circ f_\varepsilon
\]
where $\tau_M \coloneqq \Theta_{\cl{M},M}\left(\id_{\cl{M}}\right):\cl{M}\to M$. Moreover, when a $\K$-linear natural isomorphism $\Theta_{V,M}$ exists, then $A$ is unimodular and $\int_lA=\int_rA\cong\K$.
\end{lemma}

\begin{proof}
Assume firstly that a natural bijection $\Theta_{V,M} : \Hom{A}{}{}{}{V}{\cl{M}} \cong \Hom{A}{}{A}{}{V_\varepsilon}{U(M)}$ exists. Since $\cl{M}\in\Lmod{A}$, for every $f\in \Hom{A}{}{}{}{V}{\cl{M}}$ we may compute
\begin{gather*}
\Theta_{V,M}(f) = \Theta_{V,M}(\id_{\cl{M}}\circ f) = \Theta_{V,M}\left(\Hom{A}{}{}{}{f}{\cl{M}}\left(\id_{\cl{M}}\right)\right) \\
= \Hom{A}{}{A}{}{f_\varepsilon}{U({M})}\left(\Theta_{\cl{M},M}(\id_{\cl{M}})\right) = \Theta_{\cl{M},M}(\id_{\cl{M}}) \circ f_\varepsilon.
\end{gather*}
By setting $\tau_M\coloneqq \Theta_{\cl{M},M}(\id_{\cl{M}}) : \cl{M}_\varepsilon\to U(M)$, we have that $\Theta_{V,M}(f) = \tau_M\circ f_\varepsilon$ for all $f\in \Hom{A}{}{}{}{V}{\cl{M}}$, which is $\K$-linear. Now, assume that a natural isomorphism $\Theta_{V,M}$ exists and consider the particular case $V=A\in\Lmod{A}$ and $M=A\in\quasihopf{A}$. On the one hand, the assignment $\Hom{A}{}{A}{}{A_\varepsilon}{U(A)} \to \int_{r}A: f\mapsto f(1)$ is invertible with explicit inverse $\int_{r}A \to \Hom{A}{}{A}{}{A_\varepsilon}{U(A)} : t\mapsto\left[ a\mapsto at\right]$. On the other hand,
\[
\Hom{A}{}{}{}{A}{\cl{A}} \cong \Hom{A}{}{}{}{A}{\K} \cong \Hom{}{}{}{}{\K}{\K} =\K,
\]
so that any element $f\in \Hom{A}{}{}{}{A}{\cl{A}}$ is of the form $f_k:a\mapsto \varepsilon(a)k\,\cl{1_A}$ for some $k\in\K$. Therefore, since $\Theta_{A,A}$ is an isomorphism, for every $t\in\int_rA$ there exists a (unique) $k\in\K$ such that
\[
at = \Theta_{A,A}(f_k)(a) = \tau_{A}\left(f_k(a)\right) = \varepsilon(a)k\tau_A(1_A)
\]
for every $a\in A$. In particular, for $a=1_A$, $t = k\tau_A(1_A)$ and so it is a left integral as well, showing that $A$ is unimodular. Moreover, we have the $\K$-linear isomorphism
\begin{equation*}
\int_{r}A \cong \Hom{A}{}{A}{}{A_\varepsilon}{U(A)} \cong \Hom{A}{}{}{}{A}{\cl{A}} \cong \K
\end{equation*}
and hence $\int_{r}A$ is free of rank $1$ over $\K$.
\end{proof}

\begin{remark}
The interested reader may check that there is a bijection between natural transformations $\Theta_{V,M}:\Hom{A}{}{}{}{V}{\cl{M}} \to \Hom{A}{}{A}{}{V_\varepsilon}{U(M)}$ and $\K$-linear morphisms $\partial:A\to A\otimes A, a \mapsto \partial^{(1)}(a) \otimes \partial^{(2)}(a)$ satisfying, for all $a,b\in A$,
\begin{gather*}
a\partial^{(1)}(b) \otimes \partial^{(2)}(b) = \partial^{(1)}(a_2b)a_1 \otimes \partial^{(2)}(a_2b), \\
\partial^{(1)}(a) \otimes \partial^{(2)}(a)\varepsilon(b) = \partial^{(1)}(a) \otimes \partial^{(2)}(a)b = \partial^{(1)}(ab_2) \otimes b_1\partial^{(2)}(ab_2).
\end{gather*}
This is given by $\Theta \mapsto \left[a \mapsto \left( (A\otimes A)\otimes \varepsilon \right) \left( \tau _{(A\otimes A)\otimes A} \left( \cl{(1\otimes 1)\otimes a}\right) \right)\right]$ and $\partial\mapsto \Theta^{(\partial)}$, where
\[
\Theta^{(\partial)}_{V,M}(f):V_\varepsilon\to U(M), \quad v\mapsto \partial^{(1)}(f(v)_1)\cdot f(v)_0\cdot \partial^{(2)}(f(v)_1).
\]
\begin{invisible}
Assume that we have $\Theta :\Hom{A}{}{}{}{V}{\cl{M}} \to \Hom{A}{}{A}{}{V_\varepsilon}{U(M)}$ natural in $M\in \quasihopf{A}$ and $V\in \Lmod{A}$ and set $\tau _{M}\coloneqq \Theta \left( \id_{\cl{M}}\right) $ as above. It is natural in $M\in\quasihopf{A}$ as for every $f:M\to P \in\quasihopf{A}$, we have
\[
\tau_P\circ \cl{f} = \Theta_{\cl{M},P} (\cl{f}) = \Theta_{\cl{M},P} (\cl{f}\circ \id_{\cl{M}}) = U(f) \circ \Theta_{\cl{M},M}(\id_{\cl{M}}) = U(f)\circ \tau_M.
\]
Since $\delta _{M}:\quasihopfmod{M}\to \bimod{M}\otimes \quasihopfmod{A}\in \quasihopf{A}$ we have that
\begin{equation*}
\tau _{M}=\left( M\otimes \varepsilon \right) \circ \tau _{M\otimes A}\circ \cl{\delta _{M}}.
\end{equation*}
Since $\mu _{M}:\lmod{M}\otimes \rmod{A} \to \bimod{M}\in \Bimod{A}$ and $f_m:\lmod{A}\to\lmod{M},a\mapsto am\in\Lmod{A}$ we also have that for all $m\in M$
\begin{align*}
\tau _{M}\left( \cl{m}\right) & = \left( \left( M\otimes \varepsilon\right) \circ \tau _{M\otimes A}\circ \cl{\delta _{M}}\right) \left( \cl{m}\right) = \left( \left( M\otimes \varepsilon \right) \circ \tau_{M\otimes A}\right) \left( \cl{m_{0}\otimes m_{1}}\right) \\
 & = \left( \left( M\otimes \varepsilon \right) \circ \tau _{M\otimes A}\right) \left( \cl{\mu _{M}\left( m_{0}\otimes 1\right) \otimes m_{1}}\right) \\
 & = \left( \left( M\otimes \varepsilon \right) \circ \left( \mu _{M}\otimes A\right) \circ \tau _{(M\otimes A)\otimes A}\right) \left( \cl{(m_{0}\otimes 1)\otimes m_{1}}\right) \\
 & = \left( \left( M\otimes \varepsilon \right) \circ \left( \mu _{M}\otimes A\right) \circ \left( f_{m_{0}}\otimes A\otimes A\right) \circ \tau_{(A\otimes A)\otimes A}\right) \left( \cl{(1\otimes 1)\otimes m_{1}}\right) \\
 & = \left( \mu _{M}\circ \left( f_{m_{0}}\otimes A\right) \circ \left((A\otimes A)\otimes \varepsilon \right) \circ \tau _{(A\otimes A)\otimes A}\right) \left( \cl{(1\otimes 1)\otimes m_{1}}\right)
\end{align*}
If we set $\partial:A\to A\otimes A, a \mapsto \left( (A\otimes A)\otimes \varepsilon \right) \left( \tau _{(A\otimes A)\otimes A} \left( \cl{(1\otimes 1)\otimes a }\right)\right)=:\partial^{(1)}(a)\otimes \partial^{(2)}(a)$ then
\begin{equation*}
\tau _{M}\left( \cl{m}\right) =\partial^{(1)}(m_1)m_{0}\partial^{(2)}(m_1) .
\end{equation*}
This has to satisfy
\begin{align*}
a\partial^{(1)}(m_1)m_{0}\partial^{(2)}(m_1) b & = a\tau _{M}\left( \cl{m}\right) b=\tau_{M}\left( \cl{am}\varepsilon \left( b\right)\right)  = \tau_{M}\left( \cl{amb}\right) \\
 & = \partial^{(1)}\left( a_2m_{1}b_{2}\right)a_1m_{0}b_{1}\partial^{(2)}\left( a_2m_{1}b_{2}\right) .
\end{align*}

We claim that there exists a bijective correspondence between maps $\tau$ as above and maps $\partial:A\to A\otimes A$ which satisfy
\begin{gather*}
a\partial^{(1)}(b) \otimes \partial^{(2)}(b) = \partial^{(1)}(a_2b)a_1 \otimes \partial^{(2)}(a_2b), \\
\partial^{(1)}(a) \otimes \partial^{(2)}(a)b = \partial^{(1)}(ab_2) \otimes b_1\partial^{(2)}(ab_2), \\
\partial^{(1)}(a) \otimes \partial^{(2)}(a)b = \partial^{(1)}(a) \otimes \partial^{(2)}(a)\varepsilon(b).
\end{gather*}
In fact, if such a $\partial$ exists then 
\[
\tau_M(\cl{m}) \coloneqq  \partial^{(1)}(m_1)m_{0}\partial^{(2)}(m_1)
\]
is well-defined because if $m_ib_i\in MB^+$ then
\[
\partial^{(1)}(m_{i_1}b_{i_2})m_{i_0}b_{i_1}\partial^{(2)}(m_{i_1}b_{i_2}) = \partial^{(1)}(m_{i_1})m_{i_0}\partial^{(2)}(m_{i_1})\varepsilon(b_i) = 0,
\]
and it is bilinear because
\[
a\tau_M(\cl{m})b = a\partial^{(1)}(m_1)m_{0}\partial^{(2)}(m_1)b = \partial^{(1)}(a_2m_1)a_1m_{0} \otimes \partial^{(2)}(a_2m_1)\varepsilon(b) = \tau_M(\cl{am}\varepsilon(b)).
\]
Conversely, if we consider $\partial:A\to A\otimes A, a \mapsto \left( (A\otimes A)\otimes \varepsilon \right) \left( \tau _{(A\otimes A)\otimes A} \left( \cl{(1\otimes 1)\otimes a}\right) \right)=:\partial^{(1)}(a)\otimes \partial^{(2)}(a)$ then
\begin{align*}
a\partial^{(1)}(b) \otimes \partial^{(2)}(b) & = \left( (A\otimes A)\otimes \varepsilon \right) \left( (a_1\otimes 1\otimes a_2)\tau _{(A\otimes A)\otimes A} \left(\cl{ (1\otimes 1)\otimes b }\right)\right) \\
 & = \left( (A\otimes A)\otimes \varepsilon \right) \left(\tau _{(A\otimes A)\otimes A} \left(\cl{ (a_1\otimes 1)\otimes a_2b }\right)\right) \\
 & = \left( (A\otimes A)\otimes \varepsilon \right) \left((f_{a_1}\otimes A\otimes A)\tau _{(A\otimes A)\otimes A} \left(\cl{ (1\otimes 1)\otimes a_2b }\right)\right) \\
 & = (f_{a_1}\otimes A)\left(\partial^{(1)}(a_2b) \otimes \partial^{(2)}(a_2b)\right) = \partial^{(1)}(a_2b)a_1 \otimes \partial^{(2)}(a_2b),
\end{align*}
\begin{align*}
\partial^{(1)}(a) \otimes \partial^{(2)}(a)b & = \left( (A\otimes A)\otimes \varepsilon \right) \left( \tau _{(A\otimes A)\otimes A} \left( \cl{(1\otimes 1)\otimes a }\right)(1\otimes b_1\otimes b_2)\right) \\
 & = \left( (A\otimes A)\otimes \varepsilon \right) \left( \tau _{(A\otimes A)\otimes A} \left( \cl{(1\otimes 1)\otimes a }\right)\varepsilon(b)\right) \\
 & = \left( (A\otimes A)\otimes \varepsilon \right) \left( \tau _{(A\otimes A)\otimes A} \left( \cl{(1\otimes b_1)\otimes ab_2 }\right)\right) \\
 & = \left( (A\otimes A)\otimes \varepsilon \right) \left((A\otimes f_{b_1}\otimes A) \tau _{(A\otimes A)\otimes A} \left( \cl{(1\otimes 1)\otimes ab_2 }\right)\right) \\
 & = (A\otimes f_{b_1})\left( (A\otimes A)\otimes \varepsilon \right) \left( \tau _{(A\otimes A)\otimes A} \left( \cl{(1\otimes 1)\otimes ab_2 }\right)\right) \\ 
 & = \partial^{(1)}(ab_2) \otimes b_1\partial^{(2)}(ab_2). \qed
\end{align*}
Notice also that if $A$ admits a two-sided integral $t$, then $\partial:A\to A\otimes A, a\mapsto \varepsilon(a)1_A\otimes t$ satisfies the requirements and hence it corresponds to a morphism $\Hom{A}{}{}{}{V}{\cl{M}} \to \Hom{A}{}{A}{}{V_\varepsilon}{U(M)}$. If furthermore $A$ admits a preantipode $S$ and if $\lambda_M:\cl{M} \to M, \cl{m}\mapsto \Phi^1m_0S(\Phi^2m_1)\Phi^3$ then to every $f:V\to \cl{M}$ one can associate $v\mapsto \lambda_M(f(v))t$ as in \cite[Lemma 3.11]{Saracco-Frobenius}, which now however assume the form $v\mapsto m_vt$ where $m_v\in M$ is any preimage of $f(v)$ in $M$ (in the previous paper, we had the canonical morphism $m\mapsto m_0S(m_1)$ multiplied by the one-sided integral $\varepsilon(e^1)e^2$). This would have admitted a kind of left inverse given by $\Hom{A}{}{A}{}{V_\varepsilon}{U(M)} \to \Hom{}{}{}{}{V}{\cl{M}}, g\mapsto \left[v\mapsto \cl{g(v)_0}\psi\left(g(v)_1\right)\right]$ exactly as in \cite[Lemma 3.11]{Saracco-Frobenius}. In fact
\[
\cl{(m_v)_0t_1}\psi\left((m_v)_1t_2\right) = \cl{(m_v)_0}\psi\left((m_v)_1t\right) = \cl{m_v}\psi\left(t\right) = f(v)\psi(t)
\]
and as $t$ we could have chosen $\varepsilon(e^1)e^2$ itself, without loss of generality, so that $\psi(t) = 1$. Nevertheless, there's no evidence that $v\mapsto \cl{g(v)_0}\psi\left(g(v)_1\right)$ is left $A$-linear, which is the main obstacle here.\end{invisible}
\end{remark}

\begin{proposition}
Assume that $\K$ is a field. Then the following assertions are equivalent for a quasi-bialgebra $A$ over $\K$: 
\begin{enumerate}[label=(\arabic*), ref=\emph{(\arabic*)}, leftmargin=1cm, labelsep=0.3cm]
\item\label{item:3.14-1} $\left( U,-\tildetens{}A\right) $ is Frobenius and $\Hom{A}{}{A}{}{V_\varepsilon}{U(M)} \cong \Hom{A}{}{}{}{V}{\cl{M}}$ naturally in $V\in \Lmod{A} $ and $M\in \quasihopf{A}$;
\item\label{item:3.14-2} $-\otimes A$ is Frobenius, $A$ is finite-dimensional and unimodular, and $\int_{l}A=\int_{r}A\cong \Bbbk $;
\item\label{item:3.14-3} $A$ is a finite-dimensional unimodular quasi-bialgebra with preantipode;
\item\label{item:3.14-4} $A$ is a finite-dimensional unimodular quasi-Hopf algebra.
\end{enumerate}
Moreover, any one of the above implies
\begin{enumerate}[resume*]
\item\label{item:3.14-5} $A$ is a unimodular Frobenius $\K$-algebra whose Frobenius homomorphism $\psi$ is a left cointegral in the sense of \cite[Definition 4.2]{HausserNill}.
\end{enumerate}
\end{proposition}

\begin{proof}
We know that $-\otimes A$ is Frobenius if and only if $A$ admits a preantipode by Theorem \ref{thm:mainThm} and that the spaces of integrals over a finite-dimensional quasi-bialgebra with preantipode are always $1$-dimensional (see Lemma \ref{lemma:integrals}), whence $\ref{item:3.14-2} \Leftrightarrow \ref{item:3.14-3}$. The equivalence \ref{item:3.14-3} $\Leftrightarrow$ \ref{item:3.14-4} follows from the fact that, in the finite-dimensional case, quasi-Hopf algebras and quasi-bialgebras with preantipode are equivalent notions (see Proposition \ref{prop:PreQuasiHopf}). It follows from Lemma \ref{lemma:unimodular} and our observations preceding it that if \ref{item:3.14-1} holds then $-\otimes A$ is Frobenius, $A$ is unimodular and $\int_{l}A=\int_{r}A\cong \Bbbk $. In addition, in view of \cite[Theorem 5.8]{BulacuCaenepeelTorrecillas1} and \cite[Proposition 1.3]{BulacuCaenepeelTorrecillas2}, if $\left( U,-\tildetens{}A\right) $ is Frobenius then $A$ is finite-dimensional. Therefore \ref{item:3.14-1} $\Rightarrow$ \ref{item:3.14-2}. Let us conclude by showing that \ref{item:3.14-4} implies \ref{item:3.14-1}. \begin{enumerate*}[label=(\textbf{\alph*}),ref={\textbf{\alph*}}] \item\label{item:BCT} In light of \cite[Theorem 3.4(iv)]{BulacuCaenepeelTorrecillas2}, since $A$ is a finite-dimensional unimodular quasi-Hopf algebra, the pair $\left( U,-\tildetens{}A\right) $ is Frobenius. \item\label{item:io} Since $A$ is a quasi-Hopf algebra, in particular it is a quasi-bialgebra with preantipode (see \cite[Theorem 6]{Saracco}) and hence $-\otimes A$ is an equivalence of categories.\end{enumerate*} Therefore
\[
\Hom{A}{}{A}{}{V_\varepsilon}{U(M)} \stackrel{\eqref{item:BCT}}{\cong} \Hom{A}{}{A}{A}{V_\varepsilon\tildetens{}A}{M} = \Hom{A}{}{A}{A}{V\otimes A}{M} \stackrel{\eqref{item:io}}{\cong} \Hom{A}{}{}{}{V}{\cl{M}}.
\]
Finally, $\ref{item:3.14-4} \Rightarrow \ref{item:3.14-5}$ follows from \cite[Theorem 4.3]{HausserNill} (together with \cite[Theorem 2.2]{BulacuCaenepeel}. See also \cite[Lemma 3.2]{Kadison2}).
\end{proof}

\begin{remark}
It is still an open question if \ref{item:3.14-5} implies any of the other assertions or which additional conditions on $A$ in \ref{item:3.14-5} would allow us to prove that.
\end{remark}

In this direction, and for the sake of future investigations on the subject, let us provide the explicit details of an equivalent description of when the pair $\left( U,-\tildetens{}A\right) $ is Frobenius.

\begin{theorem}[{\cite[Theorem 5.8]{BulacuCaenepeelTorrecillas1}}]\label{thm:FrogetFron}
For a quasi-bialgebra $A$, the pair $( U,-\tildetens{} A) $ is Frobenius if and only if there exists $z\coloneqq z^{( 1) }\otimes z^{( 2) }\otimes z^{( 3) }\in A\otimes A\otimes A$ and $\omega :A\otimes A\to A\otimes A,a\otimes b\mapsto \omega ^{( 1) }( a\otimes b) \otimes \omega ^{( 2) }( a\otimes b) $ such that for all $a,b\in A$
\begin{gather*}
a_{1}z^{( 1) }\otimes z^{( 2) }b_{1}\otimes a_{2}z^{( 3) }b_{2}=z^{( 1) }a\otimes bz^{(
2) }\otimes z^{( 3) }, \\
\omega ^{( 1) }( x_{1_{2}}ay_{1_{2}}\otimes x_{2}by_{2}) x_{1_{1}}\otimes y_{1_{1}}\omega ^{( 2) }( x_{1_{2}}ay_{1_{2}}\otimes x_{2}by_{2}) =x\omega ^{( 1) }( a\otimes b) \otimes \omega ^{( 2) }( a\otimes b) y, \\
\begin{split}
\omega ^{( 1) }( \varphi ^{3}a_{2}\Phi ^{3}\otimes b) _{1}\varphi ^{1}\otimes \Phi ^{1}\omega ^{( 2) }( \varphi ^{3}a_{2}\Phi ^{3}\otimes b) _{1}\otimes \omega ^{( 1) }( \varphi ^{3}a_{2}\Phi ^{3}\otimes b) _{2}\varphi ^{2}a_{1}\Phi ^{2}\omega ^{( 2) }( \varphi ^{3}a_{2}\Phi ^{3}\otimes b) _{2} \\ 
=\omega ^{( 1) }( \varphi _{2}^{1}a\Phi _{2}^{1}\otimes \varphi ^{2}b_{1}\Phi ^{2}) \varphi _{1}^{1}\otimes \Phi _{1}^{1}\omega^{( 2) }( \varphi _{2}^{1}a\Phi _{2}^{1}\otimes \varphi^{2}b_{1}\Phi ^{2}) \otimes \varphi ^{3}b_{2}\Phi ^{3},
\end{split} \\
\omega ^{( 1) }( z_{2}^{( 1) }az_{2}^{(2) }\otimes z^{( 3) }) z_{1}^{( 1)}\otimes z_{1}^{( 2) }\omega ^{( 2) }(z_{2}^{( 1) }az_{2}^{( 2) }\otimes z^{( 3)}) =\varepsilon ( a) 1\otimes 1, \\
\omega ^{( 1) }( \varphi _{2}^{1}z^{( 3) }\Phi_{2}^{1}\otimes \varphi ^{2}a_{1}\Phi ^{2}) \varphi _{1}^{1}z^{(1) }\otimes z^{( 2) }\Phi _{1}^{1}\omega ^{( 2)}( \varphi _{2}^{1}z^{( 3) }\Phi _{2}^{1}\otimes \varphi^{2}a_{1}\Phi ^{2}) \otimes \varphi ^{3}a_{2}\Phi ^{3}=1\otimes 1\otimes a.
\end{gather*}
\end{theorem}

\begin{proof}
We refer to \cite{BulacuCaenepeelTorrecillas2} for the notations. In view of \cite[Proposition 3.2]{BulacuCaenepeelTorrecillas2},
\[
\left(A,\psi:A\otimes \left(\op{A}\otimes A\right)\to \left(\op{A}\otimes A\right)\otimes A,\ x\otimes (a\otimes b) \mapsto (a_1\otimes b_1)\otimes a_2xb_2\right)
\]
is a coalgebra in the category $\cT^{\#}_{\op{A}\otimes A}$ and the associated category of Doi-Hopf modules is exactly $\cM(\op{A}\otimes A)_{\op{A}\otimes A}^A\cong \quasihopf{A}$. According to \cite[Theorem 5.8]{BulacuCaenepeelTorrecillas1}, the forgetful functor $U$ is Frobenius if and only if $(A,\psi)$ is a Frobenius coalgebra in $\cT^{\#}_{\op{A}\otimes A}$. By writing explicitly the conditions reported in \cite[\S1.2]{BulacuCaenepeelTorrecillas2}, one finds exactly the ones in the statement, with $z^{(1)}\otimes z^{(2)} \otimes z^{(3)}$ playing the role fo the Frobenius element and $\omega$ the role of the Casimir morphism.
\end{proof}

\begin{invisible}
\begin{corollary}\label{cor:ForgetFrob}
For a bialgebra $B$, the pair $\left( U,-\tildetens{}B\right) $ is Frobenius if and only if there exists $z\coloneqq z^{\left( 1\right) }\otimes z^{\left( 2\right) }\otimes z^{\left( 3\right) }\in B\otimes B\otimes B$ and $\omega :B\otimes B\rightarrow B\otimes B,a\otimes b\mapsto \omega ^{\left( 1\right) }\left( a\otimes b\right) \otimes \omega ^{\left( 2\right) }\left( a\otimes b\right) $ such that for all $a,b\in B$
\begin{gather}
a_{1}z^{\left( 1\right) }\otimes z^{\left( 2\right) }b_{1}\otimes a_{2}z^{\left( 3\right) }b_{2}=z^{\left( 1\right) }a\otimes bz^{\left( 2\right) }\otimes z^{\left( 3\right) },  \label{eq:corFF1} \\
\omega ^{\left( 1\right) }\left( x_{2}ay_{2}\otimes x_{3}by_{3}\right)x_{1}\otimes y_{1}\omega ^{\left( 2\right) }\left( x_{2}ay_{2}\otimes x_{3}by_{3}\right) =x\omega ^{\left( 1\right) }\left( a\otimes b\right)\otimes \omega ^{\left( 2\right) }\left( a\otimes b\right) y,   \label{eq:corFF2} \\
\begin{split}\omega ^{\left( 1\right) }\left( a_{2}\otimes b\right) _{1}\otimes \omega^{\left( 2\right) }\left( a_{2}\otimes b\right) _{1}\otimes \omega ^{\left(1\right) }\left( a_{2}\otimes b\right) _{2}a_{1}\omega ^{\left( 2\right)}\left( a_{2}\otimes b\right) _{2} \\ =\omega ^{\left( 1\right) }\left( a\otimes b_{1}\right) \otimes \omega ^{\left( 2\right) }\left( a\otimes b_{1}\right)\otimes b_{2},  \label{eq:corFF3}\end{split} \\
\omega ^{\left( 1\right) }\left( z_{2}^{\left( 1\right) }az_{2}^{\left(2\right) }\otimes z^{\left( 3\right) }\right) z_{1}^{\left( 1\right)}\otimes z_{1}^{\left( 2\right) }\omega ^{\left( 2\right) }\left(z_{2}^{\left( 1\right) }az_{2}^{\left( 2\right) }\otimes z^{\left( 3\right)}\right) =\varepsilon \left( a\right) 1\otimes 1,  \label{eq:corFF4} \\
\omega ^{\left( 1\right) }\left( z^{\left( 3\right) }\otimes a_{1}\right)z^{\left( 1\right) }\otimes z^{\left( 2\right) }\omega ^{\left( 2\right)}\left( z^{\left( 3\right) }\otimes a_{1}\right) \otimes a_{2}=1\otimes 1\otimes a.  \label{eq:corFF5}
\end{gather}
\end{corollary}

\begin{remark}
It is noteworthy that Corollary \ref{cor:ForgetFrob} can be deduced from \cite[Theorem 38]{CaenepeelMilitaruZhu} as follows. By resorting to the notation used therein, take $A\coloneqq \op{B}\otimes B$, $C\coloneqq B$ and
\begin{equation}\label{eq:inter}
\psi:B\otimes (\op{B}\otimes B)\to (\op{B}\otimes B)\otimes B, \quad x\otimes (a\otimes b)\mapsto (a_1\otimes b_1)\otimes a_2xb_2.
\end{equation}
Write $(a\otimes b)_\psi \otimes x^\psi = (a\otimes b)_\Psi \otimes x^\Psi = \cdots$ for $\psi(x\otimes (a\otimes b))$ in $(\op{B}\otimes B)\otimes B$ and $a\cdot a'\coloneqq a'a$ for the multiplication in $\op{B}$. Since for all $a,b,x\in B$ we have
\begin{gather*}
\begin{split}
& (a\cdot a'\otimes bb')_\psi\otimes x^\psi = \psi(x\otimes (a\cdot a'\otimes bb')) = \psi(x\otimes (a'a\otimes bb')) \stackrel{\eqref{eq:inter}}{=} (a'_1a_1\otimes b_1b'_1)\otimes a'_2a_2xb_2b'_2 \\
 & = (a_1\otimes b_1)(a'_1\otimes b'_1)\otimes a'_2a_2xb_2b'_2 \stackrel{\eqref{eq:inter}}{=} (a\otimes b)_\psi(a'_1\otimes b'_1)\otimes a'_2x^\psi b'_2 \stackrel{\eqref{eq:inter}}{=} (a\otimes b)_\psi(a'\otimes b')_\Psi\otimes x^{\psi\Psi},
\end{split} \\
(1\otimes 1)_\psi\otimes x^\psi = \psi(x\otimes (1\otimes 1)) \stackrel{\eqref{eq:inter}}{=} (1\otimes 1)\otimes x, \\
\begin{split}
(a\otimes b)_\psi\otimes \Delta(x^\psi) & \stackrel{\eqref{eq:inter}}{=} (a_1\otimes b_1)\otimes \Delta(a_2xb_2) = (a_1\otimes b_1)\otimes a_2x_1b_2\otimes a_3x_2b_3 \\
 & \stackrel{\eqref{eq:inter}}{=} (a_1\otimes b_1)_\Psi\otimes x_1^{\Psi}\otimes a_2x_2b_2 \stackrel{\eqref{eq:inter}}{=} (a\otimes b)_{\psi\Psi}\otimes x_1^\Psi\otimes x_2^\psi, 
\end{split} \\
(a\otimes b)_\psi\varepsilon\left(x^\psi\right) \stackrel{\eqref{eq:inter}}{=} (a_1\otimes b_1)\varepsilon(a_2xb_2) = (a\otimes b)\varepsilon(x),
\end{gather*}
it follows that $(\op{B}\otimes B,B,\psi)$ is a right-right entwining structure on $\K$ in the sense of \cite[\S2.1]{CaenepeelMilitaruZhu}. An element $z=a^1\otimes c^1\in W_5$ (see \cite[Proposition 71, 1.]{CaenepeelMilitaruZhu}) is an element $z=z^{\left( 1\right) }\otimes z^{\left( 2\right) }\otimes z^{\left( 3\right) }\in (\op{B}\otimes B)\otimes B$ such that 
\begin{gather*}
z^{\left( 1\right) }a\otimes bz^{\left( 2\right) } \otimes z^{\left( 3\right) } =a\cdot z^{\left( 1\right) }\otimes bz^{\left( 2\right) } \otimes z^{\left( 3\right) } = (a\otimes b)\left(z^{\left( 1\right) }\otimes z^{\left( 2\right) }\right) \otimes z^{\left( 3\right) } \\
 \stackrel{\left(W_5\right)}{=} \left(z^{\left( 1\right) }\otimes z^{\left( 2\right) }\right)(a\otimes b)_\psi \otimes z^{\left( 3\right)\psi } \stackrel{\eqref{eq:inter}}{=} z^{\left( 1\right) } \cdot a_1\otimes z^{\left( 2\right) }b_1 \otimes a_2z^{\left( 3\right) }b_2= a_1z^{\left( 1\right) }\otimes z^{\left( 2\right) }b_1 \otimes a_2z^{\left( 3\right) }b_2,
\end{gather*}
which is \eqref{eq:corFF1}. An element $\vartheta\in V_5$ (see \cite[Proposition 70, 2.]{CaenepeelMilitaruZhu}) is a map $\vartheta:B\otimes B\to\op{B}\otimes B, x\otimes y\mapsto \vartheta^{(1)}(x\otimes y)\otimes \vartheta^{(2)}(x\otimes y)$ such that
\begin{gather*}
a\vartheta^{(1)}(x\otimes y)\otimes \vartheta^{(2)}(x\otimes y)b = \vartheta^{(1)}(x\otimes y)\cdot a\otimes \vartheta^{(2)}(x\otimes y)b = \vartheta(x\otimes y)(a\otimes b) \\
 \stackrel{\left(V_5\right)}{=} (a\otimes b)_{\psi\Psi}\vartheta(x^{\Psi}\otimes y^{\psi}) \stackrel{\eqref{eq:inter}}{=} (a_1\otimes b_1)\vartheta(a_2xb_2\otimes a_3yb_3) \\
= a_1\cdot \vartheta^{(1)}(a_2xb_2\otimes a_3yb_3) \otimes b_1\vartheta^{(2)}(a_2xb_2\otimes a_3yb_3) \\
= \vartheta^{(1)}(a_2xb_2\otimes a_3yb_3)a_1\otimes b_1\vartheta^{(2)}(a_2xb_2\otimes a_3yb_3),
\end{gather*}
which is \eqref{eq:corFF2}, and 
\begin{gather*}
\vartheta ^{\left( 1\right) }\left( x\otimes y_{1}\right) \otimes \vartheta ^{\left( 2\right) }\left( x\otimes y_{1}\right)\otimes y_{2} = \vartheta(x\otimes y_1)\otimes y_2 \stackrel{\left(V_5\right)}{=} \vartheta(x_2\otimes y)_\psi\otimes x_1^\psi \\
 \stackrel{\eqref{eq:inter}}{=} \vartheta^{(1)}(x_2\otimes y)_1 \otimes \vartheta^{(2)}(x_2\otimes y)_1 \otimes \vartheta^{(1)}(x_2\otimes y)_2x_1\vartheta^{(2)}(x_2\otimes y)_2,
\end{gather*}
which is \eqref{eq:corFF3}. In addition, they have to satisfy (see \cite[Theorem 38, (3.54)]{CaenepeelMilitaruZhu})
\[
\vartheta ^{\left( 1\right) }\left( z_{2}^{\left( 1\right) }az_{2}^{\left(2\right) }\otimes z^{\left( 3\right) }\right) z_{1}^{\left( 1\right)}\otimes z_{1}^{\left( 2\right) }\vartheta ^{\left( 2\right) }\left(z_{2}^{\left( 1\right) }az_{2}^{\left( 2\right) }\otimes z^{\left( 3\right)}\right) \stackrel{\eqref{eq:inter}}{=} (z^{(1)}\otimes z^{(2)})_\psi\vartheta\left(a^\psi\otimes z^{(3)}\right) \stackrel{(3.54)}{=} \varepsilon(a)1\otimes 1,
\]
which is \eqref{eq:corFF4}, and
\begin{gather*}
\vartheta ^{\left( 1\right) }\left( z^{\left( 3\right) }\otimes a\right)z^{\left( 1\right) }\otimes z^{\left( 2\right) }\vartheta ^{\left( 2\right)}\left( z^{\left( 3\right) }\otimes a\right) \stackrel{\eqref{eq:inter}}{=} \left(z^{(1)}\otimes z^{(2)}\right)\vartheta\left(z^{(3)}\otimes a\right) \stackrel{(3.54)}{=} \varepsilon(a)1\otimes 1
\end{gather*}
which is equivalent to \eqref{eq:corFF5}.
\end{remark}

\begin{theorem}[Compare with {\cite[Lemma 3.2]{Saracco-Frobenius}}]
For every quasi-Hopf bimodule $N\in\quasihopf{A}$ we have a bijection
\begin{equation}\label{eq:superadj}
\xymatrix@R=0pt{
\Hom{A}{}{A}{A}{\bimod{M}\otimes \quasihopfmod{N}}{\quasihopfmod{P}} \ar@{<->}[r] & \Hom{A}{}{A}{}{\bimod{M}}{\Hom{A}{}{A}{A}{\lmod{A}\otimes \rmod{A}\otimes \quasihopfmod{N}}{\quasihopfmod{P}}} \\
f \ar@{|->}[r] & \left[m \mapsto \left[a\otimes b\otimes n \mapsto f\left(a \cdot m\cdot b\otimes n\right)\right]\right] \\
\left[m\otimes n\mapsto g(m)(1\otimes 1\otimes n)\right] & g \ar@{|->}[l]
}
\end{equation}
which is natural in $M\in \Bimod{A}$ and $P\in\quasihopf{A}$. The left and right $A$-module structures on $\Hom{A}{}{A}{A}{\lmod{A}\otimes \rmod{A}\otimes \quasihopfmod{N}}{\quasihopfmod{P}}$ are explicitly given by 
\begin{equation}\label{eq:2triang}
\left(x\triangleright f\triangleleft y \right)(a\otimes b \otimes n) \coloneqq f\left(ax\otimes yb \otimes n\right)
\end{equation}
for every $a,b,x,y\in A$, $n\in N$ and $f\in \Hom{A}{}{A}{A}{\lmod{A}\otimes \rmod{A}\otimes \quasihopfmod{N}}{\quasihopfmod{P}}$. Therefore, the functor $-\tildetens{}N:\Bimod{A}\to\quasihopf{A}$ is left adjoint to the functor 
\[
\Hom{A}{}{A}{A}{\lmod{A} \otimes \rmod{A} \otimes \quasihopfmod{N}}{-} : \quasihopf{A} \longrightarrow \Bimod{A}.
\]
In particular, the functor $-\tildetens{}A:\Bimod{A}\to\quasihopf{A}$ always admits a right adjoint, given by
\[
\Hom{A}{}{A}{A}{\lmod{A} \otimes \rmod{A} \otimes \quasihopfmod{A}}{-} : \quasihopf{A} \longrightarrow \Bimod{A},
\]
and therefore $(U,-\tildetens{}A)$ is Frobenius if and only if there is a natural isomorphism
\begin{equation}\label{eq:Xi}
\Xi : \Hom{A}{}{A}{A}{\lmod{A} \otimes \rmod{A} \otimes \quasihopfmod{A}}{-} \cong U.
\end{equation}
\end{theorem}

\begin{proof}
For the sake of the reader, let us recall that the structures on $\bimod{M}\otimes \quasihopfmod{N}$ are given by
\begin{equation}\label{eq:trimodM}
\begin{gathered}
x\cdot (m\otimes n)\cdot y = x_1\cdot m\cdot y_1 \otimes x_2\cdot n \cdot y_2, \\
\delta(m\otimes n) = \varphi^1\cdot m \cdot \Phi^1 \otimes \varphi^2\cdot n_0 \cdot \Phi^2 \otimes \varphi^3 \cdot n_1\cdot \Phi^3 
\end{gathered}
\end{equation}
for all $m\in M, n\in N$, $x,y\in A$, while those on $\lmod{A}\otimes \rmod{A}\otimes \quasihopfmod{N}$ are given by
\begin{equation}\label{eq:trimodA}
\begin{gathered}
x\cdot (a\otimes b\otimes n)\cdot y = x_1a \otimes by_1 \otimes x_2\cdot n \cdot y_2, \\
\delta(a\otimes b\otimes n) = \varphi^1 a \otimes  b\Phi^1 \otimes \varphi^2\cdot n_0 \cdot \Phi^2 \otimes \varphi^3 \cdot n_1\cdot \Phi^3
\end{gathered} 
\end{equation}
for all $a,b,x,y\in A$ and $n\in N$. If $f$ is a morphism in $\Hom{A}{}{A}{A}{\bimod{M}\otimes \quasihopfmod{N}}{\quasihopfmod{P}}$, then
\begin{equation} \label{eq:tildelin}
\begin{gathered}
f\left(x_1\cdot m \cdot y_1\otimes x_2\cdot n \cdot y_2\right) \stackrel{\eqref{eq:trimodM}}{=} f\left(x\cdot \left( m \otimes n\right)\cdot y\right) = x\cdot f(m \otimes n)\cdot y \\
\delta_P\left(f\left(m\otimes n \right)\right) \stackrel{\eqref{eq:trimodM}}{=} f\left(\varphi^1\cdot m \cdot \Phi^1 \otimes \varphi^2\cdot n_0 \cdot \Phi^2\right)\otimes \varphi^3 \cdot n_1\cdot \Phi^3 
\end{gathered}
\end{equation}
Write $\tilde{f}$ for its image in $\Hom{A}{}{A}{}{\bimod{M}}{\Hom{A}{}{A}{A}{\lmod{A}\otimes \rmod{A}\otimes \quasihopfmod{N}}{\quasihopfmod{P}}}$ via \eqref{eq:superadj}. Since $f$ is bilinear we have
\begin{gather*}
\tilde{f}(m)\left(x\cdot (a\otimes b\otimes n)\cdot y\right) \stackrel{\eqref{eq:trimodA}}{=} \tilde{f}(m)(x_1a\otimes by_1\otimes x_2\cdot n\cdot y_2) \stackrel{\eqref{eq:superadj}}{=} f\left(x_1a\cdot m \cdot by_1\otimes x_2\cdot n \cdot y_2\right) \\
\stackrel{\eqref{eq:tildelin}}{=} x\cdot f(a\cdot m\cdot  b \otimes n)\cdot y \stackrel{\eqref{eq:superadj}}{=} x\cdot \tilde{f}(m)\left(a\otimes b\otimes n\right)\cdot y,
\end{gather*}
so that $\tilde{f}(m)$ is bilinear, and since $f$ is right colinear we have
\begin{align*}
\delta_P\left(\tilde{f}(m)\left(a\otimes b\otimes n\right)\right) & \stackrel{\eqref{eq:superadj}}{=} \delta_P\left(f\left(a\cdot m\cdot b\otimes n\right)\right) \\
 & \stackrel{\eqref{eq:tildelin}}{=} f\left(\varphi^1a\cdot m \cdot b\Phi^1 \otimes \varphi^2\cdot n_0 \cdot \Phi^2\right)\otimes \varphi^3 \cdot n_1\cdot \Phi^3 \\
 & \stackrel{\eqref{eq:superadj}}{=} \tilde{f}(m)\left(\varphi^1a \otimes b\Phi^1 \otimes \varphi^2\cdot n_0 \cdot \Phi^2\right)\otimes \varphi^3 \cdot n_1\cdot \Phi^3,
\end{align*}
so that $\tilde{f}(m)$ is colinear as well, for every $m\in M$. Moreover,
\begin{gather*}
\tilde{f}(x\cdot m\cdot y)(a\otimes b\otimes n) \stackrel{\eqref{eq:superadj}}{=} f\left(ax\cdot m\cdot yb \otimes n\right) \stackrel{\eqref{eq:superadj}}{=} \tilde{f}(m)\left(ax \otimes yb \otimes n\right) \\
\stackrel{\eqref{eq:2triang}}{=} \left(x\triangleright \tilde{f}(m)\triangleleft y \right)\left(a \otimes b \otimes n\right)
\end{gather*}
for all $a,b,x,y\in A$, $m\in M$, $n\in N$, so that $\tilde{f}$ is $A$-bilinear. The other way around, if $g\in \Hom{A}{}{A}{}{\bimod{M}}{\Hom{A}{}{A}{A}{\lmod{A}\otimes \rmod{A}\otimes \quasihopfmod{N}}{\quasihopfmod{P}}}$, then
\begin{equation}\label{eq:hatlin}
\begin{gathered}
g(x\cdot m\cdot y)(a\otimes b\otimes n) = \left(x\triangleright g \triangleleft y \right)\left(a \otimes b \otimes n\right) \stackrel{\eqref{eq:2triang}}{=} g(m)(ax\otimes yb\otimes n), \\
g(m)(x_1a\otimes by_1\otimes x_2\cdot n\cdot y_2) \stackrel{\eqref{eq:trimodA}}{=} g(m)\left(x\cdot (a\otimes b\otimes n)\cdot y\right) = x\cdot g(m)\left(a\otimes b\otimes n\right)\cdot y, \\
\delta_P\left(g(m)\left(a\otimes b\otimes n\right)\right) \stackrel{\eqref{eq:trimodA}}{=} g(m)\left(\varphi^1a \otimes b\Phi^1 \otimes \varphi^2\cdot n_0 \cdot \Phi^2\right)\otimes \varphi^3 \cdot n_1\cdot \Phi^3,
\end{gathered}
\end{equation}
for all $a,b,x,y\in A$, $m\in M$, $n\in N$. Write $\hat{g}$ for its image in $\Hom{A}{}{A}{A}{\bimod{M}\otimes \quasihopfmod{N}}{\quasihopfmod{P}}$ via \eqref{eq:superadj}. Then
\begin{gather*}
\hat{g}(x\cdot(m\otimes n)\cdot y) \stackrel{\eqref{eq:trimodM}}{=} \hat{g}(x_1\cdot m\cdot y_1\otimes x_2\cdot n\cdot y_2) \stackrel{\eqref{eq:superadj}}{=} g(m)(x_1\otimes y_1\otimes x_2\cdot n\cdot y_2) \\
\stackrel{\eqref{eq:hatlin}}{=} x\cdot g(m)\left(1\otimes 1\otimes n\right)\cdot y \stackrel{\eqref{eq:superadj}}{=} x\cdot \hat{g}\left(m\otimes n\right)\cdot y
\end{gather*}
and 
\begin{align*}
\delta_P\left(\hat{g}(m\otimes n)\right) & \stackrel{\eqref{eq:superadj}}{=} \delta_P\left(g(m)(1\otimes 1\otimes n)\right) \\
 & \stackrel{\eqref{eq:hatlin}}{=} g(m)\left(\varphi^1 \otimes \Phi^1 \otimes \varphi^2\cdot n_0 \cdot \Phi^2\right)\otimes \varphi^3 \cdot n_1\cdot \Phi^3 \\
 & \stackrel{\eqref{eq:superadj}}{=} \hat{g}\left(\varphi^1 \cdot m \cdot \Phi^1 \otimes \varphi^2\cdot n_0 \cdot \Phi^2\right)\otimes \varphi^3 \cdot n_1\cdot \Phi^3.
\end{align*}
Summing up, \eqref{eq:superadj} is well-defined and an easy check shows that the two assignments are each other inverses.
\end{proof}

\begin{remark}
The existence of a natural isomorphism $\Xi$ as in \eqref{eq:Xi} allows us to recover many of the expected properties. Set $A_2\coloneqq \lmod{A}\otimes \rmod{A}\otimes \quasihopfmod{A}$ as in the proof of Lemma \ref{lemma:sigmagen}, for the sake of brevity. First of all, since $\Xi$ is natural,
\[
\Xi_M(f) = \left(\Xi_M \circ \Hom{A}{}{A}{A}{A_2}{f}\right)(\id_{A_2}) = U(f)\left(\Xi_{A_2}\left(\id_{A_2}\right)\right).
\]
If we introduce $z^{(1)}\otimes z^{(2)} \otimes z^{(3)} \coloneqq \Xi_{A_2}\left(\id_{A_2}\right) \in A_2$, then
\[
\Xi_M(f) = f\left(z^{(1)}\otimes z^{(2)} \otimes z^{(3)}\right),
\]
for all $M$ in $\quasihopf{A}$ and every $f\in \Hom{A}{}{A}{A}{A_2}{M}$, and since it is $A$-bilinear we also have that
\begin{equation}\label{eq:Xif}
\begin{gathered}
f\left(z^{(1)}a\otimes bz^{(2)} \otimes z^{(3)}\right) \stackrel{\eqref{eq:2triang}}{=} (a\triangleright f\triangleleft b)\left(z^{(1)}\otimes z^{(2)} \otimes z^{(3)}\right) = \Xi_M(a\triangleright f\triangleleft b) \\
= a\cdot \Xi_M(f)\cdot b \stackrel{\eqref{eq:tildelin}}{=} f\left(a_1z^{(1)}\otimes z^{(2)}b_1 \otimes a_2z^{(3)}b_2\right).
\end{gathered}
\end{equation}
In particular, for $M=A_2$ and $f=\id_{A_2}$,
\[
z^{(1)}a\otimes bz^{(2)} \otimes z^{(3)} \stackrel{\eqref{eq:Xif}}{=} a_1z^{(1)}\otimes z^{(2)}b_1 \otimes a_2z^{(3)}b_2
\]
for all $a,b\in A$. In addition, for $M=A$ and $f=\varepsilon\otimes \varepsilon\otimes A$, 
\begin{gather*}
\varepsilon(a)\varepsilon\left(z^{(1)}\right)\varepsilon\left(z^{(2)}\right)z^{(3)} \varepsilon(b) = \varepsilon\left(z^{(1)}a\right)\varepsilon\left(bz^{(2)}\right)z^{(3)} \stackrel{\eqref{eq:Xif}}{=} \varepsilon\left(a_1z^{(1)}\right)\varepsilon\left(z^{(2)}b_1\right)a_2z^{(3)}b_2 \\
= a\varepsilon\left(z^{(1)}\right)\varepsilon\left(z^{(2)}\right)z^{(3)}b
\end{gather*}
shows that $t\coloneqq \varepsilon\left(z^{(1)}\right)\varepsilon\left(z^{(2)}\right)z^{(3)} = \Xi_A(\varepsilon\otimes \varepsilon\otimes A)$ is a non-zero two-sided integral in $A$. Set also $\Lambda\coloneqq \Xi_A^{-1}(1)$. Then, by applying $\varepsilon\otimes A$ to both sides of
\[
\Delta\left(\Lambda(a \otimes b\otimes c)\right) \stackrel{\eqref{eq:hatlin}}{=} \Lambda\left(\varphi^1a \otimes b\Phi^1\otimes \varphi^2c_1\Phi^2\right) \otimes \varphi^3c_2\Phi^3,
\]
we get that
\begin{equation}\label{eq:Lambda}
\Lambda(a \otimes b\otimes c) = \varepsilon\left(\Lambda\left(\varphi^1a \otimes b\Phi^1\otimes \varphi^2c_1\Phi^2\right)\right)\varphi^3c_2\Phi^3
\end{equation}
and hence
\begin{gather*}
a = \Xi_M\left(\Xi_M^{-1}(a)\right) = \Xi_M\left(a\triangleright\Lambda\right) \stackrel{\eqref{eq:Xif}}{=} \Lambda\left(z^{(1)}a\otimes z^{(2)} \otimes z^{(3)}\right) \\
\stackrel{\eqref{eq:Lambda}}{=} \varepsilon\left(\Lambda\left(\varphi^1z^{(1)}a \otimes z^{(2)}\Phi^1\otimes \varphi^2z^{(3)}_1\Phi^2\right)\right)\varphi^3z^{(3)}_2\Phi^3
\end{gather*}
for all $a\in A$, which entails that $A$ is finite-dimensional. Now, observe that
\begin{equation}\label{eq:Lambda2side}
x\triangleright \Lambda = x\triangleright \Xi_A^{-1}(1) = \Xi_A^{-1}(x) = \Xi_A^{-1}(1)\triangleleft x = \Lambda\triangleleft x
\end{equation}
for all $x\in A$, \ie
\[
\Lambda(ax\otimes b\otimes c) = (x\triangleright \Lambda) (a\otimes b\otimes c) = (\Lambda \triangleleft x)(a\otimes b\otimes c) = \Lambda(a\otimes xb\otimes c)
\]
for all $a,b,c,x\in A$, which implies that $\Lambda$ induces a well-defined 
\begin{equation}\label{eq:lambda}
\lambda : A\otimes A \to A, \qquad a\otimes b \mapsto \Lambda(a\otimes 1\otimes b).
\end{equation}
It satisfies
\begin{align*}
\lambda(x_1ay_1\otimes x_2by_2) & \stackrel{\eqref{eq:lambda}}{=} \Lambda(x_1ay_1\otimes 1 \otimes x_2by_2) \stackrel{\eqref{eq:Lambda2side}}{=} \Lambda(x_1a\otimes y_1 \otimes x_2by_2) \\
 & \stackrel{\eqref{eq:hatlin}}{=} x\Lambda(a\otimes 1\otimes b)y \stackrel{\eqref{eq:lambda}}{=} x\lambda (a\otimes b)y, \\
\Delta\left(\lambda\left(a\otimes b \right)\right) & \stackrel{\eqref{eq:lambda}}{=} \Delta\left(\Lambda(a\otimes 1\otimes b)\right) \stackrel{\eqref{eq:hatlin}}{=}\Lambda\left(\varphi^1a \otimes \Phi^1\otimes \varphi^2b_1\Phi^2\right) \otimes \varphi^3b_2\Phi^3 \\
 & \stackrel{\eqref{eq:lambda}}{=} \lambda\left(\varphi^1a\Phi^1\otimes \varphi^2b_1\Phi^2\right) \otimes \varphi^3b_2\Phi^3,
\end{align*}
that is to say, $\lambda \in\Hom{A}{}{A}{A}{\bimod{A}\otimes \quasihopfmod{A}}{\quasihopfmod{A}}$.
\end{remark}

\begin{personal}
See also the computations from 27/28 February 2019 and 12/13 February 2020.
\end{personal}
\end{invisible}


\section{Preantipodes and Hopf monads}\label{sec:HopfMonads}

We conclude this paper with one last condition equivalent to the existence of a preantipode for a quasi-bialgebra. It showed up while addressing the question in \S\ref{sec:HopfBimod}, but it is independent from the results therein and hence we dedicate to it this small section.


Recall from \cite[\S2.7]{BruguieresLackVirelizier} that a \emph{Hopf monad} on a monoidal category $(\cM,\otimes,\I)$ is a monad $(T,\mu,\nu)$ on $\cM$ such that the functor $T$ is a colax monoidal functor with $\upphi_0:T(\I)\to\I$, $\upphi_{X,Y}:T\left(X\otimes Y\right)\to T(X)\otimes T(Y)$, the natural transformations $\mu:T^2\to T,\nu:T\to \id_{\cM}$ are morphisms of colax monoidal functors and the \emph{fusion operators}
\begin{gather*}
H_{X,Y}^l \coloneqq  \left(
\xymatrix@C=50pt{
T\left(X\otimes T(Y)\right) \ar[r]^-{\upphi_{X,T(Y)}} & T(X)\otimes T^2(Y) \ar[r]^-{T(X)\otimes \mu_Y} & T(X)\otimes T(Y)
} 
\right), \\
H_{X,Y}^r \coloneqq  \left(
\xymatrix@C=50pt{
T\left(T(X)\otimes Y\right) \ar[r]^-{\upphi_{T(X),Y}} & T^2(X)\otimes T(Y) \ar[r]^-{\mu_X\otimes T(Y)} & T(X)\otimes T(Y)
} 
\right)
\end{gather*}
are natural isomorphisms in $X,Y\in\cM$.

Similarly, consider a colax-colax adjunction $\xymatrix@C=15pt{\cL:\cM \ar@<+0.3ex>[r] & \cM' :\cR\ar@<+0.3ex>[l]}$ between monoidal categories $(\cM,\otimes,\I),(\cM',\otimes',\I')$, with colax monoidal structures $(\cL,\cpsi_0,\cpsi)$ and $(\cR,\cvarphi_0,\cvarphi)$. In \cite[\S2.8]{BruguieresLackVirelizier}, the pair $(\cL,\cR)$ is called a \emph{Hopf adjunction} if the \emph{Hopf operators}
\begin{gather}\label{eq:Hopfop}
\begin{gathered}
\bH^l_{X,X'} \coloneqq  \left( 
\xymatrix@C=50pt{
\cL\left(X\otimes \cR(X')\right) \ar[r]^-{\cpsi_{X,\cR(X')}} & \cL(X)\otimes' \cL\cR(X') \ar[r]^-{\cL(X)\otimes'\epsilon_{X'}} & \cL(X)\otimes' X'
} \right), \\
\bH^r_{X',X} \coloneqq  \left(
\xymatrix@C=50pt{
\cL\left(\cR(X')\otimes X\right) \ar[r]^-{\cpsi_{\cR(X'),X}} & \cL\cR(X')\otimes'\cL(X) \ar[r]^-{\epsilon_{X'} \otimes' \cL(X)} & X'\otimes' \cL(X)
}
\right),
\end{gathered}
\end{gather}
are natural isomorphisms in $X\in\cM,X'\in\cM'$.

Let $A$ be a quasi-bialgebra over a commutative ring $\K$. In view of \cite[Proposition 3.84]{Aguiar} and the fact that $-\otimes A:\Lmod{A}\to\quasihopf{A}$ is strong monoidal with $\xi_{V,W}:(V\otimes A)\tensor{A}(W\otimes A) \to (V\otimes W)\otimes A$ as in \eqref{eq:tensAstrmon}, the functor $\overline{(-)}$ enjoys a colax monoidal structure (unique such that the adjunction $\left(\overline{(-)},-\otimes A\right)$ is colax-lax) where $\epsilon_{\K}$ provides the (iso)morphism $\cl{A}\cong\K$ connecting the unit objects and
\begin{equation}
\cpsi_{M,N}:\overline{M\tensor{A}N} \to \overline{M}\otimes \overline{N}, \quad \overline{m\tensor{A}n}\mapsto \overline{m_0}\otimes \overline{m_1n} \label{eq:natpsi}
\end{equation}
provides the natural transformation connecting the tensor products, where $M,N\in\quasihopf{A}$.\begin{invisible}
On the other hand, $\Hom{A}{}{A}{A}{A\otimes A}{-}$ enjoy a lax monoidal structure. Namely, $\gamma_{\K}$ provides the (iso)morphisms $\K\cong\Hom{A}{}{A}{A}{A\otimes A}{A}$ connecting the unit objects while
\begin{gather*}
\Hom{A}{}{A}{A}{A\otimes A}{M}\otimes \Hom{A}{}{A}{A}{A\otimes A}{N} \to \Hom{A}{}{A}{A}{A\otimes A}{M\tensor{A}N}, \\
f\otimes g \mapsto \Big[a\otimes b\mapsto f\left(\Phi^1a_1\otimes 1\right)\tensor{A}g\left(\Phi^2a_2\otimes \Phi^3b\right)\Big]
\end{gather*}
provides the natural transformations connecting the tensor products.
\end{invisible} This, in particular, makes of $\left(\cl{(-)},-\otimes A\right)$ a colax-colax adjunction (in light of \cite[Proposition 3.93]{Aguiar}, for example). 

Consider the monad $T=\cl{(-)}\otimes A$ on $\quasihopf{A}$ associated to the adjunction $\left(\cl{(-)},-\otimes A\right)$. The natural transformations $\mu$ and $\nu$ are provided by
\begin{gather}
\mu_M:\cl{\cl{M}\otimes A}\otimes A \stackrel{\cong}{\longrightarrow} \cl{M}\otimes A; \quad \cl{\cl{m}\otimes a}\otimes b\mapsto \cl{m}\otimes \varepsilon(a)b \label{eq:mu} \\
\text{and} \qquad \nu_M: M \to \cl{M}\otimes A; \quad m \mapsto \cl{m_0}\otimes m_1, \notag
\end{gather}
where $\mu$ is invertible because the counit $\epsilon$ from \eqref{eq:unitscounitsquasi} is so. It is an \emph{opmonoidal monad} by \cite[\S2.5]{BruguieresLackVirelizier} with
\begin{gather}
\upphi_{M,N}\coloneqq  \left(
\begin{gathered}
\xymatrix @C=15pt @R=0pt {
\cl{M\tensor{A}N}\otimes A \ar[rr]^-{\cpsi_{M,N}\otimes A} && \cl{M}\otimes \cl{N}\otimes A \ar[r]^-{\xi_{\cl{M},\cl{N}}^{-1}} & \left(\cl{M}\otimes A\right) \tensor{A} \left(\cl{N}\otimes A\right) \\
\cl{m\tensor{A}n}\otimes a \ar@{|->}[rrr] & & & \left(\cl{\Phi^1\cdot m_0}\otimes 1\right)\tensor{A}\left(\cl{\Phi^2m_1\cdot n}\otimes \Phi^3a\right)
}
\end{gathered}
\right) \notag \\
\text{and} \qquad \upphi_0\coloneqq \left( 
\begin{gathered}
\xymatrix @C=15pt @R=0pt{
\cl{A}\otimes A \ar[rr]^-{\epsilon_\K\otimes A} && \K\otimes A \ar[r]^-{\cong} & A \\
\cl{a}\otimes b \ar@{|->}[rrr] & & & \varepsilon(a)b
}
\end{gathered}
\right). \label{eq:phi0phi}
\end{gather}

\begin{remark}
Opmonoidal monads are monads and colax monoidal functors at the same time such that the multiplication and unit of the monad are morphisms of colax monoidal functors. They have been called \emph{Hopf monads} in \cite[Definition 1.1]{Moerdijk} and \emph{bimonads} in \cite[\S2.5]{BruguieresLackVirelizier}, \cite[\S2.3]{BruguieresVirelizier}, but we decided to adhere to the terminology introduced by \cite[page 472]{McCrudden} because it is nowadays the most widely used in the subject (see, for example, \cite[Chapter 3]{Bohm}). In particular, a Hopf monad here is an opmonoidal monad whose fusion operators are natural isomorphisms.
\end{remark}

The following is the main result of the present section.

\begin{theorem}\label{thm:main2}
For a quasi-bialgebra $A$ the following are equivalent
\begin{enumerate}[label=(\alph*), ref=\emph{(\alph*)}, leftmargin=1cm, labelsep=0.3cm]
\item\label{item:mainthm1} $A$ admits a preantipode;
\item\label{item:mainthm2} the natural transformation $\cpsi$ of equation \eqref{eq:natpsi} is a natural isomorphism;
\item\label{item:mainthm3} the component $\cpsi_{A\widehat{\otimes}A,A\otimes A}$ of $\cpsi$ is invertible;
\item\label{item:mainthm4} $\left(\cl{(-)},-\otimes A\right)$ is a lax-lax adjunction;
\item\label{item:mainthm5} $\left(\cl{(-)},-\otimes A\right)$ is a Hopf adjunction;
\item\label{item:mainthm6} $T=\cl{(-)}\otimes A$ is a Hopf monad on $\quasihopf{A}$.
\end{enumerate}
\end{theorem}

The proof of Theorem \ref{thm:main2} is postponed to \S\ref{ssec:proof}. We decided to split it in some smaller intermediate results for the sake of clearness.

\begin{corollary}\label{cor:main2Hopf}
For a bialgebra $B$ the following are equivalent
\begin{enumerate}[label=(\alph*), ref=\emph{(\alph*)}, leftmargin=1cm, labelsep=0.3cm]
\item $B$ is a Hopf algebra;
\item the natural transformation $\cpsi$ of equation \eqref{eq:natpsi} is an isomorphism;
\item the component $\cpsi_{B\widehat{\otimes}B,B\otimes B}$ of $\cpsi$ is invertible;
\item $\left(\cl{(-)},-\otimes B\right)$ is a lax-lax adjunction;
\item $\left(\cl{(-)},-\otimes B\right)$ is a Hopf adjunction;
\item $T=\cl{(-)}\otimes B$ is a Hopf monad on $\quasihopf{B}$.
\end{enumerate}
\end{corollary}

\begin{remark}\label{rem:HopfFrob}
Let $A$ be a quasi-bialgebra. Observe that we implicitly proved the following noteworthy fact: the monad $T=\cl{(-)}\otimes A$ is a Frobenius monad if and only if it is a Hopf monad, if and only if it is naturally isomorphic to the identity functor. In fact, on the one hand $T$ is a Frobenius monad if and only if $-\otimes A$ is a Frobenius functor (since we know from Remark \ref{rem:fullyfaith} that $-\otimes A$ is fully faithful, it is monadic. Therefore the claim can be easily deduced from \cite[Theorem 1.6]{Street}. For the details, see \cite[Proposition 1.5]{Saracco-Frobenius}). On the other hand, by Theorem \ref{thm:mainThm} the latter is equivalent to the existence of a preantipode for $A$ and this, in turn, is equivalent to $T$ being Hopf by Theorem \ref{thm:main2}.
\end{remark}


\subsection{Proof of Theorem \ref{thm:main2}}\label{ssec:proof}

In this subsection, we will often make use of the following isomorphism of left $A$-modules
\begin{equation}\label{eq:phi}
\chi_{M,N}\coloneqq \Big(\cl{M\tensor{A}N} \cong (M\tensor{A}N)\tensor{A}\K \cong M\tensor{A}(N\tensor{A}\K) \cong M\tensor{A}\cl{N}\Big),
\end{equation}
natural in $M,N\in\quasihopf{A}$, which closely resembles the one we used to prove Lemma \ref{lemma:cltensor}.
\begin{invisible}
The first and the third are a consequence of Snake Lemma: namely, $M/MI\cong M\tensor{A}A/I$ is an isomorphism of left $A'$-modules for all algebras $A,A'$, ideals $I\subseteq A$ and $(A',A)$-bimodules $M$.
\end{invisible}

\begin{lemma}\label{lemma:monAdj}
Let $\xymatrix@C=15pt{\cL:\cM \ar@<+0.3ex>[r] & \cM' :\cR\ar@<+0.3ex>[l]}$ be a colax-colax adjunction between monoidal categories $(\cM,\otimes,\I),(\cM',\otimes',\I')$, with colax monoidal structures $(\cL,\cpsi_0,\cpsi)$ and $(\cR,\cvarphi_0,\cvarphi)$. Then $\cpsi_0$ and $\cpsi$ are natural isomorphisms if and only if $(\cL,\cR)$ is a lax-lax adjunction.
\end{lemma}

\begin{proof}
The proof is already contained in \cite[Propositions 3.93 and 3.96]{Aguiar}. Let us sketch it anyway, for the sake of the reader. Since $\cR$ is right adjoint to an colax monoidal functor, in light of \cite[Proposition 3.84]{Aguiar} it naturally inherits a unique lax monoidal structure such that the pair $(\cL,\cR)$ is a colax-lax adjunction. Moreover, by the (dual of the) proof of \cite[Proposition 3.96]{Aguiar}, this unique lax monoidal structure is provided by the inverses of $\cvarphi_0$ and $\cvarphi$, thus making of $\cR$ a strong monoidal functor. 
Now, if $\cpsi_0$ and $\cpsi$ are natural isomorphisms, then $\cL$ is a strong monoidal functor. By the direct implication of \cite[Proposition 3.93 (1)]{Aguiar}, $(\cL,\cR)$ is a lax-lax adjunction. Conversely, assume that $(\cL,\cR)$ is a lax-lax adjunction where the lax monoidal structure on $\cL$ is denoted by $(\cL,\upgamma_0,\upgamma)$. As left adjoint of a lax monoidal functor, $\cL$ inherits a unique colax monoidal structure such that $(\cL,\cR)$ is a colax-lax adjunction, by \cite[Proposition 3.84]{Aguiar} again, and this has to be provided by the inverses of $\upgamma_0$ and $\upgamma$. However, $\cL$ already has a colax monoidal structure such that $(\cL,\cR)$ is a colax-lax adjunction: $(\cL,\cpsi_0,\cpsi)$. Therefore, $\upgamma_0^{-1}=\cpsi_0$ and $\upgamma^{-1}=\cpsi$.
\end{proof}

Since in the context of Theorem \ref{thm:main2} we have that $\upphi_0=\epsilon_\K$ is always invertible, the equivalence between \ref{item:mainthm2} and \ref{item:mainthm4} follows from Lemma \ref{lemma:monAdj}: $\left(\cl{(-)},-\otimes A\right)$ is a lax-lax adjunction if and only if $\upphi = \xi^{-1}\circ\cpsi$ is a natural isomorphism, if and only if $\cpsi$ is.

\begin{personal}
\begin{proposition}\label{prop:catHopfmonad}
Let $\xymatrix@C=15pt{\cL:\cM \ar@<+0.3ex>[r] & \cM' :\cR\ar@<+0.3ex>[l]}$ be an opmonoidal adjunction between monoidal categories $(\cM,\otimes,\I),(\cM',\otimes',\I')$, with unit $\eta$, counit $\epsilon$ and opmonoidal structures $(\cL,\cpsi_0,\cpsi)$ and $(\cR,\cvarphi_0,\cvarphi)$. Assume in addition that $\cR$ is fully faithful (\ie $\epsilon$ is a natural isomorphism). The following assertions are equivalent.
\begin{enumerate}[label=(\arabic*), ref=\emph{(\arabic*)}, leftmargin=1cm, labelsep=0.3cm]
\item\label{item:catHopfmonad1} $\cpsi_0$ and $\cpsi$ are natural isomorphisms;
\item\label{item:catHopfmonad2} $(\cL,\cR)$ is a monoidal adjunction;
\item\label{item:catHopfmonad3} $(\cL,\cR)$ is a Hopf adjunction.
\end{enumerate}
\end{proposition}

\begin{proof}
The equivalence between \ref{item:catHopfmonad1} and \ref{item:catHopfmonad2} is essentially \cite[Proposition 3.93]{Aguiar}. Let us sketch it anyway, for the sake of the reader. Since $\cR$ is right adjoint to an opmonoidal functor, it naturally inherits a unique monoidal structure such that the pair $(\cL,\cR)$ is a colax-lax adjunction in the sense of \cite{Aguiar}\footnote{We may have resoundingly called such a pair an \emph{opmonoidal-monoidal adjunction}, but in this case we prefer to use the original terminology.}. Moreover, by the (dual of the) proof of \cite[Proposition 3.96]{Aguiar}, this unique monoidal structure is provided by the inverses of $\cvarphi_0$ and $\cvarphi$, thus making of $\cR$ a strong monoidal functor in fact. 
Now, if $\cpsi_0$ and $\cpsi$ are natural isomorphisms, then $\cL$ is a strong monoidal functor. By the direct implication of \cite[Proposition 3.93 (1)]{Aguiar}, $(\cL,\cR)$ is a monoidal adjunction. Conversely, assume that $(\cL,\cR)$ is a monoidal adjunction where the monoidal structure on $\cL$ is denoted by $(\cL,\gamma_0,\gamma)$. As left adjoint of a monoidal functor, $\cL$ inherits a unique opmonoidal structure such that $(\cL,\cR)$ is a colax-lax adjunction and this has to be provided by the inverses of $\gamma_0$ and $\gamma$. However, $\cL$ already has an opmonoidal structure such that $(\cL,\cR)$ is a colax-lax adjunction: $(\cL,\cpsi_0,\cpsi)$. Therefore, $\gamma_0^{-1}=\cpsi_0$, $\gamma^{-1}=\cpsi$ and \ref{item:catHopfmonad1} holds.

Concerning the implication from \ref{item:catHopfmonad1} to \ref{item:catHopfmonad3}, this follows by looking at the expression \eqref{eq:Hopfop} of the Hopf operators and by recalling that, under the stated hypotheses, $\epsilon$ is a natural isomorphism.

\textbf{However, it doesn't seem to be true at this level of generality. The magic trick in the particular case of interest is the fact that $\cL(X\otimes \eta_Y)$ is an iso since $\chi$ is a natural isomorphism. It seems we don't have anything similar in general case, so that from the invertibility of $\cpsi_{X,\cR(X')}$ is not easy to deduce the invertibility of any $\cpsi_{X,Y}$.}
\end{proof}
\end{personal}

\begin{proposition}\label{prop:Hopfmonad}
The following assertions are equivalent for a quasi-bialgebra $A$.
\begin{enumerate}[label=(\arabic*), ref=\emph{(\arabic*)}, leftmargin=1cm, labelsep=0.3cm]
\item\label{item:corHopfmonad1} The natural transformation $\cpsi$ of equation \eqref{eq:natpsi} is a natural isomorphism;
\item\label{item:corHopfmonad2} $\left(\cl{(-)},-\otimes A\right)$ is a Hopf adjunction between $\quasihopf{A}$ and $\Lmod{A}$;
\item\label{item:corHopfmonad3} $T=\cl{(-)}\otimes A$ is a Hopf monad on $\quasihopf{A}$.
\end{enumerate}
\end{proposition}

\begin{proof}
Since the functor $-\otimes A$ is fully faithful (see Remark \ref{rem:fullyfaith}), the counit $\epsilon$ of the adjunction $\left(\cl{(-)},-\otimes A\right)$ is a natural isomorphism. Thus the implication from \ref{item:corHopfmonad1} to \ref{item:corHopfmonad2} follows by looking at the explicit form \eqref{eq:Hopfop} of the Hopf operators: if $\epsilon$ and $\cpsi$ are natural isomorphisms, then $\bH^l$ and $\bH^r$ are natural isomorphisms as well. The implication from \ref{item:corHopfmonad2} to \ref{item:corHopfmonad3} is \cite[Proposition 2.14]{BruguieresLackVirelizier}. Finally, let us show that \ref{item:corHopfmonad3} implies \ref{item:corHopfmonad1}. By using the explicit form \eqref{eq:phi0phi} of $\upphi$, the left fusion operator can be rewritten as
\[
H_{M,N}^l = \left(\left(\cl{M}\otimes A\right) \tensor{A} \mu_{N}\right)\circ \upphi_{M,\cl{N}\otimes A} \stackrel{\eqref{eq:phi0phi}}{=} \left(\left(\cl{M}\otimes A\right) \tensor{A} \mu_{N}\right)\circ \xi^{-1}_{\cl{M},\cl{\cl{N}\otimes A}} \circ \left(\cpsi_{M,\cl{N}\otimes A}\otimes A\right),
\]
whence if $H^l_{M,N}$ is invertible then $\cpsi_{M,\cl{N}\otimes A}\otimes A$ is invertible as well (because both $\xi$ of \eqref{eq:tensAstrmon} and $\mu$ of \eqref{eq:mu} are). Now, consider the following facts: for every $M,N$ quasi-Hopf $A$-bimodules,
\begin{enumerate}[label=(\roman*), ref={(\roman*)}, leftmargin=1cm, labelsep=0.3cm]
\item\label{item1:psi} we have that $\cpsi_{M,\cl{N}\otimes A} \circ \epsilon_{M\tensor{A}(\cl{N}\otimes A)} = \epsilon_{\cl{M}\otimes \cl{\cl{N}\otimes A}} \circ \cl{\cpsi_{M,\cl{N}\otimes A}\otimes A}$ by naturality of $\epsilon$, so that if $\cpsi_{M,\cl{N}\otimes A}\otimes A$ is an isomorphism then $\cpsi_{M,\cl{N}\otimes A}$ is an isomorphism (because $\epsilon$ is always an isomorphism);
\item\label{item2:psi} in view of the triangular identity $\epsilon_{\cl{N}} \circ \cl{\eta_N} = \id_{\cl{N}}$ for the adjunction $\cl{(-)} \dashv -\otimes A$, $\cl{\eta_N}$ is an isomorphism with inverse $\epsilon_{\cl{N}}$; 
\item\label{item3:psi} we have that $\chi_{M,\cl{N}\otimes A}\circ\cl{M\tensor{A}\eta_N} = \left(M\tensor{A}\cl{\eta_N}\right)\circ \chi_{M,N}$ by naturality of $\chi$ of \eqref{eq:phi}, so that $\cl{M\tensor{A}\eta_N}$ is an isomorphism by \ref{item2:psi} and 
\item\label{item4:psi} $\left(\cl{M}\otimes \cl{\eta_N}\right)\circ\cpsi_{M,N} = \cpsi_{M,\cl{N}\otimes A}\circ \cl{M\tensor{A}\eta_N}$ by naturality of $\cpsi$. Thus, by \ref{item2:psi} and \ref{item3:psi}, if $\cpsi_{M,\cl{N}\otimes A}$ is an isomorphism then $\cpsi_{M,N}$ is an isomorphism.
\end{enumerate}
Therefore, by \ref{item1:psi} and \ref{item4:psi}, if $\cpsi_{M,\cl{N}\otimes A}\otimes A$ is invertible for all $M,N\in\quasihopf{A}$ then $\cpsi_{M,N}$ is invertible as well, concluding the proof.
\end{proof}

\begin{invisible}[Old proof of the equivalence between \ref{item:corHopfmonad1} and \ref{item:corHopfmonad3} in Proposition \ref{prop:Hopfmonad}]

Observe that the fusion operators can be rewritten as
\begin{gather*}
H_{M,N}^l = \left(\left(\cl{M}\otimes A\right) \tensor{A} \mu_{N}\right)\circ \xi^{-1}_{\cl{M},\cl{\cl{N}\otimes A}} \circ \left(\cpsi_{M,\cl{N}\otimes A}\otimes A\right), \\
H_{M,N}^r = \left(\mu_{M} \tensor{A} \left(\cl{N}\otimes A\right)\right)\circ \xi^{-1}_{\cl{\cl{M}\otimes A},\cl{N}} \circ \left(\cpsi_{\cl{M}\otimes A,N}\otimes A\right),
\end{gather*}
whence if $\cpsi$ is a natural isomorphism then both $H^l$ and $H^r$ are (because both $\xi$ and $\mu$ are). Conversely, if $H^l_{M,N}$ is invertible then $\cpsi_{M,\cl{N}\otimes A}\otimes A$ is as well. Now, consider the following facts: for every $M,N\in\quasihopf{A}$ we have that $\cpsi_{M,\cl{N}\otimes A} \circ \epsilon_{M\tensor{A}(\cl{N}\otimes A)} = \epsilon_{\cl{M}\otimes \cl{\cl{N}\otimes A}} \circ \cl{\cpsi_{M,\cl{N}\otimes A}\otimes A}$ by naturality of $\epsilon$, that $\cl{\eta_N}$ is an isomorphism with inverse $\epsilon_{\cl{N}}$, that $\chi_{M,\cl{N}\otimes A}\circ\cl{M\tensor{A}\eta_N} = \left(M\tensor{A}\cl{\eta_N}\right)\circ \chi_{M,N}$ by naturality of $\chi$ and that $\left(\cl{M}\otimes \cl{\eta_N}\right)\circ\cpsi_{M,N} = \cpsi_{M,\cl{N}\otimes A}\circ \cl{M\tensor{A}\eta_N}$ by naturality of $\cpsi$. Therefore, it follows that if $\cpsi_{M,\cl{N}\otimes A}\otimes A$ is invertible for all $M,N\in\quasihopf{A}$ then $\cpsi_{M,N}$ is as well, concluding the proof.
\end{invisible}

As a consequence of Proposition \ref{prop:Hopfmonad}, we have that \ref{item:mainthm2} $\Leftrightarrow$ \ref{item:mainthm5} $\Leftrightarrow$ \ref{item:mainthm6} in Theorem \ref{thm:main2}.

\begin{proposition}\label{prop:PreLaxLax}
The natural transformation $\cpsi$ of equation \eqref{eq:natpsi} is a natural isomorphism if and only if the unit $\eta$ of the adjunction $\left(\cl{(-)},-\otimes A\right)$ is a natural isomorphism. Moreover, the component $\cpsi_{A\widehat{\otimes}A,A\otimes A}$ is invertible if and only if the component $\eta_{A\widehat{\otimes}A}$ is.
\end{proposition}

\begin{proof}
Denote by $\kappa_{V,W}$ the obvious isomorphism $\left(V\otimes A\right)\tensor{A}W \cong V\otimes W$, which is natural in $V,W\in\Lmod{A}$. One can check by a direct computation that
\begin{equation}\label{eq:psikappa}
\cpsi_{M,N} = \kappa_{\cl{M},\cl{N}} \circ \left(\eta_M\tensor{A}\cl{N}\right)\circ \chi_{M,N},
\end{equation}
so that $\cpsi$ is a natural isomorphism if $\eta$ is. Conversely, take $N=\lmod{A}\otimes \quasihopfmod{A}$. 
For every $\cl{m\tensor{A}(a\otimes b)} \in \cl{M\tensor{A}(A\otimes A)}$, compute
\begin{gather*}
\left(\eta_M\circ (M\tensor{A}\epsilon_A) \circ \chi_{M,A\otimes A}\right)\left(\cl{m\tensor{A}(a\otimes b)}\right) \stackrel{\eqref{eq:phi}}{=} \left(\eta_M\circ (M\tensor{A}\epsilon_A)\right)\left(m\tensor{A}\cl{(a\otimes b)}\right) \\
\stackrel{\eqref{eq:unitscounitsquasi}}{=} \eta_M(m\cdot a \varepsilon(b)) \stackrel{\eqref{eq:unitscounitsquasi}}{=} \cl{m_0\cdot a_1}\otimes m_1a_2\varepsilon(b) = \cl{m_0} \otimes m_1a\varepsilon(b) \\
 \stackrel{\eqref{eq:unitscounitsquasi}}{=} \left(\cl{M}\otimes \epsilon_A\right)\left(\cl{m_0}\otimes \cl{m_{1_1}a\otimes m_{1_2}b}\right) \stackrel{\eqref{eq:natpsi}}{=} \left(\left(\cl{M}\otimes \epsilon_A\right) \circ \cpsi_{M,A\otimes A}\right)\left(\cl{m\tensor{A}(a\otimes b)}\right).
\end{gather*}
Therefore $\eta_M\circ (M\tensor{A}\epsilon_A) \circ \chi_{M,A\otimes A} = \left(\cl{M}\otimes \epsilon_A\right) \circ \cpsi_{M,A\otimes A}$ and hence $\eta$ is a natural isomorphism if $\cpsi$ is.\begin{invisible}
In fact, there is a commutative diagram
\[
\xymatrix{
\cl{M\tensor{A}(A\otimes A)} \ar[rrr]^-{\chi_{M,A\otimes A}} \ar[ddd]_-{\cpsi_{M,A\otimes A}} & & & M\tensor{A}\cl{A\otimes A} \ar[ddd]^-{\eta_{M}\tensor{A}\cl{A\otimes A}} \ar[dl]|-{M\tensor{A}\epsilon_A} \\
 & M \ar[d]_-{\eta_M} & M\tensor{A}A \ar[l]_-{\cong} \ar[d]^-{\eta_M\tensor{A}A} & \\
 & \cl{M}\otimes A & (\cl{M}\otimes A)\tensor{A}A \ar[l]^-{\kappa_{\cl{M},A}} & \\
\cl{M}\otimes \cl{A\otimes A} \ar[ur]|-{\cl{M}\otimes \epsilon_A} & & & (\cl{M}\otimes A)\tensor{A}\cl{A\otimes A} \ar[lll]^-{\kappa_{\cl{M},\cl{A\otimes A}}} \ar[ul]|-{(\cl{M}\otimes A)\tensor{A}\epsilon_A}
}
\]
\end{invisible}
In particular, for $M=A\cmptens{}A = \rmod{A}\otimes \quasihopfmod{A}$, $\eta_{A\widehat{\otimes}A}$ is invertible if and only if $\cpsi_{A\widehat{\otimes}A,A\otimes A}$ is.
\end{proof}

In light of \cite[Theorem 4]{Saracco}, $A$ admits a preantipode if and only if $\eta$ is a natural isomorphism (because the counit $\epsilon$ is always a natural isomorphism), if and only if the distinguished component $\eta_{A\widehat{\otimes}A}$ is invertible. Therefore, if follows from Proposition \ref{prop:PreLaxLax} that \ref{item:mainthm1} $\Leftrightarrow$ \ref{item:mainthm2} $\Leftrightarrow$ \ref{item:mainthm3} in Theorem \ref{thm:main2} and this concludes its proof.

\begin{remark}
Concerning the implication from \ref{item:mainthm1} to \ref{item:mainthm2}, it follows from \cite[Equations (16) and (28)]{Saracco} that if $A$ admits a preantipode $S$, then $\eta_M^{-1}(\cl{m}\otimes a) = \Phi^1m_0S(\Phi^2m_1)\Phi^3a$ for all $m\in M, a\in A$. In this case, an explicit inverse for $\cpsi_{M,N}$ is given by
\[
\cl{M}\otimes \cl{N}\to \cl{M\tensor{A}N}: \qquad \cl{m}\otimes \cl{n}\mapsto \cl{\Phi^1m_0S(\Phi^2m_1)\Phi^3\tensor{A}n}.
\]
\end{remark}

\begin{invisible}
For our own sake, the canonical projection $N\to \cl{N}$ is left $A$-linear, whence we may consider $m\tensor{A} n\mapsto m\tensor{A}\cl{n}$ and it satisfies
\begin{gather*}
m_i\tensor{A}n_iz_i\mapsto m_i\tensor{A}\cl{n_iz_i} = m_i\tensor{A}\cl{n_i}\varepsilon(z_i)=0,
\end{gather*}
so that it factors through $\cl{M\tensor{A}N} \to M\tensor{A}\cl{N}$ and the latter one is a left $A$-linear epimorphism (tensoring is exact on the right). To show that it is an isomorphism, we consider the projection $\pi: M\tensor{A}N\to \cl{M\tensor{A}N}$. It satisfies
\[
\pi \circ (M\tensor{A}\iota_{NA^+}) = 0 
\]
and hence, since $M\tensor{A}-$ is right exact, it factors through $M\tensor{A}\cl{N}$ giving the inverse.
\end{invisible}
\begin{personal}
The composite $\chi_{M,N}$ is natural in $M,N$ and hence the following diagram commutes
\[
\xymatrix @C=50pt @R=15pt {
\cl{M\tensor{A}N} \ar[rr]^-{\cl{\eta_M\otimes \eta_N}} \ar[d]_-{\chi_{M,N}} & & \cl{(\cl{M}\otimes A)\tensor{A}(\cl{N}\otimes A)} \ar[d]^-{\chi_{\cl{M}\otimes A,\cl{N}\otimes A}} \\
M\tensor{A}\cl{N} \ar[r]_-{\eta_M\tensor{A} \cl{N}} & (\cl{M}\otimes A)\tensor{A}\cl{N} \ar[r]_-{(M\otimes A)\tensor{A}\cl{\eta_N}} & (\cl{M}\otimes A)\tensor{A} \cl{(\cl{N}\otimes A)}
}
\]
Since $\cl{\eta_N}$ is an isomorphism for every $N$ and $\cpsi_{M,N} = \epsilon_{\cl{M}\otimes \cl{N}} \circ \cl{\xi^{-1}_{M,N}} \circ \cl{\eta_M\tensor{A}\eta_N}$, we conclude that $\eta$ is a natural isomorphism if and only if $\cl{\eta_M\otimes \eta_N}$ is an isomorphism for all $M,N\in\quasihopf{A}$ (for the \emph{if} part, take $N=\lmod{A}\otimes \quasihopfmod{A}$), if and only if $\cpsi$ is a natural isomorphism.
\end{personal}

\begin{invisible}[Old proof of Theorem \ref{thm:main2}]

We already know that \ref{item:mainthm4} $\Leftrightarrow$ \ref{item:mainthm2} $\Leftrightarrow$ \ref{item:mainthm1} by Propositions \ref{prop:Hopfmonad} and \ref{prop:PreLaxLax}, whence we are left to prove that \ref{item:mainthm3} is equivalent to \ref{item:mainthm2}. Since $\cpsi$ is the unique opmonoidal structure on $\cl{(-)}$ that makes of $\left(\cl{(-)},-\otimes A\right)$ a colax-lax adjunction in the sense of \cite{Aguiar}, by (the proof of) \cite[Proposition 3.96]{Aguiar} we can claim that if $\big(\cl{(-)},-\otimes A\big)$ is a monoidal adjunction, then $\cpsi$ is a natural isomorphism, proving that \ref{item:mainthm3} implies \ref{item:mainthm2}. Conversely, assume that $\cpsi$ is a natural isomorphism. This means that $\cl{(-)}$ is a strong monoidal functor. Since $\left(\cl{(-)},-\otimes A\right)$ is still a colax-lax adjunction, \cite[Proposition 3.93 (1)]{Aguiar} entails that $\left(\cl{(-)},-\otimes A\right)$ is a monoidal adjunction.
\end{invisible}

\end{document}